\documentclass[12pt]{article}
\usepackage{amsmath,amscd,amssymb,amsthm,mathrsfs,amsfonts,layout,indentfirst,graphicx,caption,mathabx,  tikz, stmaryrd,appendix}
\usepackage{palatino}  %template
\usepackage{slashed} % Dirac operator
\usepackage{mathrsfs} % Enable using \mathscr

\allowdisplaybreaks  %allow aligns to break between pages
\usepackage{latexsym}
\usepackage{chngcntr}
\usepackage[colorlinks,linkcolor=blue,anchorcolor=blue,citecolor=red,linktocpage]{hyperref}

\usepackage{lipsum}
\let\OLDthebibliography\thebibliography
\renewcommand\thebibliography[1]{
	\OLDthebibliography{#1}
	\setlength{\parskip}{0pt}
	\setlength{\itemsep}{2pt} 
}

\usepackage{fullpage}
\counterwithin{figure}{section}

\theoremstyle{definition}
\newtheorem{df}{Definition}[section]

\newtheorem{rem}[df]{Remark}

\newtheorem{cv}[df]{Convention}

\theoremstyle{plain}
\newtheorem{thm}[df]{Theorem}
\newtheorem{pp}[df]{Proposition}
\newtheorem{co}[df]{Corollary}
\newtheorem{lm}[df]{Lemma}

\newcommand{\tr}{\mathrm{t}} %transpose
\newcommand{\End}{\mathrm{End}} %endomorphism
\newcommand{\id}{\mathrm{id}}
\newcommand{\Hom}{\mathrm{Hom}}
\newcommand{\Conf}{\mathrm{Conf}}
\newcommand{\Res}{\mathrm{Res}}

\newcommand{\ev}{\mathrm{ev}}
\newcommand{\coev}{\mathrm{coev}}

\newcommand{\Rep}{\mathrm{Rep}}
\newcommand{\diag}{\mathrm{diag}}

\newcommand{\uni}{\mathrm{u}}
\newcommand{\di}{\slashed d}

\numberwithin{equation}{section}

\title{Unitarity of the modular tensor categories associated to unitary vertex operator algebras, I}
\author{{\sc Bin Gui}\\
{\small Department of Mathematics, Vanderbilt University}\\
{\small  bin.gui@vanderbilt.edu}
}
\date{}
\begin{document}\sloppy % avoid stretch into margins
	\pagenumbering{arabic}
	%\pagenumbering{gobble}
	%\newpage
	%\setcounter{page}{1}
	\setcounter{section}{-1}
\maketitle

\begin{abstract}
This is the first part in a two-part series of papers constructing a unitary structure for the  modular tensor category (MTC) associated to a unitary rational vertex operator algebra (VOA). Given a rational VOA, we know that its MTC is constructed using the (finite dimensional) vector spaces of intertwining operators of this VOA. Moreover, the tensor-categorical structures can be described by the monodromy behaviors of the intertwining operators. Thus, constructing a unitary structure for the MTC of a unitary rational VOA amounts to defining an inner product on each (finite dimensional) vector space of intertwining operators, and showing that the monodromy matrices of the intertwining operators (e.g. braiding matrices, fusion matrices) are unitary under these inner products.

In this paper, we develop necessary tools and techniques for constructing our unitary structures. This includes giving a systematic treatment of one of the most important functional analytic properties of the intertwining operators: the energy bounds condition.  On the one side, we give some useful criteria for proving the energy bounds condition of intertwining operators. On the other side,  we show that for energy bounded  intertwining operators, one can define the smeared ones, which are (unbounded) closed operator. We prove  that the (well-known) braid relations and adjoint relations for unsmeared intertwining operators have the corresponding smeared versions. We also give criteria on the strong commutativity between smeared intertwining operators and smeared vertex operators localized in disjoint open intervals of $S^1$ (the strong intertwining property).  Besides investigating the energy bounds condition, we also study certain genus 0 geometric properties of intertwining operators. Most importantly, we prove the convergence  of certain mixed products-iterations of intertwining operators. Many useful braid and fusion relations will also be discussed.
\end{abstract}
\tableofcontents

\section{Introduction}

\subsubsection*{Vertex operator algebras: unitarity and reflection positivity, intertwining operators, and modular tensor categories}

Wightman axioms and algebraic quantum field theories (AQFTs) are two major ways to formulate quantum field theories (QFTs) in the rigorous language of mathematics. Roughly speaking, the main difference between these two approaches is that the first one focuses on field operators localized at points, whereas the latter one studies bounded or unbounded but (pre)closed field operators (as well as the von Neumann algebras they generate) localized on open subsets of the space-time. For 2d conformal field theories (CFTs), the AQFT approach   goes under the name ``conformal net". Many fruitful results have been achieved in this functional analytic approach.  We refer the reader to \cite{Kaw15} for a brief survey on this topic.

In many senses, the theory of vertex operator algebras (VOAs) can be regarded as the Wightman axiomatization of CFT. In fact, given a VOA  we have a vertex operator $Y$, which associates to each state vector $v$ and each point $z\in\mathbb C$ a field operator $Y(v,z)$ localized at $z$.  However, one has to be careful when regarding  VOAs as  Wightman CFTs for the following reasons.\\

1. Wightman QFTs are defined on the Minkowski space-time, while VOAs actually correspond to CFTs in the Euclidean picture. It is well known that one can do Wick rotation to pass from Minkowskian QFTs to Euclidean ones.  However, it is not true that any Euclidean QFT can arise from a Minkowskian one. One has to ensure that the Euclidean QFT satisfies \emph{reflection positivity} \cite{OS73}. A VOA satisfying reflection positivity is called \textbf{unitary} \cite{DL}.

2. In most cases, a VOA does not give all field operators of a (closed-string) CFT. In fact, a VOA $V^L$ is  the chiral part of a CFT, consisting  only of meromorphic fields, and the anti-chiral part corresponds to the complex conjugate of another VOA $V^R$. However, in a CFT there are field operators which are locally neither holomorphic nor anti-holomorphic. The typical way of studying these general field operators is through conformal blocks, or equivalently, the \textbf{intertwining operators}. An intertwining operator $\mathcal Y$ of a VOA $V$ is a generalization of the vertex operator $Y$, which intertwines the actions of $V$ on three $V$-modules (the charge space, the source space, and the target space), and which is locally holomorphic but globally multi-valued field.\footnote{See section \ref{lb29} for the precise definition of intertwining operators.} Then a field operator $\Phi(z,\overline z)$ should look like $\Phi(z,\overline z)=\sum_{\alpha,\beta}\mathcal Y_\alpha(w^{\alpha},z)\overline{\mathcal Y_\beta(w^{\beta},z)}$, where each $\mathcal Y_\alpha$ (resp. $\mathcal Y_\beta$) is an intertwining operator of $V^L$ (resp. $V^R$), and $w^\alpha$ (resp. $w^\beta$) is a vector inside the source space of $\mathcal Y_\alpha$ (resp. $\mathcal Y_\beta$).\footnote{cf. \cite{MS88}. Intertwining operators are called chiral vertex operators in that paper.} This means that intertwining operators are indeed the ``chiral halves" of the full field operators. 

Since full field operators  satisfy commutativity (locality), associativity (existence of operator product expansions), and modular invariance,\footnote{cf. \cite{MS88}. In \cite{HK07,HK10} the reader can find the precise statement of these properties in the language of vertex operator algebras.} one may expect that their chiral halves should also satisfy similar properties. But since  intertwining operators are  multi-valued functions, monodromy behaviors will appear when considering these properties. Thus, for intertwining operators, one should expect braiding and fusion instead of commutativity and associativity, and, rather than thinking of the modular invariance of (the trace of) one single intertwining operator, one should consider the modular invariance of the vector space of intertwining operators \cite{MS88}.\footnote{See section \ref{lb77} for the statement of braid and fusion relations. Modular invariance in its most general form can be found in \cite{H MI}.} Hence one will have braid, fusion and modular ($S$ and $T$) matrices, and, written in a coordinate-independent way, one has a \textbf{modular tensor category} (MTC) \cite{MS89, MS90}.\footnote{A mathematically rigorous and complete construction is due to Y.Z.Huang and J.Lepowsky. See \cite{HL13} for a brief review of their theory.}

3. The reason why we have  commutativity, associativity, and modular invariance in CFTs is not quite obvious from the Wightman axioms. These properties have highly geometric nature, and can more easily be seen in the Euclidean picture, where the CFTs are defined, not only on the flat complex plane (or punctured Riemann spheres), but on any compact Riemann surface. Indeed, these properties are among the most important examples of the \emph{sewing property}\footnote{Sewing property says that if a punctured Riemann surfaces $M$ is obtained by attaching another two $M_1$ and $M_2$, then the correlation function on $M$ can always be obtained by taking the composition of two correlation functions defined on $M_1$ and $M_2$ respectively. Note that intertwining operators are nothing but the chiral halves of the correlation functions on the Riemann sphere with three holes.}, which is clear from an (highly geometric) axiomatization of CFT not yet mentioned: G.Segal's definition of CFTs \cite{Seg88}.\\

\subsubsection*{Motivations}

Thus, VOAs are deeply rooted in the geometric nature of CFT, but can be formulated without assuming unitarity (or reflection positivity). On the other hand, conformal nets, the Wightman axiomatization of CFT, are not so geometric but manifestly unitary. The goal of this paper (as well as the forthcoming second part of this series) is to develop a unitary theory for the MTCs of unitary rational VOAs. We explain some motivations behind this theory.

First, we have seen that MTCs are important for the construction of full CFTs. Having constructed  MTCs from rational VOAs, one can use Frobenius algebras over MTCs to  classify  full rational CFTs \cite{Kong06,Kong08} (the word ``rational" means that the sum $\Phi=\sum\mathcal Y_\alpha\overline{\mathcal Y_\beta}$ mentioned earlier is always finite). However, in order to classify \emph{unitary} full CFTs, i.e., full CFTs with \emph{reflection positivity}, one needs the unitarity of these MTCs, and then one studies unitary Frobenius algebras (i.e., Q-systems) over these unitary MTCs. Besides full (closed-string) CFTs, the unitarity of  MTCs is also necessary for studying the unitary extensions of unitary rational VOAs, and unitary open-string CFTs, just as MTCs are important for studying general VOA extensions \cite{HKL15,CKM17} and general open-string CFTs \cite{Kong08}.

The second motivation is to prepare for the investigation of the relations between conformal nets and unitary VOAs. Just like unitary VOAs, conformal nets also describe the chiral parts of  unitary CFTs, and one can construct MTCs from rational conformal nets  \cite{DHR71, FRS89, KLM01}, which is automatically unitary. It is important to know whether the MTCs constructed from conformal nets and from unitary VOAs are equivalent. Clearly, if one can show the equivalence, then the unitarizability of the MTCs of conformal nets will imply that of the MTCs of unitary VOAs. It turns out, however, that in order to prove this equivalence, one has to first equip the MTCs of unitary VOAs with a unitary structure.

\subsubsection*{A glance at the theory}

Now we briefly explain what we shall do in this series of papers in order to find a unitary structure on the MTCs. For simplicity, we assume that $V$ is a unitary ``rational"\footnote{The exact meaning of rationality in this paper will be made clear later. See conditions \eqref{eq302}-\eqref{eq304}.} VOA whose representations are always unitarizable. (For example, $V$ can be a unitary Virasoro VOA (minimal model), or a unitary affine VOA (WZW model).) If $W_i,W_j,W_k$ are unitary representations of $V$, then a type $k\choose i~j$ intertwining operator $\mathcal Y_\alpha$ linearly associates to each $w^{(i)}\in W_i$ a multivalued holomorphic operator-valued function
\begin{align*}
\mathcal Y_\alpha(w^{(i)},z):W_j\rightarrow \widehat W_k,
\end{align*}
where $\widehat W_k$ is the algebraic completion of $W_k$ (see section \ref{lb103}). Moreover, one requires that $\mathcal Y$ ``intertwines" the actions of $V$ on $W_i,W_j,W_k$ (Jacobi identity), and that $\mathcal Y_\alpha$ is conformal covariant (translation property).\footnote{Rigorous definition can be found in definition \ref{lb28}.} The $V$-modules $W_i,W_j,W_k$ are called, respectively, the charge space, the source space, and the target space of $\mathcal Y_\alpha$.  We denote by $\mathcal V{k\choose i~j}$ the vector space of type $k\choose i~j$ intertwining operators. Note that if we set $W_0=V$, then the vertex operator $Y$ is a type $0\choose 0~0$ intertwining operator.

Now, for each equivalence class of irreducible unitary $V$-module, we choose a representing element to form a set $\{W_k:k\in\mathcal E \}$. With abuse of notation, we also let $\mathcal E$ denote this set. For any unitary $V$-modules $W_i,W_j$, their tensor product $W_i\boxtimes W_j$ is a $V$-module defined by
$$W_i\boxtimes W_j=\bigoplus_{k\in\mathcal E}\mathcal V{k\choose i~j}^*\otimes W_k.$$\footnote{This definition is due to Y.Z.Huang and J.Lepowsky, cf. \cite{H 1}.}
By rationality of $V$, $\mathcal E$ is a finite set (i.e., there are only finitely many equivalence classes of irreducible $V$-modules), and $\mathcal V{k\choose i~j}$, as well as its dual space $\mathcal V{k\choose i~j}^*$, is finite-dimensional. Note that although $W_i\boxtimes W_j$ is unitarizable, we don't know how to choose a canonical unitary structure on $W_i\boxtimes W_j$, because we don't know how to choose a meaningful inner product on the vector space $\mathcal V{k\choose i~j}^*$. But this is exactly the goal of our theory. In part II of this series, we will define a sesquilinear form  $\Lambda$ on $\mathcal V{k\choose i~j}^*$ for each $W_i,W_j$ and irreducible $W_k$. After choosing a basis of $\mathcal V{k\choose i~j}$, $\Lambda$ will be defined using certain fusion or braid matrix under this basis. The most difficult part of our theory is to prove that these sesquilinear forms (or equivalently, the corresponding fusion or braid matrices) are positive definite, i.e., they are inner products. Once this is proved, then  it is not hard to show the unitarity of all braid and fusion matrices under any orthonormal basis with respect to this inner product, and hence the unitarity of the MTC.

\subsubsection*{Smeared intertwining operators}

The non-degeneracy of  $\Lambda$ will follow from the rigidity of the MTC. So what we actually need to prove is the positivity of $\Lambda$. Although this problem is purely vertex-operator-algebraic, it seems very difficult to solve it using only VOA methods. We tackle this problem by investigating some analytic and algebraic properties of  the \textbf{smeared intertwining operators} of $V$, so that many results in conformal nets (most importantly, the Haag duality) can be used in our theory. Here,  for any type $k\choose i~j$ intertwining operator $\mathcal Y_\alpha$, $w^{(i)}\in W_i$, $I$ an open interval in $S^1$, and $f\in C^\infty_c(I)$, the smeared intertwining operator is defined to be
$$\mathcal Y_\alpha (w^{(i)},f):=\oint_{S^1} \mathcal Y_\alpha(w^{(i)},z)f(z)\frac{dz}{2i\pi}.$$
This generalizes the smeared vertex operators considered in \cite{CKLW}. Similar to \cite{CKLW}, we require that $\mathcal Y_\alpha(w^{(i)},\cdot)$ satisfies the following \textbf{energy bounds condition}:
there exist $M,r,t\geq0$, such that for any open interval $I\in S^1,f\in C^\infty_c(I),w^{(j)}\in W_j$, 
$$ \lVert \mathcal Y_\alpha (w^{(i)},f)w^{(j)} \lVert\leq M|f|_t\lVert (1+L_0)^rw^{(j)}\lVert,$$
where $|f|_t$ is the $t$-th order Sobolev norm of $f$. Then $\mathcal Y_\alpha (w^{(i)},f)$ will be a (pre)closed unbounded operator mapping $\mathcal H_j\rightarrow \mathcal H_k$.

One of the main purposes in the present paper is to prove the algebraic and analytic properties of smeared intertwining operators that are necessary for showing the positivity of $\Lambda$. First we  discuss  \textbf{braiding of smeared intertwining operators}. As we mentioned above, braiding, fusion, and modular invariance are among the most important geometric properties of intertwining operators. However,  only  braid relation can be translated onto smeared intertwining operators. More specifically, if $I,J$ are disjoint open intervals in $S^1$ with chosen continuous $\arg$ functions, and we have intertwining operators $\mathcal Y_\alpha,\mathcal Y_\beta,\mathcal Y_{\alpha'},\mathcal Y_{\beta'}$ such that the braid relation
$$\mathcal Y_\beta(w^{(j)},\zeta)\mathcal Y_\alpha(w^{(i)},z)=\mathcal Y_{\alpha'}(w^{(i)},z)\mathcal Y_{\beta'}(w^{(j)},\zeta)$$
holds for any vectors $w^{(i)},w^{(j)}$, and any $z\in I,\zeta\in J$, and if these four intertwining operators are energy bounded, then we will show that the corresponding braid relation for smeared intertwining operators
 $$\mathcal Y_\beta(w^{(j)},g)\mathcal Y_\alpha(w^{(i)},f)=\mathcal Y_{\alpha'}(w^{(i)},f)\mathcal Y_{\beta'}(w^{(j)},g)$$
hold  for any vectors $w^{(i)},w^{(j)}$, and any $f\in C^\infty_c(I),g\in C^\infty_c(J)$.
Note  that these two braid relations are understood in different ways. The second one is a completely algebraic relation, where products of smeared intertwining  operators just mean  compositions. However, as compositions of (non-smeared) intertwining operators, the two sides of the first braid relation cannot be defined on  the same region. Braiding of intertwining operators, unlike its smeared version, should be understood in the sense of analytic continuation.

Braid relations tell us what we shall get if we exchange the product of two smeared intertwining operators localized in disjoint open intervals. With the help of \textbf{adjoint relation}, we can obtain the result of exchanging the product of an intertwining operator with the adjoint of another one, say $\mathcal Y_\beta(w^{(j)},g)\mathcal Y_\alpha(w^{(i)},f)^\dagger$, which is also very important in our theory. Given a type $k\choose i~j$ intertwining operator $\mathcal Y_\alpha$, one can define in a canonical way a type $j\choose \overline i~k$ intertwining operator $\mathcal Y_{\alpha^*}$, called the \textbf{adjoint intertwining operator} of $\mathcal Y_\alpha$. (Here $W_{\overline i}$ is the contragredient  module (the dual) of $W_i$.) For any eigenvector $w^{(i)}\in W_i$ of $L_0$ (with eigenvalue $\Delta$) satisfying $L_1w^{(i)}=0$ (i.e., $w^{(i)}$ is a quasi-primary vector), $\mathcal Y_\alpha(w^{(i)},z)^\dagger$ can be related to $\mathcal Y_{\alpha^*}(\overline{w^{(i)}},z)$ by the following very simple relation
$$\mathcal Y_\alpha(w^{(i)},z)^\dagger=e^{-i\pi\Delta}\overline{z^{-2\Delta}}\mathcal Y_{\alpha^*}(\overline{w^{(i)}},\overline {z^{-1}}).$$
We shall prove a similar relation for smeared intertwining operators, so that the result of exchanging $\mathcal Y_\beta(w^{(j)},g)\mathcal Y_\alpha(w^{(i)},f)^\dagger$ will follow from the braiding of $\mathcal Y_\beta(w^{(j)},g)\mathcal Y_{\alpha^*}(\overline{w^{(i)}},f)$

Braiding and adjoint relations are algebraic properties of smeared intertwining operators. To be able to use the powerful machinery of conformal nets, we need an analytic property of smeared intertwining operators: the \textbf{strong intertwining property}. It says that for any disjoint open intervals $I,J\in S^1$, and $f\in C^\infty_c(I),g\in C^\infty_c(J)$, the commuting relation
$$Y_k(v,g) \mathcal Y_\alpha(w^{(i)},f)=\mathcal Y_\alpha(w^{(i)},f)Y_j(v,g)$$
(as a special case of braiding) not only holds when acting on $W_j$, but also holds in a  \textbf{strong} sense, which means that $\mathcal Y_\alpha(w^{(i)},f)$, when extended to an unbounded operator on $\mathcal H_j\oplus\mathcal H_k$ mapping $\mathcal H_k$ to zero, commutes with the von Neumann algebra generated by $Y_j(v,g)\oplus Y_k(v,g)$.\footnote{That the commutativity of two unbounded operators acting on a common invariant core does not imply the strong commutativity of these two operators is well known due to Nelson's counterexample \cite{Nel}.}  The strong intertwining property could be understood as a generalization of the strong locality property (i.e., the strong commutativity of smeared vertex operators) discussed in \cite{CKLW}.\footnote{A natural question is whether one can generalize the strong intertwining property one step further to the strong braiding between smeared intertwining operators. Strong braiding is very important for showing the equivalence between the fusion categories of a unitary VOA and  the corresponding conformal net. But since it will not be used in our present theory, we  leave the discussion of this interesting topic to future work.}

\subsubsection*{Generalized (smeared) intertwining operators}

The above discussion is based on the assumption that the intertwining operators are energy-bounded. However, in practice it might be not easy to show the energy bounds condition for \emph{all} intertwining operators of a given unitary rational VOA.
Let us choose $V$ to be the unitary  level-$l$ affine $\mathfrak{su}_n$ VOA for instance. Then the energy bounds condition for  type $k\choose i~j$ intertwining operators is established only when the charge space $W_i$ is a direct sum of $V$-modules equivalent to $W_0=V$ or $W_\square=L_{\mathfrak {su}_n}(\square,l)$ \cite{Wassermann}. Here $W_\square$ corresponds to the irreducible level $l$ integrable highest weight representation of the affine Lie algebra $\widehat{\mathfrak{su}_n}$ whose  highest weight $\square$ is the one of the vector representation $\mathfrak {su}_n\curvearrowright \mathbb C^n$. So when $W_i$ is a general $V$-module,  it might not be helpful to consider smeared intertwining operators of type $k\choose i~j$.
 
To overcome this difficulty, we consider \textbf{generalized intertwining operators} and their smeared versions. The key observation is that $W_\square$ is a generating object inside the tensor category of $V$. Let us assume, without loss of generality, that $W_i$ is irreducible. Then from the well-known fusion rules of $V$, one can easily find $n=1,2,\dots$ such that $W_i$ is equivalent to a $V$-submodule of $\underbrace{W_\square\boxtimes\cdots\boxtimes W_\square}_n$. It follows that there exist intertwining operators $\mathcal Y_{\sigma_2},\dots,\mathcal Y_{\sigma_n}$ with charge spaces equaling $W_\square$, such that the source space of $\mathcal Y_{\sigma_2}$ is $W_\square$, the target space of $\mathcal Y_{\sigma_n}$ is $W_i$, and for any $3\leq m\leq n$ the source space of $\mathcal Y_{\sigma_m}$ equals the target space of $\mathcal Y_{\sigma_{m-1}}$. (Any sequence of intertwining operators satisfying the last condition is called a \textbf{chain of intertwining operators}.) Now, for any type $k\choose i~j$ intertwining operator $\mathcal Y_\alpha$, we define a \textbf{generalized intertwining operator} $\mathcal Y_{\sigma_n\cdots\sigma_2,\alpha}$ which linearly associates to any $w^{(\square)}_1,\dots,w^{(\square)}_n\in W_\square$   a $\Hom(W_j,\widehat W_k)$-valued multi-valued holomorphic function $\mathcal Y_{\sigma_n\cdots\sigma_2,\alpha}(w^{(\square)}_n,z_n;\dots;w^{(\square)}_1,z_1)$ of the complex variables $z_1,\dots,z_n$ by  setting
\begin{align}
 \mathcal Y_{\sigma_n\cdots\sigma_2,\alpha}(w^{(\square)}_n,z_n;\dots;w^{(\square)}_1,z_1)=\mathcal Y_\alpha\big(\mathcal Y_{\sigma_n}(w^{(\square)}_n,z_n-z_1)\cdots\mathcal Y_{\sigma_2}(w^{(\square)}_2,z_2-z_1)w^{(\square)}_1,z_1 \big).
\end{align}
Then for  any mutually disjoint open intervals $I_1,\dots,I_n\subset S^1$ and $f_1\in C^\infty_c(I_1),\dots,f_n\in C^\infty_c(I_n)$, the corresponding \textbf{smeared generalized intertwining operator} is defined to be
\begin{align*}
&\mathcal Y_{\sigma_n\cdots\sigma_2,\alpha}(w^{(\square)}_n,f_n;\dots;w^{(\square)}_1,f_1)\\
=&\oint_{S^1}\cdots\oint_{S^1} \mathcal Y_{\sigma_n\cdots\sigma_2,\alpha}(w^{(\square)}_n,z_n;\dots;w^{(\square)}_1,z_1)f_n(z_n)\cdots f_1(z_1)\frac {dz_1}{2i\pi}\cdots\frac {dz_n}{2i\pi}.
\end{align*}
Thanks to  fusion relations, there exist a chain of intertwining operators $\mathcal Y_{\alpha_1},\dots,\mathcal Y_{\alpha_n}$ with charge spaces equaling $W_\square$ (hence these intertwining operators are energy-bounded!), such that
\begin{align}\label{eq200}
\mathcal Y_{\sigma_n\cdots\sigma_2,\alpha}(w^{(\square)}_n,z_n;\dots;w^{(\square)}_1,z_1)=\mathcal Y_{\alpha_n}(w^{(\square)}_n,z_n)\cdots \mathcal Y_{\alpha_1}(w^{(\square)}_1,z_1).
\end{align}
So the smeared generalized intertwining operator will be a product of smeared intertwining operators. This shows that $\mathcal Y_{\sigma_n\cdots\sigma_2,\alpha}(w^{(\square)}_n,f_n;\dots;w^{(\square)}_1,f_1)$ has similar analytic properties as smeared intertwining operators: it is a (pre)closed unbounded operator mapping $\mathcal H_j\rightarrow \mathcal H_k$, and it satisfies the strong intertwining property.

Braiding and adjoint of smeared generalized intertwining operators are much harder to prove than those analytic properties. The difficulty is mainly on the unsmeared side: we want to determine the braid relation
\begin{align}
&\mathcal Y_{\tau_m\cdots\tau_2,\beta}(\widetilde w^{(\square)}_m,\zeta_m;\cdots \widetilde w^{(\square)}_1,\zeta_1)\mathcal Y_{\sigma_n\cdots\sigma_2,\alpha}( w^{(\square)}_n,z_n;\cdots w^{(\square)}_1,z_1)\nonumber\\
=&\mathcal Y_{\sigma_n\cdots\sigma_2,?}( w^{(\square)}_n,z_n;\cdots w^{(\square)}_1,z_1)\mathcal Y_{\tau_m\cdots\tau_2,?}(\widetilde w^{(\square)}_m,\zeta_m;\cdots \widetilde w^{(\square)}_1,\zeta_1)\label{eq201}
\end{align}
and the adjoint relation (when the vectors are quasi-primary)
\begin{align*}
\mathcal Y_{\sigma_n\cdots\sigma_2,\alpha}( w^{(\square)}_n,z_n;\cdots w^{(\square)}_1,z_1)^\dagger=(\cdots)\cdot \mathcal Y_{?\cdots?,\alpha^*}( \overline{w^{(\square)}_n},\overline{z_n^{-1}};\cdots \overline{w^{(\square)}_1},\overline{z_1^{-1}}).
\end{align*}
These problems  will be treated in part II of this series. However, certain preparatory results, including general fusion relations (that generalized intertwining operators can be written as the products of several intertwining operators), general braid relations (braiding of the products of more than two intertwining operators), and the well-definedness (convergence) of the products of generalized intertwining operators (the convergence of both sides of \eqref{eq201} for instance),  will be proved in this paper. 

\subsubsection*{Outline of this paper}

Part I is organized as follows. In chapter 1 we review the basic definitions of (unitary) VOAs, their (unitary) representations, and intertwining operators. We define unitary representations of unitary VOAs, adjoint intertwining operators, creation and annihilation operators, and prove some basic properties.

Chapter 3 is devoted to the study of energy bounds condition and smeared intertwining operators. In section 3.1 we define the energy bounds condition for intertwining operators, and give some useful criteria. In section 3.2 we define, for energy-bounded intertwining operators, the corresponding smeared intertwining operators. We prove the braid relations, the adjoint relation, and the strong intertwining property of smeared intertwining operators. We also prove the rotation covariance of smeared intertwining operators\footnote{In fact, the more general conformal covariance can be proved for smeared intertwining operators using the similar argument for proving the conformal covariance of smeared vertex operators (cf. \cite{CKLW} proposition 6.4).}, which will be used in part II to prove some density results.

The purpose of chapter 2 needs more explanations. One of the main goals of this chapter is to give a brief and self-contained introduction to Huang-Lepowsky's tensor product theory of rational VOAs based on the braid and fusion relations of intertwining operators. So, unlike chapter 1, before reading which we suggest that the reader has some basic knowledge on VOAs, this chapter does not require any previous knowledge on Huang-Lepowsky's theory. Moreover,  the results that we shall cite but not prove again in our papers will be kept to a minimum. Such results include: (1) The absolute convergence of the products of intertwining operators (theorem \ref{lb65}). (2) The analytic continuation principle of (chiral) correlation functions due to the existence of holomorphic differential equations (theorem \ref{lb75}.) (3) The existence of the fusion relation for two intertwining operators. (Theorem \ref{lb73} in the special case when $n=2$. The proof of the general case in section A.3 relies on this special case.) (4) The rigidity of the braided tensor category of rational VOAs. (It will only be used in part II to prove the non-degeneracy of $\Lambda$.) 

All the other  results used in our theory are proved in chapter 2 or A. These results are either known in Huang-Lepowsky's theory explicitly or implicitly, or can be easily derived using the machinery they have developed. Such results  include the description of a linear basis of the vector space of correlation functions (proposition \ref{lb66}), the braiding of two or more than two intertwining operators (theorem \ref{lb78}), the relation between the braiding of intertwining operators and the maps $B_{\pm}:\mathcal V{k\choose i~j}\rightarrow\mathcal V{k\choose j~i}$ introduced in section 1.3 (proposition \ref{lb86}), and the fusion and braiding of intertwining operators with vertex operators or creation operators (section 2.3). We give complete proofs of these results in this paper, since  the language and notations  used in Huang and Lepowsky's papers are very different from ours, and also because there are many analytic subtleties in the proofs of these results.\footnote{As an examples of these analytic subtleties, let us assume that we have a braid relation of intertwining operators looking like $AB=B'A'$. If we have another intertwining operator $C$, then the braid relation of three intertwining operators $CAB=CB'A'$ does \emph{not} follow directly from ``multiplying" both sides of the original braid relation by $C$. As we have emphasized before, the braiding of intertwining operators is understood using analytic continuation, but not as the direct composition of operators. Therefore, the braiding of several intertwining operators does not follow from that of two intertwining operators through a direct and algebraic argument. } 
Readers with a background in functional analysis or conformal net might  especially care about these subtleties.

What's new in chapter 2 is the convergence of the products of generalized intertwining operators (theorem \ref{lb12}). Another type of convergence property (corollary \ref{lb13}), which will be used in part II to prove the braid and adjoint relations of generalized (smeared) intertwining operators, is also  given. The conditions on the complex variables under which the absolute convergence holds are especially important for our theory.

%In chapter 2 we review and study many genus-0 geometric properties of the intertwining operators of rational VOAs. In section 2.1 we review Huang's construction of genus-0 (chiral) correlation functions using products of intertwining operators and analytic continuation (cf.\cite{H ODE}). In section 2.2 we give the precise statements of the general fusion  and braid relations for intertwining operators (theorems \ref{lb73} and \ref{lb78}), and  The proofs of these results are technical. So they are postponed to section A.3. 

%In section 2.3 we discuss some important fusion and braid relations. The first one is the braiding and fusion of an intertwining operator with a vertex operator. This relation can be regarded as a geometric interpretation of  Jacobi identity in the definition of intertwining operators and vertex operators. In fact,  Jacobi identity might seem weird to those who encounter VOAs and their intertwining operators for the first time. We hope that our discussion in this section (see propositions \ref{lb80} and \ref{lb81} and their proofs) will make Jacobi identity more understandable. Then we discuss fusion and braiding with creation operators. Braiding of creation operator gives a geometric explanation of the braid construction $B_\pm:\mathcal V{k\choose i~j}\rightarrow\mathcal V{k\choose j~i}$ introduced in section 1.3. Finally, section 2.4 is a review of Huang-Lepowski's tensor categories for rational VOAs. 

\subsubsection*{Acknowledgment}

The author would like to thank Professor Vaughan Jones. This paper, as well as the second half of the series, cannot be finished without his constant support, guidance, and encouragement. The author was supported by NSF grant DMS-1362138.

\subsubsection*{Notations.}
In this paper, we  assume that $V$ is a vertex operator algebra of CFT type. Except in chapter \ref{lb93}, we assume that $V$ also satisfies the following conditions:
\begin{flalign}
&\text{(1) $V$ is isomorphic to $V'$.}&\label{eq302}\\
&\text{(2) Every $\mathbb N$-gradable weak $V$-module is completely reducible.}&\label{eq303}\\
&\text{(3) $V$ is $C_2$-cofinite.}&\label{eq304}
\end{flalign}
(See \cite{H MI} for the definitions of these terminologies.) The following notations are used throughout this paper.\\

\noindent
$A^\tr$: the transpose of the linear operator $A$.\\
$A^\dagger$: the formal adjoint of the linear operator $A$.\\
$A^*$: the ajoint of the possibly unbounded linear operator $A$.\\
$\overline A$: the closure of the pre-closed linear operator $A$.\\
$C_i$: the antiunitary map $W_i\rightarrow W_{\overline i}$.\\
$\mathbb C^\times=\{z\in\mathbb C:z\neq0 \}$.\\
$\Conf_n(\mathbb C^\times)$: the $n$-th configuration space of $\mathbb C^\times$.\\
$\widetilde\Conf_n(\mathbb C^\times)$: the universal covering space of $\Conf_n(\mathbb C^\times)$.\\
$\mathscr D(A)$: the domain of the possibly unbounded operator $A$.\\
$\di\theta=  \frac {e^{i\theta}}{2\pi}d\theta$.\\
$e_r(e^{i\theta})=e^{ir\theta}\quad(-\pi<\theta<\pi)$.\\
$\mathcal E$: a complete list of mutually inequivalent irreducible $V$-modules.\\
$\mathcal E^\uni$: the set of unitary $V$-modules in $\mathcal E$.\\
$\Hom_V(W_i,W_j)$: the vector space of $V$-module homomorphisms from $W_i$ to $W_j$.\\
$\mathcal H_i$: the norm completion of the vector space $W_i$.\\ %If the $V$-module $W_i$ is a unitary, energy-bounded, and strongly integrable, then $\mathcal H_i$ is the  $\mathcal M_V$-module associated with $W_i$.
$\mathcal H^r_i$: the vectors of $\mathcal H_i$ that are inside $\mathscr D((1+\overline{L_0})^r)$.\\
$\mathcal H^\infty_i=\bigcap_{r\geq0}\mathcal H^r_i$.\\
$I^c$: the complement of the open interval $I$.\\
$I_1\subset\joinrel\subset I_2$: $I_1,I_2\in\mathcal J$ and $\overline {I_1}\subset I_2$.\\
$\id_i=\id_{W_i}$: the identity operator of $W_i$.\\
$\mathcal J$: the set of (non-empty, non-dense) open intervals of $S^1$.\\
$\mathcal J(U)$: the set of open intervals of $S^1$ contained in the open set $U$.\\
$P_s$: the projection operator of $W_i$ onto $W_i(s)$.\\
$\mathfrak r(t):S^1\rightarrow S^1$: $\mathfrak r(t)(e^{i\theta})=e^{i(\theta+t)}$.\\
$\mathfrak r(t):C^\infty(S^1)\rightarrow C^\infty(S^1)$: $\mathfrak r(t)h=h\circ\mathfrak r(-t)$.\\
$\Rep(V)$: the modular tensor category of the representations of $V$.\\
$\Rep^\uni(V)$: the category of the unitary representations of $V$.\\
$\Rep^\uni_{\mathcal G}(V)$: When $\mathcal G$ is additively closed, it is the subcategory of $\Rep^\uni(V)$ whose objects are unitary $V$-modules in $\mathcal G$. When $\mathcal G$ is multiplicatively closed, then it is furthermore equipped with the structure of a ribbon tensor category.\\
$S^1=\{z\in\mathbb C:|z|=1 \}$.\\
$\mathcal V{k\choose i~j}$: the vector space of type $k\choose i~j$ intertwining operators.\\
$W_0=V$, the vacuum module of $V$.\\
$W_i$: a $V$-module.\\
$\widehat W_i$: the algebraic completion of $W_i$.\\
$W_{\overline i}\equiv W'_i$: the contragredient module of $W_i$.\\
$W_{ij}\equiv W_i\boxtimes W_j$: the tensor product of $W_i,W_j$.\\
$w^{(i)}$: a vector in $W_i$.\\
$\overline {w^{(i)}}=C_iw^{(i)}.$\\
$x$:  a formal variable.\\
$Y_i$: the vertex operator of $W_i$.\\
$\mathcal Y_\alpha$: an intertwining operator of $V$.\\
$\mathcal Y_{\overline{\alpha}}\equiv\overline{\mathcal Y_\alpha}$: the conjugate intertwining operator of $\mathcal Y_\alpha$.\\
$\mathcal Y_{\alpha^*}\equiv\mathcal Y_{\alpha}^\dagger$: the adjoint intertwining operator of $\mathcal Y_\alpha$.\\
$\mathcal Y_{B_\pm\alpha}\equiv B_{\pm}\mathcal Y_\alpha$: the braided intertwining operators of $\mathcal Y_\alpha$.\\
$\mathcal Y_{C\alpha}\equiv C\mathcal Y_\alpha$: the contragredient intertwining operator of $\mathcal Y_\alpha$.\\
$\mathcal Y^i_{i0}$: the creation operator of $W_i$.\\
$\mathcal Y^0_{\overline ii}$: the annihilation operator of $W_i$.\\
$\Delta_i$: the conformal weight of  $W_i$.\\
$\Delta_w$: the conformal weight (the energy) of the homogeneous vector $w$.\\
$\Theta^k_{ij}$: a set of linear basis of $\mathcal V{k\choose i~j}$.\\
$\Theta^k_{i*}=\coprod_{j\in\mathcal E}\Theta^k_{ij},\Theta^k_{*j}=\coprod_{i\in\mathcal E}\Theta^k_{ij},\Theta^*_{ij}=\coprod_{k\in\mathcal E}\Theta^k_{ij}.$\\
$\theta$: the PCT operator of $V$, or a real variable.\\
$\vartheta_i$: the twist of $W_i$.\\
$\nu$: the conformal vector of $V$.\\
$\sigma_{i,j}$: the braid operator  $\sigma_{i,j}:W_i\boxtimes W_j\rightarrow W_j\boxtimes W_i$.\\
$\Omega$: the vacuum vector of $V$.

\section{Intertwining operators of unitary vertex operator algebras (VOAs)}\label{lb93}

We refer the reader to \cite{FHL} for the general theory of VOAs, their representations, and  intertwining operators. Other standard references on VOAs include \cite{Conformal blocks,FLM,VOA beginners,LL}. Unitary VOAs were defined by Dong, Lin in  \cite{DL}. Our approach in this article  follows \cite{CKLW}.
\subsection{Unitary VOAs}

Let $x$ be a formal variable. For a complex vector space $U$, we set
\begin{gather}
U[[x]]=\bigg\{\sum_{n\in\mathbb Z_{\geq0}}u_nx^n:u_n\in U\bigg\},\\
U((x))=\bigg\{\sum_{n\in\mathbb Z}u_nx^n:u_n\in U,u_n=0\text{ for sufficiently small }n\bigg\},\\
U[[x^{\pm1}]]=\bigg\{\sum_{n\in\mathbb Z}u_nx^n:u_n\in U\bigg\},\\
U\{x \}=\bigg\{\sum_{s\in\mathbb R}u_sx^s:u_s\in U\bigg\}.
\end{gather}
We define the formal derivative $\frac d{dx}$ to be
\begin{equation}
\frac d{dx}\bigg(\sum_{n\in\mathbb R}u_nx^n\bigg)=\sum_{n\in\mathbb R}nu_nx^{n-1}. 
\end{equation}

Let $V$ be a  complex vector space with grading $V=\bigoplus_{n\in\mathbb Z}V(n)$. Assume that $\dim V(n)<\infty$ for each $n\in\mathbb Z$, and $\dim V(n)=0$ for $n$ sufficiently small. We say that $V$ is a \textbf{vertex operator algebra} (VOA), if the following conditions are satisfied:\\
(a) There is a linear map 	
\begin{gather*}
V\rightarrow(\text{End }V)[[x^{\pm1}]]\\
u\mapsto Y(u,x)=\sum_{n\in\mathbb Z}Y(u,n) x^{-n-1}\\
\text{(where $Y(u,n)\in$End $V$)},
\end{gather*}
such that for any $v\in V$, $Y(u,n)v=0$ for $n$ sufficiently large.\\
(b) (\textbf{Jacobi identity}) For any $u,v\in V$ and $m,n,h\in\mathbb Z$, we have
	\begin{align}
	&\sum_{l\in\mathbb Z_{\geq0}}{m\choose l}Y(Y(u,n+l)v,m+h-l)\nonumber\\
	=&\sum_{l\in\mathbb Z_{\geq0}}(-1)^l{n\choose l}Y(u,m+n-l)Y(v,h+l)-\sum_{l\in\mathbb Z_{\geq0}}(-1)^{l+n}{n\choose l}Y(v,n+h-l)Y(u,m+l).\label{eq86}
	\end{align}
(c) There exists a vector $\Omega\in V(0)$ (the \textbf{vacuum vector}) such that $Y(\Omega,x)=\text{id}_V$.\\
(d) For any $v\in V$ and $n\in\mathbb Z_{\geq0}$, we have $Y(v,n)\Omega=0$, and $Y(v,-1)\Omega=v$. This condition is simply written as $\lim_{x\rightarrow 0}Y(v,x)\Omega=v$.\\
(e) There exists a vector $\nu\in V(2)$ (the \textbf{conformal vector}) such that the operators $L_n=Y(\nu,n+1)$ ($n\in\mathbb Z$) satisfy the Virasoro relation: $[L_m,L_n]=(m-n)L_{m+n}+\frac 1 {12}(m^3-m)\delta_{m,-n}c$. Here the number $c\in\mathbb C$ is called the \textbf{central charge} of $V$.\\
(f) If $v\in V(n)$ then $L_0v=nv$. $n$ is called the \textbf{conformal weight} (or the \textbf{energy}) of $v$ and will be denoted by $\Delta_v$. $L_0$ is called the \textbf{energy operator}.\\
(g)	(\textbf{Translation property}) $\frac d{dx} Y(v,x)=Y(L_{-1}v,x)$.\\

\begin{cv}
In this article, we  always assume that $V$ is a VOA of \textbf{CFT type}, i.e., $V(0)=\mathbb C\Omega$, and $\dim V(n)=0$ when $n<0$.
\end{cv}

Given a (anti)linear bijective map $\phi:V\rightarrow V$, we say that $\phi$ is an  \textbf{(antilinear) automorphism} of $V$ if the following conditions are satisfied:
\begin{flalign}
&(a) \phi\Omega=\Omega,~~\phi\nu=\nu.&\\
&(b) \text{For any }v\in V,~~\phi Y(v,x)=Y(\phi v,x)\phi.\label{eq92}&
\end{flalign}
It is easy to deduce from these two conditions that $\phi L_n=L_n\phi$ (for any $n\in\mathbb Z$). In particular, since $\phi$ commutes with $L_0$, we have $\phi V(n)=V(n)$ for each $n\in\mathbb Z$.

\begin{df}
 Suppose that  $V$ is equipped with an  inner  product $\langle\cdot|\cdot\rangle$ (antilinear on the second variable) satisfying $\langle\Omega|\Omega\rangle=1$. Then we call $V$ a  \textbf{unitary vertex operator algebra}, if there exists  an antilinear automorphism $\theta$, such that  for any $v\in V$ we have
 \begin{equation}\label{eq94}
 Y(v,x)^\dagger=Y(e^{xL_1}(-x^{-2})^{L_0}\theta v,x^{-1}),
 \end{equation}
 where $\dagger$ is the formal adjoint operation. More precisely, this equation means that for any $v,v_1,v_2\in V$ we have
 \begin{equation}\label{eq91}
 \langle Y(v,x)v_1|v_2\rangle=\langle v_1|Y(e^{xL_1}(-x^{-2})^{L_0}\theta v,x^{-1})v_2\rangle.
 \end{equation}
\end{df}

\begin{rem}
Such $\theta$, if exists, must be unique. Moreover, $\theta$ is anti-unitary (i.e. $\langle \theta v_1|\theta v_2\rangle=\langle v_2|v_1\rangle$ for any $v_1,v_2\in V$), and $\theta^2=id_V$ (i.e. $\theta$ is an involution). We call $\theta$ the \textbf{PCT operator} of $V$. (cf. \cite{CKLW} proposition 5.1.) In this article, $\theta$  denotes either the PCT operator of $V$, or a real variable. These two meanings will be used in different situations. So no confusion will arise.
\end{rem}

We say that a vector $v\in V$ is \textbf{homogeneous} if $v\in V(n)$ for some $n\in \mathbb Z$. If moreover, $L_1v=0$,  we say that $v$ is   \textbf{quasi-primary}. It is clear that the vacuum vector $\Omega$ is quasi-primary. By translation property, we have $L_{-1}\Omega=0$. Therefore, $L_1\nu=L_1Y(\nu,-1)\Omega=L_1L_{-2}\Omega=[L_1,L_{-2}]\Omega=3L_{-1}\Omega=0$. We conclude that \emph{the conformal vector is quasi-primary}. 
 
  Now suppose that $V$ is unitary and $v\in V$ is quasi-primary, then equation \eqref{eq94} can be simplified to
 \begin{gather}
 Y(v,x)^\dagger=(-x^{-2})^{\Delta_v}Y(\theta v,x^{-1}).\label{eq95}
 \end{gather}
 If we take $v=\nu$, then we obtain
 \begin{gather}
 L_n^\dagger=L_{-n}~~~(n\in\mathbb Z).
 \end{gather}
In particular, we have $L_0^\dagger=L_0$. This shows that different energy subspaces are orthogonal, i.e., \emph{the grading $V=\bigoplus_{n\geq0}V(n)$ is orthogonal under the inner product $\langle\cdot|\cdot\rangle$ .}

\subsection{Unitary representations of unitary VOAs}\label{lb103}

\begin{df}
Let $W_i$ be a complex vector space with grading
$W_i=\bigoplus_{s\in\mathbb R}W_i(s)$. Assume $\text{dim }W_i(s)<\infty$ for each $s\in\mathbb R$, and $\dim W_i(s)=0$ for $s$ sufficiently small. We say that $W_i$ is a \textbf{representation of $V$} (or \textbf{$V$-module}), if the following conditions are satisfied:\\
(a) There is a linear map 	
\begin{gather*}
V\rightarrow(\text{End }W_i)[[x^{\pm1}]]\\
v\mapsto Y_i(v,x)=\sum_{n\in\mathbb Z}Y_i(v,n) x^{-n-1}\\
\text{(where $Y(v,n)\in$End $W_i$)},
\end{gather*}
such that for any $w^{(i)}\in W_i$, $Y_i(v,n)w^{(i)}=0$ for $n$ sufficiently large. $Y_i$ is called the \textbf{vertex operator} of $W_i$.\\
(b) (\textbf{Jacobi identity}) For any $u,v\in V$ and $m,n,h\in\mathbb Z$, we have
	\begin{align}
	&\sum_{l\in\mathbb Z_{\geq0}}{m\choose l}Y_i(Y(u,n+l)v,m+h-l)\nonumber\\
	=&\sum_{l\in\mathbb Z_{\geq0}}(-1)^l{n\choose l}Y_i(u,m+n-l)Y_i(v,h+l)-\sum_{l\in\mathbb Z_{\geq0}}(-1)^{l+n}{n\choose l}Y_i(v,n+h-l)Y_i(u,m+l).\label{eq97}
	\end{align}
(c) $Y_i(\Omega,x)=\text{id}_{W_i}$.\\
(d) The operators $L_n=Y_i(\nu,n+1)$ ($n\in\mathbb Z$) satisfy the Virasoro relation: $[L_m,L_n]=(m-n)L_{m+n}+\frac 1 {12}(m^3-m)\delta_{m,-n}c$, where $c$ is the central charge of $V$.\\
(e) If $w^{(i)}\in W_i(s)$ then $L_0w^{(i)}=sw^{(i)}$. $s$ is called the \textbf{conformal weight} (or the \textbf{energy}) of $w^{(i)}$ and will be denoted by $\Delta_{w^{(i)}}$, and $L_0$ is called the \textbf{energy operator}.\\
(f)	(\textbf{Translation property}) $\frac d{dx} Y_i(v,x)=Y_i(L_{-1}v,x)$.
\end{df}

Clearly $V$ itself is a representation of $V$. We  call it the \textbf{vacuum module} of $V$. Modules of $V$ are  denoted by $W_i,W_j,W_k,\dots$, or simply $i,j,k,\dots$. The vacuum module is sometimes denoted by $0$. We let $\id_i=\id_{W_i}$ and $\id_0=\id_V$ be the identity operators on $W_i$ and $V$ respectively.

A \textbf{$V$-module homomorphism} is, by definition, a linear map $\phi:W_i\rightarrow W_j$, such that for any $v\in V$ we have $\phi Y_i(v,x)=Y_j(v,x)\phi$. It is clear that $\phi$ preserves the gradings of $W_i,W_j$, for $\phi$ intertwines the actions of $L_0$ on these spaces. The vector space of homomorphisms $W_i\rightarrow W_j$ is denoted by $\Hom_V(W_i,W_j)$.

\begin{rem}\label{lb102}
If the $V$-module $W_i$ has a subspace $W$ that is invariant under the action of $V$, then the restricted action of $V$ on $W$ produces a submodule of $W_i$. In fact, the only non-trivial thing to check is that $W$ inherits the grading of $W_i$. But this follows from the fact that $L_0$, when restriced to $W$, is diagonalizable on $W$. (In general, if a linear operator of a complex vector space is diagonalizable, then by polynomial interpolations, it must also be diagonalizable on any invariant subspace.)	
\end{rem}

From the  remark above, we see that a module $W_i$ is irreducible if and only if the vector space $W_i$ has no $V$-invariant subspace. If $W_i$ is irreducible, we call $$\Delta_i=\inf\{s:\dim W_i(s)>0 \}$$  the \textbf{conformal weight} of $W_i$. It is easy to show that $W_i=\bigoplus_{n\in\mathbb Z_{\geq0}}W_i(n+\Delta_i)$.

We now review the definition of contragredient modules introduced in \cite{FHL}. Let again $W_i$ be a $V$-module. First we note that  the dual space $W_i^*$ of $W_i$ has the grading $W_i^*=\prod_{s\in\mathbb R}W_i(s)^*$.  Here $W_i(s)^*$ is the dual space of the finite dimensional vector space $W(s)$, and if $s\neq t$, the evaluations of $W_i(s)^*$ on $W_i{(t)}$ are set to be zero. Now we consider the subspace $W_{\overline i}\equiv W_i'=\bigoplus_{s\in\mathbb R}W(s)^*$  of $W^*$.  We define the action of $V$ on $W_{\overline i}$ as follows:
\begin{equation}\label{eq98}
Y_{\overline i}(v,x)=Y_i(e^{xL_1}(-x^{-2})^{L_0} v,x^{-1})^\mathrm{t}
\end{equation}
where the superscript ``t'' stands for the  transpose operation. In other words, for any $w^{(\overline i)}\in W_{\overline i}\subset W_i^*$ and $w^{(i)}\in W_i$, we have
\begin{align}
\langle Y_{\overline i}(v,x)w^{(\overline i)},w^{(i)} \rangle=\langle w^{(\overline i)},Y_i(e^{xL_1}(-x^{-2})^{L_0} v,x^{-1})w^{(i)} \rangle.
\end{align}
We refer the reader to \cite{FHL} section 5.2 for a proof that $(W_{\overline i},Y_{\overline i})$ is  a representation of $V$. This representation is called  the \textbf{contragredient module} of $W_i$. 

In general, for each $V$-module $W_i$,  the vector space $\widehat W_i=\prod_{s\in\mathbb R}W_i(s)$ is called the \textbf{algebraic completion} of $W_i$. The action $Y_i$ of $V$ on $W_i$ can be clearly extended onto $\widehat W_i$. It is clear that $\widehat W_i$ can be identified with $W_{\overline i}^*$.

Equation \eqref{eq98} can be written in terms of modes: if $v\in V$ is a quasi-primary vector with conformal weight $\Delta_v$, then
\begin{equation}
Y_{\overline i}(v,n)=\sum_{m\in\mathbb Z_{\geq0}}\frac{(-1)^{\Delta_v}}{m!}Y_i(L_1^mv,-n-m-2+2\Delta_v)^\tr.
\end{equation}
In particular, by letting $v=\nu$, we obtain $L_n^\tr=L_{-n}$. More precisely, if $w^{(i)}\in W_i,w^{(\overline i)}\in W_{\overline i}$, we have $\langle L_n w^{(i)},w^{(\overline i)} \rangle=\langle w^{(i)},L_{-n} w^{(\overline i)} \rangle$.

The contragredient operation is an involution: $W_i$ is the contragredient module of $W_{\overline i}$. In particular, we have 
\begin{equation}
Y_i(v,x)=Y_{\overline i}(e^{xL_1}(-x^{-2})^{L_0} v,x^{-1})^\mathrm{t}.
\end{equation}
Hence we identify $i$ with $\overline {\overline i}$, the contragredient module of $\overline i$.\\

Now we turn to the definition of unitary VOA modules.

\begin{df}
Suppose that $V$ is unitary and $W_i$ is a $V$-module equipped with an inner product $\langle\cdot|\cdot\rangle$. We call $W_i$ \textbf{unitary} if for any $v\in V$ we have
	\begin{gather}
	Y_i(v,x)^\dagger=Y_i(e^{xL_1}(-x^{-2})^{L_0}\theta v,x^{-1}).\label{eq219}
	\end{gather}
\end{df}

In the remaining part of this section, we assume that $V$ is unitary. Let $W_i$ be a unitary $V$-module. Then formula  \eqref{eq219}, with $v=\nu$, implies that the action of the Virasoro subalgebras $\{L_n\}$ on $W_i$ satisfies $L_n^\dagger=L_{-n}$.  In particular, $L_0$ is symmetric, and hence the decompsition $W_i=\bigoplus_{s\in\mathbb R}W_i(s)$ is orthogonal. If we let $P_s$ be the projection operator of $W_i$ onto $W_i(s)$ (this operator can be defined whether $W_i$ is unitary or not),  we have $P_s^\dagger=P_s$. 

\begin{pp}[Positive energy]\label{lb88}
If $W_i$ is unitary, then we have the grading $W_i=\bigoplus_{s\geq0}W_i(s)$. In particular, if $W_i$ is irreducible, then $\Delta_i\geq0$.
\end{pp}
\begin{proof}
We choose an arbitrary non-zero homogeneous  vector $w^{(i)}\in W_i$ and show that $\Delta_{w^{(i)}}\geq0$. First, assume that $w^{(i)}$ is  \textbf{quasi-primary} (i.e., $L_1w^{(i)}=0$). Then we have
\begin{gather*}
2\Delta_{w^{(i)}}\langle w^{(i)}|w^{(i)}\rangle =2\langle L_0w^{(i)}|w^{(i)}\rangle =\langle [L_1,L_{-1}]w^{(i)}|w^{(i)}\rangle =\lVert L_{-1}w^{(i)}\lVert^2\geq0,
\end{gather*}
which implies that $\Delta_{w^{(i)}}\geq0$.  In general, we may find  $m\in\mathbb Z_{\geq0}$ such that $L_1^mw^{(i)}\neq0$, and $L_1^{m+1}w^{(i)}=0$. So $\Delta_{L_1^mw^{(i)}}\geq0$, and hence $\Delta_{w^{(i)}}=\Delta_{L_1^mw^{(i)}}+m\geq0$. 
\end{proof}

\begin{pp}\label{lb64}
If $W_i$ is unitary, then its contragredient module $W_{\overline i}$ is unitarizable.
\end{pp}
\begin{proof}
Assume that  $W_i$ has inner product $\langle\cdot|\cdot\rangle$ . Define an anti-linear bijective map $C_i:W_i\rightarrow W_{\overline i}$ such that  $\langle C_iw_1^{(i)},w_2^{(i)} \rangle=\langle w_2^{(i)}|w_1^{(i)} \rangle$ for any $w_1^{(i)},w_2^{(i)}\in W$. We simply write $C_iw^{(i)}=\overline {w^{(i)}}$. Now we may define the inner product on $W_{\overline i}$ such that $C_i$ becomes   antiunitary. 

For any $v\in V$, we show that $Y_{\overline i}(v,x)$ satisfies equation \eqref{eq219}. Note that for any  $A\in\mathrm{End}(W_i)$, if $A$ has a  transpose $A^{\tr}\in\mathrm{End}(W_{\overline i})$, then $A$ also has a formal adjoint $A^\dagger\in\End(W)$, and it satisfies $A^\dagger=C_i^{-1}A^{\tr}C_i$. Thus we have 
\begin{align}
&Y_{\overline{i}}(v,x)=Y_i(e^{xL_1}(-x^{-2})^{L_0} v,x^{-1})^{\tr}\nonumber\\
=&C_iY_i(e^{xL_1}(-x^{-2})^{L_0} v,x^{-1})^\dagger C_i^{-1}=C_iY_i(\theta v,x)C_i^{-1},\label{eq104}
\end{align}
which implies that $Y_{\overline i}$ satisfies \eqref{eq219}.
\end{proof}
From now on, if $W_i$ is a unitary $V$-module, we fix an inner product on $W_{\overline i}$ to be the one constructed in the proof of proposition \ref{lb64}. We view $W_{\overline i}$ as a unitary $V$-module under this inner product.

Note that if we let $v=\nu$, then \eqref{eq104} implies that $L_nC_i=C_iL_n$ ($n\in\mathbb Z$).

Since we use  $W_0$ (or simply $0$) to denote the vacuum module $V$, it is natural to let $C_0$ represent the conjugation map from $V$ onto its  contragredient module $W_{\overline 0}\equiv V'$. By equation \eqref{eq104} (with $i=0$) and \eqref{eq92}, we have:
\begin{co}\label{lb83}
$C_0\theta: V\rightarrow  V'$ is a unitary $V$-module isomorphism.
\end{co}

Therefore, we identify the vacuum module $V$ with its contragredient module $V'$. This fact can be simply written as $\overline 0=0$. The operators $\theta$ and $C_0$ are also identified. The evaluation map $V\otimes V'\rightarrow \mathbb C$ is equivalent to the symmetric bilinear form $V\otimes V\rightarrow \mathbb C$ defined by $\langle v_1,v_2\rangle=\langle v_1|\theta v_2\rangle$, where $v_1,v_2\in V$.

Recall that we also identify $W_i$ with $W_{\overline{\overline i}}$. It is easy to see that the anti-unitary map $C_{\overline i}:W_{\overline i}\rightarrow W_i=W_{\overline{\overline i}}$ satisfies $C_{\overline i}=C_i^{-1}$.

We now give a criterion for unitary $V$-modules. First, we say that $V$ is \textbf{generated by} a subset $E$ if $V$ is spanned by vectors of the form $Y(v_1,n_1)\cdots Y(v_k,n_m)\Omega$ where $v_1,v_2,\dots,v_m\in E$ and $n_1,\dots,n_m\in\mathbb Z$. By the Jacoby identity \eqref{eq97} (with $m=0$), any vertex operator $Y_i$ is determined by its values on $E$. 

Now we have a useful criterion for unitarity of $V$-modules.

\begin{pp}\label{lb116}
If $V$ is unitary, $W_i$ is a $V$-module equipped with an inner product $\langle \cdot|\cdot\rangle$, $E$ is a generating subset of $V$, and  equation \eqref{eq219} holds under the inner product $\langle \cdot|\cdot\rangle$ for any $v\in E$, then $W_i$ is a unitary $V$-module.
\end{pp}

\begin{proof}
For any $v\in V$ we define $\widetilde Y_i(v,x)=Y_i(e^{xL_1}(-x^{-2})^{L_0}\theta v,x^{-1})^\dagger$. As in the proof of proposition \ref{lb64}, we have $\widetilde Y_i(v,x)=C_i^{-1}Y_{\overline i}(\theta v,x)C_i$. It follows that $\widetilde Y_i$ satisfies the  Jacobi identity. Since $Y_i$ also satisfies the Jacobi identity, and since $Y_i(v,x)=\widetilde Y_i(v,x)$ for any $v\in E$, we must have $Y_i(v,x)=\widetilde Y_i(v,x)$ for all $v\in V$, which proves that $W_i$ is unitary.
\end{proof}

\subsection{Intertwining operators of unitary VOAs}\label{lb29}
\begin{df}\label{lb28}
Let $W_i,W_j,W_k$ be $V$-modules. A type $W_k\choose W_iW_j$ (or type $k\choose i~j$) \textbf{intertwining operator} $\mathcal Y_\alpha$
 is a linear map 
 \begin{gather*}
 W_i\rightarrow(\Hom(W_j,W_k))\{x\},\\
 w^{(i)}\mapsto \mathcal Y_\alpha(w^{(i)},x)=\sum_{s\in\mathbb R}\mathcal Y_\alpha(w^{(i)},s) x^{-s-1}\\
 \text{ (where $\mathcal Y_\alpha(w^{(i)},s)\in\Hom(W_j,W_k)$)},
 \end{gather*}
 such that:\\
 (a) For any $w^{(j)}\in W_j$, $\mathcal Y_\alpha(w^{(i)},s)w^{(j)}=0$ for $s$ sufficiently large.\\
 (b) (\textbf{Jacobi identity}) For any $u\in V,w^{(i)}\in W_i,m,n\in\mathbb Z,s\in\mathbb R$, we have
	\begin{align}
	&\sum_{l\in\mathbb Z_{\geq0}}{m\choose l}\mathcal Y_\alpha\big(Y_i(u,n+l)w^{(i)},m+s-l\big)\nonumber\\
	=&\sum_{l\in\mathbb Z_{\geq0}}(-1)^l{n\choose l}Y_k(u,m+n-l)\mathcal Y_\alpha(w^{(i)},s+l)\nonumber\\
	&-\sum_{l\in\mathbb Z_{\geq0}}(-1)^{l+n}{n\choose l}\mathcal Y_\alpha(w^{(i)},n+s-l)Y_j(u,m+l).\label{eq128}
	\end{align}
(c)	(\textbf{Translation property}) $	\frac d{dx} \mathcal Y_\alpha(w^{(i)},x)=\mathcal Y_\alpha(L_{-1}w^{(i)},x)$.
\end{df}

Intertwining operators will be denoted by $\mathcal Y_\alpha,\mathcal Y_\beta,\mathcal Y_\gamma,\dots$, or just $\alpha,\beta,\gamma,\dots$.

Note that if we let $n=0$ and $m=0$ respectively, \eqref{eq128} becomes:
\begin{flalign}
&\sum_{l\geq0}{m \choose l}\mathcal Y_\alpha(Y_i(u,l)w^{(i)},m+s-l)=Y_k(u,m)\mathcal Y_\alpha(w^{(i)},s)-\mathcal Y_\alpha(w^{(i)},s)Y_j(u,m),\label{eq106} \\
&\mathcal Y_\alpha(Y_i(u,n)w^{(i)},s)\nonumber\\
	=&\sum_{l\geq0}(-1)^l{n\choose l}Y_k(u,n-l)\mathcal Y_\alpha(w^{(i)},s+l)-\sum_{l\geq0}(-1)^{l+n}{n\choose l}\mathcal Y_\alpha(w^{(i)},n+s-l)Y_j(u,l).
\end{flalign}
In particular, if we let $u=\nu$ and $m=0,1$ respectively,  the first equation implies that
\begin{gather}
[L_{-1},\mathcal Y_\alpha(w^{(i)},x)]=\mathcal Y_\alpha(L_{-1}w^{(i)},x)=\frac d{dx}\mathcal Y_\alpha(w^{(i)},x);\\
[L_0,\mathcal Y_\alpha(w^{(i)},x)]=\mathcal Y_\alpha(L_0w^{(i)},x)+\frac d{dx}\mathcal Y_\alpha(w^{(i)},x).\label{eqa2}
\end{gather}
The second equation is equivalent to that
\begin{align}
[L_0,\mathcal Y_\alpha(w^{(i)},s)]=(-s-1+\Delta_{w^{(i)}})\mathcal Y_\alpha(w^{(i)},s)~~~\text{if $w^{(i)}$ is homogeneous.}\label{eq132}
\end{align}
Hence $\mathcal Y_\alpha(w^{(i)},s)$ raises the energy by $-s-1+\Delta_{w^{(i)}}$. Equation \eqref{eq132} implies the relation
\begin{align}
z^{L_0}\mathcal Y_\alpha(w^{(i)},x)z^{-L_0}=\mathcal Y_{\alpha}(z^{L_0}w^{(i)},zx)\label{eq231}
\end{align}
(cf. \cite{FHL} section 5.4), where $z$ is either a  non-zero complex number, or a formal variable which commutes with and  is independent of $x$. In the former case, we need to assign to $z$ an argument, i.e., a real number  $\arg z$ such that $z=|z|e^{i\arg z}$. Then, for any $s\in\mathbb R$, we let $z^{s}=|z|^{s}e^{is\arg z}$, i.e., we let the argument of $z^s$ be $s\arg z$. 
\begin{cv}In this article, unless otherwise stated, we make the following assumptions:\\
(1) If $t\in\mathbb R$ then $\arg e^{it}=t$.\\
(2) If $z\in\mathbb C^\times$ with argument $\arg z$, then  $\arg \overline z=-\arg z$. If $s\in\mathbb R$, then  $\arg (z^s)=s\arg z$.\\
(3) If $z_1,z_2\in\mathbb C^\times$ with arguments $\arg z_1$ and $\arg z_2$ respectively, then $\arg(z_1z_2)=\arg z_1+\arg z_2$.
\end{cv}
\begin{df}
Let  $U$ be an open subset of	$\mathbb C$ and $f:U\rightarrow \mathbb C^\times$ be a continuous function.  Suppose that $z_1,z_2\in U$, and for any $t\in[0,1]$, $tz_1+(1-t)z_2\in U$. Then we say that the argument  \textbf{$\arg f(z_2)$ is close to $\arg f(z_1)$ as $z_2\rightarrow z_1$}, if there exists a (unique) continuous function $A:[0,1]\rightarrow \mathbb R$, such that $A(0)=\arg z_1,A(1)=\arg z_2$, and that for any $t\in[0,1]$, $A(t)$ is an argument of $f(tz_1+(1-t)z_2)$.
\end{df}

Let $\mathcal V{k\choose i~j}$ be the vector space of type $k\choose i~j$ intertwining operators. If $\mathcal Y_\alpha\in\mathcal V{k\choose i~j}$, we say that $W_i,W_j$ and $W_k$ are the \textbf{charge space}, the \textbf{source space}, and the \textbf{target space} of $\mathcal Y_\alpha$ respectively. We say that $\mathcal Y_\alpha$ is \textbf{irreducible} if $W_i,W_j,W_k$ are irreducible $V$-modules.   If $\mathcal Y_\alpha$ is irreducible, then by \eqref{eq132}, it is easy to see  that   $\mathcal Y(w^{(i)},s)=0$ except possibly when $s\in\Delta_i+\Delta_j-\Delta_k+\mathbb Z$. If $V$ is unitary, and $W_i,W_j,W_k$ are unitary $V$-modules, then we say that $\mathcal Y_\alpha$ is \textbf{unitary}.\\

We have several ways to construct new  intertwining operators from old ones. First, for any $\mathcal Y_\alpha\in\mathcal V{k\choose i~j}$, we define its \textbf{contragredient intertwining operator} (cf. \cite{FHL}) $C\mathcal Y_\alpha\equiv\mathcal Y_{C\alpha}\in\mathcal V{\overline j\choose i~\overline k}$ by letting
\begin{gather}
\mathcal Y_{C\alpha}(w^{(i)},x)=\mathcal Y_\alpha(e^{xL_1}(e^{-i\pi}x^{-2})^{L_0}w^{(i)},x^{-1})^\tr,~~~w^{(i)}\in W_i.\label{eq232}
\end{gather}
In other words, if $w^{(j)}\in W_j$ and $w^{(\overline k)}\in W_{\overline k}$, then
\begin{gather}
\langle  \mathcal Y_{C\alpha}(w^{(i)},x)w^{(\overline k)},w^{(j)} \rangle=\langle w^{(\overline k)}, \mathcal Y_\alpha(e^{xL_1}(e^{-i\pi}x^{-2})^{L_0}w^{(i)},x^{-1})w^{(j)}  \rangle.
\end{gather}

We also define, for each $\mathcal Y_\alpha\in\mathcal V{k\choose i~j}$, an intertwining operator $C^{-1}\mathcal Y_\alpha\equiv\mathcal Y_{C^{-1}\alpha}\in\mathcal V{\overline j\choose i~\overline k}$ such that
\begin{gather}
\mathcal Y_{C^{-1}\alpha}(w^{(i)},x)=\mathcal Y_\alpha(e^{xL_1}(e^{i\pi}x^{-2})^{L_0}w^{(i)},x^{-1})^\tr,~~~w^{(i)}\in W_i.\label{eq102}
\end{gather}
One can show that $C^{-1}C\alpha=CC^{-1}\alpha=\alpha$. (To prove this, we first show that $(xL_1)x_0^{L_0}=x_0^{L_0}(xx_0L_1)$ by checking this relation on any homogeneous vector. We then
show that
\begin{align}
e^{xL_1}x_0^{L_0}=x_0^{L_0}e^{xx_0L_1}\label{eq101},
\end{align}
where $x,x_0$ are independent commuting formal variables. Finally, we may use \eqref{eq101} to prove the desired relation.)

We now define, for any $\mathcal Y_\alpha\in\mathcal V{k\choose i~j}$, a pair of \textbf{braided intertwining operators} (cf. \cite{FHL}) $B_\pm\mathcal Y_\alpha\equiv\mathcal Y_{B_\pm\alpha}\in\mathcal V{k\choose j~i}$  in the following way: If $w^{(i)}\in W_i,w^{(j)}\in W_j$, then
\begin{gather}
\mathcal Y_{B_+\alpha}(w^{(j)},x)w^{(i)}=e^{xL_{-1}}\mathcal Y_\alpha(w^{(i)},e^{i\pi}x)w^{(j)},\\
\mathcal Y_{B_-\alpha}(w^{(j)},x)w^{(i)}=e^{xL_{-1}}\mathcal Y_\alpha(w^{(i)},e^{-i\pi}x)w^{(j)}.
\end{gather}
It's easy to see that $B_\mp$ is the inverse operation of $B_\pm$. We refer the reader to \cite{FHL} chapter 5 for a proof that contragredient  intertwining operators and braided  intertwining operators satisfy the Jacobi identity.\\

In the remaining part of this section, we  assume that $V$ is unitary. Let $W_i,W_j,W_k$ be unitary $V$-modules with conjugation maps $C_i:W_i\rightarrow W_{\overline i}, C_j:W_j\rightarrow W_{\overline j},C_k:W_k\rightarrow W_{\overline k}$ respectively. Given $\mathcal Y_\alpha\in\mathcal V{k\choose i~j}$, we define its \textbf{conjugate intertwining operator} $\overline{\mathcal Y_\alpha}\equiv\mathcal Y_{\overline\alpha} \in\mathcal V{\overline k\choose \overline i~\overline j}$ by setting
\begin{align}
\mathcal Y_{\overline\alpha}(\overline{w^{(i)}},x)=C_k\mathcal Y_\alpha(w^{(i)},x)C_j^{-1},~~~w^{(i)}\in W_i.
\end{align}
It is clear that $\mathcal Y_{\overline{\alpha}}$ satisfies the Jacobi identity.

For any $\mathcal Y_\alpha\in\mathcal V{k\choose i~j}$, it is easy to check that
$$\mathcal Y_{\overline{B_\pm\alpha}}=\mathcal Y_{B_\mp\overline\alpha},\qquad\mathcal Y_{\overline{C^{\pm1}\alpha}}=\mathcal Y_{C^{\mp1}\overline\alpha}.$$
We define $\mathcal Y_\alpha^\dagger\equiv\mathcal Y_{\alpha^*}=\mathcal Y_{\overline{C\alpha}}\in\mathcal V{j\choose \overline i~k}$ and call it the \textbf{adjoint intertwining operator} of $\mathcal Y_{\alpha}$. One can easily check, for any $w^{(i)}\in W_i$, that
\begin{gather}
\mathcal Y_{\alpha^*}(\overline{w^{(i)}},x)=\mathcal Y_\alpha(e^{xL_1}(e^{-i\pi}x^{-2})^{L_0}w^{(i)},x^{-1})^\dagger.\label{eq230}
\end{gather}
where the symbol $\dagger$ on the right hand side  means the formal adjoint. In other words, for any $w^{(j)}\in W_j, j,w^{(k)}\in W_k$, we have
\begin{align}
\langle \mathcal Y_{\alpha^*}(\overline{w^{(i)}},x)w^{(k)}|w^{(j)}  \rangle=\langle w^{(k)}|\mathcal Y_\alpha(e^{xL_1}(e^{-i\pi}x^{-2})^{L_0}w^{(i)},x^{-1})w^{(j)} \rangle.
\end{align}
If $w^{(i)}$ is  homogeneous,  we can write \eqref{eq230} in terms of  modes:
\begin{align}
\mathcal Y_{\alpha^*}(\overline{w^{(i)}},s)=\sum_{m\in\mathbb Z_{\geq0}}\frac {e^{i\pi\Delta_{w^{(i)}}}} {m!}\mathcal Y(L_1^mw^{(i)},-s-m-2+2\Delta_{w^{(i)}})^\dagger\label{eq213}
\end{align}
for all $s\in\mathbb R$.

It is also obvious that the adjoint operation  is an involution, i.e., $\mathcal Y_{\alpha^{**}}=\mathcal Y_\alpha$. Hence $*:\mathcal V{k\choose i~j}\rightarrow \mathcal V{j\choose\overline i~k}$ is an antiunitary map.

We define the cardinal number $N^k_{ij}$ to be the dimension of the vector space $\mathcal V{k\choose i~j}$. $N^k_{ij}$ is called a \textbf{fusion rule} of $V$. The above constructions of intertwining operators imply the following:
\begin{equation}
N^k_{ij}=N^{\overline j}_{i\overline k}=N^k_{ji}=N^{\overline k}_{\overline i~\overline j}=N^j_{\overline ik}.
\end{equation}\\

 We now construct several  intertwining operators related to a given $V$-module $W_i$. First, note that $Y_i\in\mathcal V{i\choose 0~i}$. It is obvious that  $B_+Y_i=B_-Y_i\in\mathcal V{i\choose i~0}$. We define $\mathcal Y^i_{i0}=B_\pm Y_i$ and call it the \textbf{creation operator}  of $W_i$. Using the definition of $B_\pm$, we have, for any $w^{(i)}\in W_i,v\in V$,
\begin{equation}\label{eq120}
\mathcal Y^i_{i0}(w^{(i)},x)v=e^{xL_{-1}}Y_i(v,-x)w^{(i)}.
\end{equation}
In particular, we have
\begin{equation}\label{eq233}
\mathcal Y^i_{i0}(w^{(i)},x)\Omega=e^{xL_{-1}}w^{(i)}.
\end{equation}

We define $\mathcal Y^0_{i\overline i}:=C^{-1}\mathcal Y^i_{i0}=C^{-1}B_\pm Y_i\in\mathcal V{0\choose i~\overline i}$.  Thus for any $w^{(i)}_1\in W_i$ and $w^{(\overline i)}_2\in W_{\overline i}$, we may use \eqref{eq233} and \eqref{eq101} to compute that
\begin{align}
\langle \mathcal Y^0_{ i \bar i}(w^{(i)}_1,x)w^{(\overline i)}_2,\Omega\rangle=&\langle w^{(\overline i)}_2,\mathcal Y^i_{i0}(e^{xL_1}(e^{i\pi}x^{-2})^{L_0}w^{(i)}_1,x^{-1})\Omega\rangle\nonumber\\
=&\langle w^{(\overline i)}_2,e^{x^{-1}L_{-1}}e^{xL_1}(e^{i\pi}x^{-2})^{L_0}w^{(i)}_1\rangle\nonumber\\
=&\langle e^{x^{-1}L_{1}}w^{(\overline i)}_2,e^{xL_1}(e^{i\pi}x^{-2})^{L_0}w^{(i)}_1\rangle\nonumber\\
=&\langle e^{x^{-1}L_{1}}w^{(\overline i)}_2,(e^{i\pi}x^{-2})^{L_0}e^{-x^{-1}L_1}w^{(i)}_1\rangle.\label{eq2}
\end{align}

Note that by \eqref{eq98}, $Y_{\overline i}=C^{\pm1}Y_i\in\mathcal V{\overline i\choose 0~\overline i}$.  $\mathcal Y^0_{\overline ii}=C^{-1}B_\pm Y_{\overline i}$ is called the \textbf{annihilation operator} of $W_i$.

Define $\vartheta_i\in \End_V(W_i)$ by setting $\vartheta_i=e^{2i\pi L_0}$. That $\vartheta_i$ is a $V$-module homomorphism follows from \eqref{eq231}. $\vartheta_i$ is called the \textbf{twist} of $W_i$. Then the intertwining operators $\mathcal Y^0_{i\overline i}$ and $\mathcal Y^0_{\overline i i}$ can be related in the following way: 
\begin{pp}
\begin{gather}
\mathcal Y^0_{i\overline i}(w^{( i)},x)=(B_+\mathcal Y^0_{\overline ii})(\vartheta_i w^{(i)},x)=(B_-\mathcal Y^0_{\overline ii})(\vartheta_i^{-1} w^{(i)},x),\label{eq152}\\
\mathcal Y^0_{i\overline i}(w^{( i)},x)=(B_+\mathcal Y^0_{\overline ii})( w^{(i)},x)\vartheta_i=(B_-\mathcal Y^0_{\overline ii})( w^{(i)},x)\vartheta_i^{-1}.\label{eq151}
\end{gather}
\end{pp}
\begin{proof}
Using equations \eqref{eq101}, \eqref{eq2}, and that $L_1\Omega=0$, we see that for any $w^{(i)}_1\in W_i,w^{(\overline i)}_2\in W_{\overline i}$,
\begin{align}
&\langle (B_\pm\mathcal Y^0_{\overline ii}) (\vartheta_i^{\pm1} w^{(i)}_1,x)w^{(\overline i)}_2,\Omega \rangle\nonumber\\
=&\langle \mathcal Y^0_{\overline ii}(w^{(\overline i)}_2,e^{\pm i\pi}x)e^{\pm2i\pi L_0}w^{(i)}_1  , \Omega  \rangle\nonumber\\
=&\langle  e^{-x^{-1}L_1}e^{\pm2i\pi L_0}w^{(i)}_1 , (e^{i\pi\mp2i\pi}x^{-2})^{L_0}e^{x^{-1}L_1}w^{(\overline i)}_2  \rangle\nonumber\\
=&\langle  e^{\pm2i\pi L_0}e^{-x^{-1}L_1}w^{(i)}_1 , (e^{i\pi\mp2i\pi}x^{-2})^{L_0}e^{x^{-1}L_1}w^{(\overline i)}_2  \rangle\nonumber\\
=&\langle (e^{i\pi}x^{-2})^{L_0} e^{-x^{-1}L_1}w^{(i)}_1 , e^{x^{-1}L_1}w^{(\overline i)}_2  \rangle\nonumber\\
=&\langle \mathcal Y^0_{ i \bar i}(w^{(i)}_1,x)w^{(\overline i)}_2,\Omega\rangle.
\end{align}
Since $V$ is of CFT type and isomorphic to $V'$ as a $V$-module, $V$ is a simple VOA, i.e., $V$ is an irreducible $V$-module (cf., for example, \cite{CKLW} proposition 4.6-(iv)). Hence $\Omega$ is a cyclic vector in $V$. By \eqref{eq106}, we have $\langle (B_\pm\mathcal Y^0_{\overline ii}) (\vartheta_i^{\pm1} w^{(i)}_1,x)w^{(\overline i)}_2,v \rangle=\langle \mathcal Y^0_{ i \bar i}(w^{(i)}_1,x)w^{(\overline i)}_2,v\rangle$ for any $v\in V$, which proves \eqref{eq152}. \eqref{eq151} can be proved in a similar way.
\end{proof}
When $W_i$ is unitary, we also have
\begin{align}
\mathcal Y^0_{\overline ii}=(\mathcal Y^i_{i0})^\dagger.
\end{align}
Indeed, by \eqref{eq104},  $Y_{\overline i}=\overline {Y_i}$. Hence
\begin{align*}
\mathcal Y^0_{\overline ii}=C^{-1}B_{\pm}\overline {Y_i}=\overline{CB_{\mp}Y_i}=(B_{\mp}Y_i)^\dagger=(\mathcal Y^i_{i0})^\dagger.
\end{align*}

\section{Braiding and fusion of intertwining operators}\label{lb77}
Starting from this chapter, we assume that $V$ satisfies conditions \eqref{eq302}, \eqref{eq303}, and \eqref{eq304}. Recall that, by corollary \ref{lb83}, a unitary VOA automatically satisfies condition \eqref{eq302}.

By \cite{H ODE} theorem 3.5, \emph{the fusion rules of $V$ are finite numbers, and there are only finitely many equivalence classes of irreducible $V$-modules.} Let us choose, for each equivalence class $[W_k]$ of irreducible $V$-module, a representing element $W_k$, and let these modules form a finite set $\{W_k:k\in\mathcal E\}$. (With abuse of notations, we also let $\mathcal E$ denote this finite set.) In other words, $\mathcal E$ is a complete list of mutually inequivalent irreducible $V$-modules. We also require that $V$ is inside $\mathcal E$.  If, moreover, $V$ is unitary, then for any unitarizable $W_k$ ($k\in\mathcal E$), we fix a unitary structure on $W_k$. The unitary structure on $V$ is the standard one. We let $\mathcal E^\uni$ be the set of all unitary $V$-modules in $\mathcal E$.
 
Let $W_i,W_j,W_k$ be $V$-modules. Then $\Theta^k_{ij}$ will always denote (the index set  of) a basis $\{\mathcal Y_\alpha:\alpha\in\Theta^k_{ij} \}$ of the vector space $\mathcal V{k\choose i~j}$. If bases of the vector spaces of intertwining operators are chosen, then for any $W_i,W_k$, we set $\Theta^k_{i*}=\coprod_{j\in\mathcal E}\Theta^k_{ij}$. The notations $\Theta^k_{*j},\Theta^*_{ij}$ are understood in a similar way.

\subsection{Genus $0$ correlation functions}
In this section, we review the construction of genus $0$ correlation functions from intertwining operators. We first give a complex analytic point of view of intertwining operators. Let $\mathcal Y_\alpha\in\mathcal V{k\choose i~j}$. For any $w^{(i)}\in W_i,w^{(j)}\in W_j,w^{(\overline k)}\in W_{\overline k}$,
\begin{align}
\langle \mathcal Y_{\alpha}(w^{(i)},z)w^{(j)},w^{(\overline k)}\rangle=\langle \mathcal Y_{\alpha}(w^{(i)},x)w^{(j)},w^{(\overline k)}\rangle\big|_{x=z}=\sum_{s\in\mathbb R}\langle \mathcal Y_{\alpha}(w^{(i)},s)w^{(j)},w^{(\overline k)}\rangle z^{-s-1}\label{eq234}
\end{align}
is  a finite sum of powers of $z$. (Indeed, if all the vectors are homogeneous then, by \eqref{eq132}, the coefficient before each $z^{-s-1}$ is zero, except when $s=\Delta_{w^{(i)}}+\Delta_{w^{(j)}}-\Delta_{w^{(k)}}-1$.) Since the powers of $z$ are not necessarily integers,  \eqref{eq234} is a \emph{multivalued holomorphic function} defined for $z\in\mathbb C^\times= \mathbb C\setminus\{0\}$: the exact value of $\eqref{eq234}$ depends not only on $z$, but also on $\arg z$. We can also regard $\mathcal Y_\alpha$ as a multivalued $(W_i\otimes W_j\otimes W_{\overline k})^*$-valued holomorphic function on $\mathbb C^\times$. Note that by proposition \ref{lb70}, the transition from the formal series viewpoint to the complex analytic one is faithful.
\begin{cv}
At this point, the notations $\mathcal Y_\alpha(w^{(i)},x)$, $\mathcal Y_\alpha(w^{(i)},z)$, and $\mathcal Y_\alpha(w^{(i)},s)$ seem confusing. We clarify their meanings as follows.

Unless otherwise stated, $\mathcal Y_\alpha(w^{(i)},x)$ is a formal series of the formal variable $x$. If $z\neq0$ is a  complex number, or if $z$ is a complex variable (possibly taking real values), $\mathcal Y_\alpha(w^{(i)},z)$ is defined by \eqref{eq234}.
If $s$ is a real number, $\mathcal Y_\alpha(w^{(i)},s)$ is a mode of $\mathcal Y_\alpha(w^{(i)},x)$, i.e., the  coefficient before $x^{-s-1}$ in $\mathcal Y_\alpha(w^{(i)},x)$.
\end{cv}

Intertwining operators are also called $3$-point (correlation) functions. In \cite{H ODE}, Y. Z. Huang constructed general $n$-point functions by taking the products of intertwining operators. His approach can be sketched as follows: 

For any $n=1,2,3,\dots$, we define the \textbf{configuration space} $\Conf_{n}(\mathbb C^\times)$ to be the complex sub-manifold of $(\mathbb C^\times)^n$ whose points are $(z_1,z_2,\dots,z_n)\in \Conf_{n}(\mathbb C^\times)$ satisfying that $z_m\neq z_l$ whenever $1\leq m<l\leq n$. We let $\widetilde\Conf_{n}(\mathbb C^\times)$ be the universal covering space of $\Conf_{n}(\mathbb C^\times)$.

Let $\mathcal Y_{\alpha_1},\mathcal Y_{\alpha_2},\dots,\mathcal Y_{\alpha_n}$ be intertwining operators $V$. We say that they form a \textbf{chain of intertwining operators}, if for each $1\leq m\leq n-1$, the target space of $\mathcal Y_{\alpha_m}$ equals the source space of $\mathcal Y_{\alpha_{m+1}}$. The following theorem was proved by Huang.

\begin{thm}[cf. \cite{H ODE} theorem 3.5]\label{lb65}
Suppose that $\mathcal Y_{\alpha_1},\dots,\mathcal Y_{\alpha_n}$ form a chain of intertwining operators. For each $1\leq m\leq n$, we let $W_{i_m}$ be the charge space of $\mathcal Y_{\alpha_m}$. We let $W_{i_0}$ be the source space of $\mathcal Y_{\alpha_1}$, and let $W_k$ be the target space  of $\mathcal Y_{\alpha_n}$. Then for any $w^{(i_0)}\in W_{i_0},w^{(i_1)}\in W_{i_1},\dots, w^{(i_n)}\in W_{i_n},w^{(\overline k)}\in W_{\overline k}$, and $z_1,z_2,\dots,z_n\in\mathbb C$ such that $0<|z_1|<|z_2|<\dots<|z_n|$, the expression
\begin{align}
\langle \mathcal Y_{\alpha_n}(w^{(i_n)},z_n)\mathcal Y_{\alpha_{n-1}}(w^{(i_{n-1})},z_{n-1})\cdots\mathcal Y_{\alpha_1}(w^{(i_1)},z_1)w^{(i_0)}, w^{(\overline k)}    \rangle\label{eq235}
\end{align}
\textbf{converges absolutely}, which means that the series
\begin{align}
&\sum_{s_1,s_2,\dots,s_{n-1}\in\mathbb R}\big|\langle \mathcal Y_{\alpha_n}(w^{(i_n)},z_n)P_{s_{n-1}}\mathcal Y_{\alpha_{n-1}}(w^{(i_{n-1})},z_{n-1})P_{s_{n-2}}\nonumber\\
&\qquad\qquad\qquad\cdots P_{s_1}\mathcal Y_{\alpha_1}(w^{(i_1)},z_1)w^{(i_0)}, w^{(\overline k)}    \rangle\big|\label{eq236}
\end{align}
converges, where each $P_{s_m}$ ($1\leq m\leq n-1$) is the  projection of the  target space of $\mathcal Y_{\alpha_m}$ onto its weight-$s_m$ component.
\end{thm}

Note that \eqref{eq235} also \textbf{converges absolutely and locally uniformly}, which means that there exists a neighborhood $U\subset\Conf_n(\mathbb C^\times)$ of $(z_1,z_2,\dots,z_n)$, and a finite number $M>0$, such that  for any $(\zeta_1,\zeta_2,\dots,\zeta_n)\in U$, \eqref{eq236} is bounded by $M$ if we replace each $z_1,z_2,\dots$ with $\zeta_1,\zeta_2,\dots$ in that expression. 

To see this, we assume, without loss of generality, that all the vectors in \eqref{eq235} are homogeneous, and that all the intertwining operators are irreducible. Consider a new set of coordinates $\omega_1,\omega_2,\dots,\omega_n$ such that $z_m=\omega_m\omega_{m+1}\cdots\omega_n$ ($1\leq m\leq n$). Then the condition that $0<|z_1|<|z_2|<\cdots<|z_n|$ is equivalent to that $0<|\omega_1|<1,\dots,0<|\omega_{n-1}|<1,0<|\omega_n|$. By \eqref{eq231}, expression \eqref{eq235} as a formal series also equals
\begin{align}
&\big\langle \mathcal Y_{\alpha_n}\big(w^{(i_n)},\omega_n\big)\mathcal Y_{\alpha_{n-1}}\big(w^{(i_{n-1})},\omega_{n-1}\omega_n\big)\cdots\mathcal Y_{\alpha_1}\big(w^{(i_1)},\omega_1\omega_2\cdots\omega_n\big)w^{(i_0)}, w^{(\overline k)}    \big\rangle\nonumber\\
=&\big\langle \omega_n^{L_0}\mathcal Y_{\alpha_n}\big(\omega_n^{-L_0}w^{(i_n)},1\big)\omega_{n-1}^{L_0}\mathcal Y_{\alpha_{n-1}}\big((\omega_{n-1}\omega_n)^{-L_0}w^{(i_{n-1})},1\big)\cdots\nonumber \\
&~\cdot \omega_1^{L_0}\mathcal Y_{\alpha_1}\big((\omega_1\omega_2\cdots\omega_n)^{-L_0}w^{(i_1)},1\big)(\omega_1\omega_2\cdots\omega_n)^{-L_0}w^{(i_0)}, w^{(\overline k)}    \big\rangle\nonumber\\
=&\big\langle \omega_n^{L_0}\mathcal Y_{\alpha_n}\big(w^{(i_n)},1\big)\omega_{n-1}^{L_0}\mathcal Y_{\alpha_{n-1}}\big(w^{(i_{n-1})},1\big)\cdots \omega_1^{L_0}\mathcal Y_{\alpha_1}\big(w^{(i_1)},1\big)w^{(i_0)}, w^{(\overline k)}    \big\rangle\nonumber\\
&\cdot \prod_{1\leq m\leq n} \omega_m^{-\big(\Delta_{w^{(i_0)}}+\Delta_{w^{(i_1)}}+\cdots+\Delta_{w^{(i_m)}}\big)},\label{eq237}
\end{align} 
where $\mathcal Y_{\alpha_m}(w^{(i_m)},1)=\mathcal Y_{\alpha_m}(w^{(i_m)},x)\big|_{x=1}$ . Since the target space of each $\mathcal Y_{\alpha_m}$ is irreducible, \eqref{eq237} is a \textbf{quasi power series of $\omega_1,\dots,\omega_n$} (i.e., a power series of $\omega_1,\dots,\omega_n$ multiplied by a monomial  $\omega_1^{s_1}\cdots\omega_n^{s_n}$, where $s_1,\dots,s_n\in\mathbb C$), and the convergence of \eqref{eq236} is equivalent to the absolute convergence of the quasi power series \eqref{eq237}.  Therefore, pointwise absolute convergence implies locally uniform absolute convergence.\\

We see that \eqref{eq235} is  a multi-valued holomorphic function defined when $0<|z_1|<\cdots<|z_n|$. We let $\varphi$ be the multi-valued $(W_{i_0}\otimes W_{i_1}\otimes \cdots\otimes W_{i_n}\otimes W_{\overline k})^*$-valued holomorphic function on $\{0<|z_1|<\cdots<|z_n| \}$ defined by \eqref{eq235}. $\varphi$ is called an \textbf{$(n+2)$-point (correlation) function}\footnote{So far our definition of genus $0$ correlation functions is local. We will give a global definition at the end of next section.} of $V$, and is denoted by $ \mathcal Y_{\alpha_n}\mathcal Y_{\alpha_{n-1}}\cdots\mathcal Y_{\alpha_1}$. We define $\mathcal V{k\choose i_n~i_{n-1}~\cdots~ i_0}$ to be the vector space of $(W_{i_0}\otimes W_{i_1}\otimes \cdots\otimes W_{i_n}\otimes W_{\overline k})^*$-valued $n+2$-point functions of $V$. The following proposition can be used to find  a basis of $\mathcal V{k\choose i_n~i_{n-1}~\cdots~ i_0}$.
\begin{pp}\label{lb66}
	Define a linear map $\Phi:$
	\begin{gather*}
	\bigoplus_{j_1,\dots,j_{n-1}\in\mathcal E} \Bigg(\mathcal V{k\choose {i_n}~j_{n-1}}\otimes \mathcal V{j_{n-1}\choose i_{n-1}~j_{n-2}}\otimes \mathcal V{j_{n-2}\choose i_{n-2}~j_{n-3}}\otimes\cdots\otimes \mathcal V{j_1\choose i_1~i_0}\Bigg)\\
	\rightarrow \mathcal V{k\choose i_n~i_{n-1}~\cdots~ i_0},\\
	\mathcal Y_{\alpha_n}\otimes \mathcal Y_{\alpha_{n-1}}\otimes \mathcal Y_{\alpha_{n-2}}\otimes\cdots\otimes\mathcal Y_{\alpha_1}\mapsto  \mathcal Y_{\alpha_n}\mathcal Y_{\alpha_{n-1}} \mathcal Y_{\alpha_{n-2}}\cdots\mathcal Y_{\alpha_1}.
	\end{gather*}
	Then $\Phi$ is an isomorphism.
\end{pp}
Therefore, if elements in $\{\mathcal Y_{\alpha_1} \},\dots,\{\mathcal Y_{\alpha_n} \}$ are linearly independent respectively, then the correlation functions $\{\mathcal Y_{\alpha_n}\mathcal Y_{\alpha_{n-1}}\cdots\mathcal Y_{\alpha_1} \}$ are also linearly independent. The proof of this proposition is postponed to section \ref{lb68}.\\

It was also shown in \cite{H ODE} that correlations functions satisfy a system of linear differential equations, the coefficients of which are holomorphic functions defined on $\Conf_n(\mathbb C^\times)$. More precisely, we have the following:

\begin{thm}[cf. \cite{H ODE}, especially theorem 1.6]\label{lb75}
For any $w^{(i_0)}\in W_{i_0},w^{(i_1)}\in W_{i_1},\dots,w^{(i_n)}\in W_{i_n},w^{(\overline k)}\in W_{\overline k}$, there exist $h_1,\dots,h_n\in\mathbb Z_{\geq0}$, and single-valued holomorphic functions $a_{1,m}(z_1,\dots,z_n),a_{2,m}(z_1,\dots,z_n),\dots,a_{h_m,m}(z_1,\dots,z_n)$  on $\Conf_n(\mathbb C^\times)$, such that for any $(W_{i_0}\otimes W_{i_1}\otimes \cdots\otimes W_{i_n}\otimes W_{\overline k})^*$-valued $(n+2)$-point correlation function $\varphi$ defined  on $\{0<|z_1|<\cdots<|z_n| \}$, the function $\varphi(w^{(i_0)},w^{(i_1)},\dots,w^{(i_n)},w^{(\overline k)};z_1,z_2,\dots,z_n)$ of $(z_1,\dots,z_n)$ satisfies the following system of differential equations:
\begin{align}
\frac{\partial^{h_m}\varphi}{\partial z_m^{h_m}}+a_{1,m}\frac{\partial^{h_m-1}\varphi}{\partial z_m^{h_m-1}}+a_{2,m}\frac{\partial^{h_m-2}\varphi}{\partial z_m^{h_m-2}}+\cdots+a_{h_m,m}\varphi=0\qquad(m=1,\dots,n).\label{eq238}
\end{align}
\end{thm}

Hence, due to elementary ODE theory,  $\varphi$ can be analytically continued to  a multivalued holomorphic function on $\Conf_n(\mathbb C^\times)$ (or equivalently, a single-valued holomorphic function on $\widetilde\Conf_n(\mathbb C^\times)$), which satisfies system \eqref{eq238} globally.

Note that \emph{(global) correlation functions are determined by their values at any fixed point in $\widetilde\Conf_n(\mathbb C^\times)$}. Indeed, since $\varphi$ satisfies \eqref{eq238}, the function $\varphi$ is determined by the values of $\{\frac{\partial^l }{\partial z_m^l}\varphi:1\leq m\leq n, 0\leq l\leq h_m-1 \}$ at any fixed point. On the other hand, by translation property and the locally uniform absolute convergence of \eqref{eq235}, we have
\begin{align}
&\frac \partial {\partial z_m}\varphi(w^{(i_0)},w^{(i_1)},\dots,w^{(i_n)},w^{(\overline k)};z_1,z_2,\dots,z_n)\nonumber\\
=&\varphi(w^{(i_0)},w^{(i_1)},\dots,L_{-1}w^{(i_m)},\dots,w^{(i_n)},w^{(\overline k)};z_1,z_2,\dots,z_n).
\end{align}
Hence $\varphi$ is determined by its value at a point.\\

\subsection{General braiding and fusion relations for intertwining operators}\label{lb72}

The braid and the fusion relations for two intertwining operators were proved by Huang and Lepowsky in \cite{H 1,H 2,H 3,H 4,H ODE}. In this section, we generalize these relations to more than two  intertwining operators.  We also state some useful convergence theorems. The proofs are technical, so we leave them to section \ref{lb71}. 

\subsubsection*{General fusion relations and convergence properties}
\begin{thm}[Fusion of a chain of intertwining operators]\label{lb73}
Let $\mathcal Y_{\sigma_2},\mathcal Y_{\sigma_3},\dots,\mathcal Y_{\sigma_n}$ be a chain of intertwining operators of $V$ with charge spaces $W_{i_2},W_{i_3},\dots,W_{i_n}$ respectively. Let $\mathcal Y_\gamma$ be another intertwining operator of $V$, whose charge space is the same as the target space of $\mathcal Y_{\sigma_n}$. Let $W_{i_0}$ be the source space of $\mathcal Y_\gamma$, $W_{i_1}$ be the source space of $\mathcal Y_{\sigma_2}$, and $W_k$ be the target space of $\mathcal Y_\gamma$. Then for any $w^{(i_0)}\in W_{i_0},w^{(i_1)}\in W_{i_1},\dots,w^{(i_n)}\in W_{i_n}, w^{(\overline k)}\in W_{\overline k}$, and  any $(z_1,z_2,\dots,z_n)\in\Conf_n(\mathbb C^\times)$ satisfying
\begin{align}
0<|z_2-z_1|<|z_3-z_1|<\cdots<|z_n-z_1|<|z_1|,\label{eq246}
\end{align}
 the expression
\begin{align}
&\big\langle \mathcal Y_\gamma\big(\mathcal Y_{\sigma_n}(w^{(i_n)},z_n-z_1)\mathcal Y_{\sigma_{n-1}}(w^{(i_{n-1})},z_{n-1}-z_1)\nonumber\\
&\qquad\cdots\mathcal Y_{\sigma_2}(w^{(i_2)},z_2-z_1)w^{(i_1)},z_1\big)w^{(i_0)},w^{(\overline k)}\big\rangle\label{eq245}
\end{align}
\textbf{converges absolutely and locally uniformly}, which means that there exists a neighborhood $U\subset\Conf_n(\mathbb C^\times)$ of $(z_1,z_2,\dots,z_n)$, and a finite number $M>0$, such that  for any $(z_1,z_2,\dots,z_n)\in U$, 
\begin{align}
&\sum_{s_2,\dots,s_n\in\mathbb R}\big|\langle\mathcal Y_\gamma( P_{s_n}\mathcal Y_{\sigma_n}(w^{(i_n)},\zeta_n-\zeta_1) P_{s_{n-1}}\mathcal Y_{\sigma_{n-1}}(w^{(i_{n-1})},\zeta_{n-1}-\zeta_1)\nonumber\\
&\qquad\qquad\cdots P_{s_2}\mathcal Y_{\sigma_2}(w^{(i_2)},\zeta_2-\zeta_1)w^{(i_1)},\zeta_1)w^{(i_0)},w^{(\overline k)}\rangle\big|<M.
\end{align}

Moreover, if $(z_1,z_2,\dots,z_n)$ satisfies \eqref{eq246} and 
\begin{align}
0<|z_1|<|z_2|<\cdots<|z_n|,
\end{align}
then \eqref{eq245} as a $(W_{i_0}\otimes W_{i_1}\otimes\cdots\otimes W_{i_n}\otimes W_{\overline k})^*$-valued holomorphic function defined  near $(z_1,\dots,z_n)$
is an element in $\mathcal V{k\choose i_n~i_{n-1}~\cdots~ i_0}$, and any element in $\mathcal V{k\choose i_n~i_{n-1}~\cdots~ i_0}$ defined near $(z_1,\dots,z_n)$ can be written as \eqref{eq245}.
\end{thm}

The following convergence theorem for products of generalized intertwining operators is necessary for our theory. (See the discussion in the introduction.)
\begin{thm}\label{lb12}
Let $m$ be a positive integer. For each $a=1,\dots,m$, we choose a positive integer $n_a$. Let $W_{i^1},\dots,W_{i^m}$ be $V$-modules, and let $\mathcal Y_{\alpha^1},\dots,\mathcal Y_{\alpha^m}$ be a chain of intertwining operators with charge spaces $W_{i^1},\dots,W_{i^m}$ respectively. We let $W_i$ be the source space of $\mathcal Y_{\alpha^1}$, and let $W_k$ be the target space of $\mathcal Y_{\alpha^m}$. For each $a=1,\dots,m$ we choose a chain of intertwining operators $\mathcal Y_{\alpha^a_2},\dots,\mathcal Y_{\alpha^a_{n_a}}$ with charge spaces $W_{i^a_2},\dots,W_{i^a_{n_a}}$ respectively. We let $W_{i^a_1}$ be the source space of $\mathcal Y_{\alpha^a_2}$, and assume that the target space of $\mathcal Y_{\alpha^a_{n_a}}$ is $W_{i^a}$.
	
For any $a=1,\dots,m$ and $b=1,\dots,n_a$, we choose a non-zero complex number $z^a_b$. Choose $w^a_b\in W_{i^a_b}$. We also choose vectors $w^i\in W_i,w^{\overline k}\in W_{\overline k}$. Suppose that the complex numbers $\{z^a_b \}$ satisfy the following conditions:\\
(1) For each $a=1,\dots,m$, $0<|z^a_2-z^a_1|<|z^a_3-z^a_1|<\cdots<|z^a_{n_a}-z^a_1|<|z^a_1|$;\\
(2) For each $a=1,\dots,m-1$, $|z^a_1|+|z^a_{n_a}-z^a_1|<|z^{a+1}_1|-|z^{a+1}_{n_{a+1}}-z^{a+1}_1|$,\\
then the expression
	\begin{align}
	&\Big\langle\Big[ \prod_{m\geq a\geq 1}\mathcal Y_{\alpha^a}\Big(\Big(\prod_{n_a\geq b\geq 2}\mathcal Y_{\alpha^a_b}(w^a_b,z^a_b-z^a_1)\Big) w^a_1,z^a_1\Big)\Big]w^i,w^{\overline k} \Big\rangle\nonumber\\
	\equiv&\big\langle \mathcal Y_{\alpha^m}\big(\mathcal Y_{\alpha^m_{n_m}}(w^m_{n_m},z^m_{n_m}-z^m_1)\cdots\mathcal Y_{\alpha^m_2}(w^m_2,z^m_2-z^m_1)w^m_1,z^m_1\big)\nonumber\\
	&\quad\qquad\qquad\qquad\qquad\qquad\qquad\vdots\nonumber\\
	&\cdot\mathcal Y_{\alpha^1}\big(\mathcal Y_{\alpha^1_{n_1}}(w^1_{n_1},z^1_{n_1}-z^1_1)\cdots\mathcal Y_{\alpha^1_2}(w^1_2,z^1_2-z^1_1)w^1_1,z^1_1\big)w^i,w^{\overline k} \big\rangle\label{eq48}
	\end{align}
 \textbf {converges absolutely and locally uniformly}, i.e., there exists $M>0$ and a neighborhood $U^a_b$ of each $z^a_b$, such that for any $\zeta^a_b\in U^a_b$ $(1\leq a\leq m,1\leq b\leq n_a)$ we have:
	\begin{align}
	\sum_{s^a_1,s^a_b\in\mathbb R}\bigg|\Big\langle \Big[\prod_{m\geq a\geq 1}P_{s^a_1}\mathcal Y_{\alpha^a}\Big(\Big(\prod_{n_a\geq b\geq 2}P_{s^a_b}\mathcal Y_{\alpha^a_b}(w^a_b,\zeta^a_b-\zeta^a_1)\Big) w^a_1,\zeta^a_1\Big)\Big]w^i,w^{\overline k} \Big\rangle\bigg|<M.
	\end{align}

Assume, moreover, that $\{z^a_b:1\leq a\leq m,1\leq b\leq n_a \}$ satisfies the following condition:\\
(3) For any $1\leq a,a'\leq m,1\leq b\leq n_a,1\leq b'\leq n_{a'}$, the inequality $0<|z^a_b|<|z^{a'}_{b'}|$ holds when $a<a'$, or $a=a'$ and $b<b'$.\\
Then \eqref{eq48} defined near $\{z^a_b:1\leq a\leq m,1\leq b\leq n_a \}$ is an element in $\mathcal V{k\choose i^m_{n_m}~\cdots~i^m_1~\cdots\cdots~i^1_{n_1}~\cdots~i^1_1~i}$.
	\end{thm}
	
We need another type of convergence property. The notion of absolute and locally uniform convergence is understood as usual.
\begin{co}\label{lb13}
	Let $\mathcal Y_{\sigma_2},\mathcal Y_{\sigma_3},\dots,\mathcal Y_{\sigma_m}$ be a chain of intertwining operators of $V$ with charge spaces $W_{i_2},W_{i_3},\dots,W_{i_m}$ respectively. Let $W_{i_1}$ be the source space of $\mathcal Y_{\sigma_2}$ and $W_i$ be the target space of $\mathcal Y_{\sigma_m}$. Similarly we let $\mathcal Y_{\rho_2},\mathcal Y_{\rho_3},\dots,\mathcal Y_{\rho_m}$ be a chain of intertwining operators of $V$ with charge spaces $W_{j_2},W_{j_3},\dots,W_{j_n}$ respectively. Let $W_{j_1}$ be the source space of $\mathcal Y_{\rho_2}$ and $W_j$ be the target space of $\mathcal Y_{\rho_n}$. Moreover we choose $V$-modules $W_{k_1},W_{k_2},W_{k_3}$, a type $k_1\choose i~j$ intertwining operator $\mathcal Y_{\alpha}$ and a type $k_2\choose k_1~k_0$ intertwining operator $\mathcal Y_{\beta}$. Choose $w^{(i_1)}\in W_{i_1},w^{(i_2)}\in W_{i_2},\dots,w^{(i_m)}\in W_{i_m},w^{(j_1)}\in W_{j_1},w^{(j_2)}\in W_{j_2},\dots,w^{(j_m)}\in W_{j_m},w^{(k_0)}\in W_{i_0},w^{(\overline{k_2})}\in W_{\overline{k_2}}$. Then for any non-zero complex numbers $z_1,z_2,\dots,z_m,\zeta_1,\zeta_2,\dots,\zeta_n$, satisfying $0<|\zeta_2-\zeta_1|<|\zeta_3-\zeta_1|<\cdots<|\zeta_n-\zeta_1|<|z_1-\zeta_1|-|z_m-z_1|$ and $0<|z_2-z_1|<|z_3-z_1|<\cdots<|z_m-z_1|<|z_1-\zeta_1|<|\zeta_1|-|z_m-z_1|$, the expression
	\begin{align}
	&\bigg\langle\mathcal Y_\beta\bigg(\mathcal Y_\alpha\Big(\mathcal Y_{\sigma_m}(w^{(i_m)},z_m-z_1)\cdots\mathcal Y_{\sigma_2}(w^{(i_2)},z_2-z_1)w^{(i_1)},z_1-\zeta_1\Big)\nonumber\\
	&\qquad\cdot\mathcal Y_{\rho_n}(w^{(j_n)},\zeta_n-\zeta_1)\cdots\mathcal Y_{\rho_2}(w^{(j_2)},\zeta_2-\zeta_1)w^{(j_1)},\zeta_1\bigg)w^{(k_0)},w^{(\overline{k_2})}\bigg\rangle
	\end{align}
	exists and converges absolutely and locally uniformly.
\end{co}

\subsubsection*{General braid relations}
Let $z_1,z_2,\dots,z_n$ be distinct complex values in $\mathbb C^\times$. Assume that $0<|z_1|=|z_2|=\dots=|z_n|$, and choose  arguments $\arg z_1,\arg z_2,\dots,\arg z_n$.  We define the expression
\begin{align}
\langle \mathcal Y_{\alpha_n}(w^{(i_n)},z_n)\mathcal Y_{\alpha_{n-1}}(w^{(i_{n-1})},z_{n-1})\cdots\mathcal Y_{\alpha_1}(w^{(i_1)},z_1)w^{(i_0)}, w^{(\overline k)} \rangle\label{eq263}
\end{align}
in the following way: Choose $0<r_1<r_2<\dots<r_n$. Then the expression 
\begin{align}
\langle \mathcal Y_{\alpha_n}(w^{(i_n)},r_nz_n)\mathcal Y_{\alpha_{n-1}}(w^{(i_{n-1})},r_{n-1}z_{n-1})\cdots\mathcal Y_{\alpha_1}(w^{(i_1)},r_1z_1)w^{(i_0)}, w^{(\overline k)} \rangle\label{eq264}
\end{align}
converges absolutely. We define \eqref{eq263} to be the limit of \eqref{eq264} as $r_1,r_2,\dots,r_n\rightarrow 1$. The existence of this limit is guaranteed by theorem \ref{lb75}.

Let $S_n$ be the symmetric group of degree $n$, and choose any $\varsigma\in S_n$. The general braid relations can be stated in the following way:
\begin{thm}[Braiding of intertwining operators]\label{lb78}
Choose distinct $z_1,\dots,z_n\in\mathbb C^\times$ satisfying $0<|z_1|=\cdots=|z_n|$. Let $\mathcal Y_{\alpha_{\varsigma(1)}},\mathcal Y_{\alpha_{\varsigma(2)}},\dots,\mathcal Y_{\alpha_{\varsigma(n)}}$ be a chain of intertwining operators of $V$. For each $1\leq m\leq n$, we let $W_{i_m}$ be the charge space of $\mathcal Y_{\alpha_m}$. Let $W_{i_0}$ be the source space of $\mathcal Y_{\alpha_{\varsigma(1)}}$, and let $W_k$ be the target space of $\mathcal Y_{\alpha_{\varsigma(n)}}$. Then there exists a chain of intertwining operators $\mathcal Y_{\beta_1},\mathcal Y_{\beta_2},\dots,\mathcal Y_{\beta_n}$ with charge spaces $W_{i_1},W_{i_2},\dots,W_{i_n}$ respectively, such that the source space of $\mathcal Y_{\beta_1}$ is $W_{i_0}$, that the target space of $\mathcal Y_{\beta_n}$ is $W_k$, and that for any $w^{(i_0)}\in W_{i_0},w^{(i_1)}\in W_{i_1},\dots w^{(i_n)}\in W_{i_n},w^{(\overline k)}\in W_{\overline k}$, the following braid relation holds:
\begin{align}
&\langle \mathcal Y_{\alpha_{\varsigma(n)}}(w^{(i_{\varsigma(n)})},z_{\varsigma(n)})\cdots\mathcal Y_{\alpha_{\varsigma(1)}}(w^{(i_{\varsigma(1)})},z_{\varsigma(1)})w^{(i_0)},w^{(\overline k)}  \rangle\nonumber\\
=&\langle \mathcal Y_{\beta_n}(w^{(i_n)},z_n)\mathcal Y_{\beta_{n-1}}(w^{(i_{n-1})},z_{n-1})\cdots\mathcal Y_{\beta_1}(w^{(i_1)},z_1)w^{(i_0)},w^{(\overline k)}  \rangle.\label{eq265}
\end{align}
We usually omit the vectors $w^{(i_0)},w^{(\overline k)}$, and write the above equation as
\begin{align}
\mathcal Y_{\alpha_{\varsigma(n)}}(w^{(i_{\varsigma(n)})},z_{\varsigma(n)})\cdots\mathcal Y_{\alpha_{\varsigma(1)}}(w^{(i_{\varsigma(1)})},z_{\varsigma(1)})=\mathcal Y_{\beta_n}(w^{(i_n)},z_n)\cdots\mathcal Y_{\beta_1}(w^{(i_1)},z_1).\label{eqa3}
\end{align}

\end{thm}

When $n=2$, the proof of   braid relations is based on the following well-known property.  For the reader's convenience, we include a proof in section \ref{lb71}.
\begin{pp}\label{lb85}
	Let $\mathcal Y_\gamma,\mathcal Y_\delta$ be intertwining operators of $V$, and assume $\mathcal Y_\gamma\in\mathcal V{k\choose i~j}$. Choose $z_i,z_j\in\mathbb C^\times$ satisfying $0<|z_j-z_i|<|z_i|,|z_j|$. Choose $\arg(z_j-z_i)$,  and let $\arg z_j$ be close to $\arg z_i$ as $z_j\rightarrow z_i$. Then for any $w^{(i)}\in W_i,w^{(j)}\in W_j$,
	\begin{align}
	\mathcal Y_\delta\big(\mathcal Y_{B_\pm\gamma}(w^{(j)},z_j-z_i)w^{(i)},z_i \big)=\mathcal Y_\delta\big(\mathcal Y_\gamma(w^{(i)},e^{\pm i\pi}(z_j-z_i))w^{(j)},z_j \big).\label{eq275}
	\end{align}
\end{pp}

\begin{rem}
The braid relation \eqref{eqa3} is unchanged if we scale the norm of the complex variables $z_1,z_2,\dots,z_n$, or rotate each variable without meeting the others, and change its $\arg$ value continuously. The braid relation might change, however, if $z_1,z_2,\dots,z_n$ are fixed, but their arguments are changed by $2\pi$ multiplied by some integers. 
\end{rem}

The proof of theorem \ref{lb78} (see section \ref{lb71})  implies the following:
\begin{pp}\label{lb79}
	Let $\mathcal Y_{\gamma_1},\dots,\mathcal Y_{\gamma_m},\mathcal Y_{\alpha_{\varsigma(1)}},\dots,\mathcal Y_{\alpha_{\varsigma(n)}},\mathcal Y_{\delta_1},\dots,\mathcal Y_{\delta_l}$ be a chain of intertwining operator of $V$ with charge spaces $W_{i'_1},\dots,W_{i'_m},W_{i_{\varsigma(1)}},\dots,W_{i_{\varsigma(n)}},W_{i''_1},\dots,W_{i''_l}$ respectively. Let $W_{j_1}$ be the source space of $\mathcal Y_{\gamma_1}$ and $W_{j_2}$ be the target space of $\mathcal Y_{\delta_l}$. Let $z_1,\dots,z_n,z'_1,\dots,z'_m,z''_1,\dots,z''_l$ be distinct  complex numbers in $S^1$ with fixed arguments. Choose vectors $w^{(j_1)}\in W_{j_1},w^{(i'_1)}\in W_{i'_1},\dots,w^{(i'_m)}\in W_{i'_m},w^{(i''_1)}\in W_{i''_1},\dots,w^{(i''_l)}\in W_{i''_l},w^{(\overline{j_2})}\in W_{\overline{j_2}}$. Let
	\begin{gather*}
	\mathcal X_1=\mathcal Y_{\gamma_m}(w^{(i'_m)},z'_m)\cdots\mathcal Y_{\gamma_1}(w^{(i'_1)},z'_1),\\
	\mathcal X_2=\mathcal Y_{\delta_l}(w^{(i''_l)},z''_l)\cdots\mathcal Y_{\delta_1}(w^{(i''_1)},z''_1).
	\end{gather*}
	Suppose that the braid relation \eqref{eq265} holds for all $w^{(i_0)}\in W_{i_0},w^{(i_1)}\in W_{i_1},\dots,w^{(i_n)}\in W_{i_n},w^{(\overline k)}\in W_{\overline k}$. Then we also have the braid relation
	\begin{align}
	&\langle\mathcal X_2\mathcal Y_{\alpha_{\varsigma(n)}}(w^{(i_{\varsigma(n)})},z_{\varsigma(n)})\cdots\mathcal Y_{\alpha_{\varsigma(1)}}(w^{(i_{\varsigma(1)})},z_{\varsigma(1)})\mathcal X_1 w^{(j_1)},w^{(\overline{j_2})} \rangle\nonumber\\
	=&\langle\mathcal X_2\mathcal Y_{\beta_n}(w^{(i_n)},z_n)\cdots\mathcal Y_{\beta_1}(w^{(i_1)},z_1)\mathcal X_1 w^{(j_1)},w^{(\overline{j_2})} \rangle.
	\end{align}
\end{pp}

The braiding operators $B_\pm$ and the braid relations of intertwining operators are related in the following way:
\begin{pp}\label{lb86}
	Let $z_i,z_j\in S^1$ and $\arg z_j<\arg z_i<\arg z_j+\pi/3$. Let $\arg(z_i-z_j)$ be close to $\arg z_i$ as $z_j\rightarrow 0$, and let $\arg(z_j-z_i)$ be close to $\arg z_j$ as $z_i\rightarrow 0$.
	
	Let $\mathcal Y_\beta,\mathcal Y_\alpha$ be a chain of intertwining operators with charge spaces $W_j,W_i$ respectively, and let $\mathcal Y_{\alpha'},\mathcal Y_{\beta'}$ be a chain of intertwining operators with charge spaces $W_i,W_j$ respectively. Assume that the source spaces of $\mathcal Y_\beta$ and $\mathcal Y_{\alpha'}$ are $W_{k_1}$, and that the target spaces of $\mathcal Y_\alpha$ and $\mathcal Y_{\beta'}$ are $W_{k_2}$.

If there exist a $V$-module $W_k$, and $\mathcal Y_\gamma\in{k\choose i~j},\mathcal Y_\delta\in{k_2\choose k~k_1}$, such that for any $w^{(i)}\in W_i,w^{(j)}\in W_j$, we have the fusion relations:
\begin{gather}
\mathcal Y_\alpha(w^{(i)},z_i)\mathcal Y_\beta(w^{(j)},z_j)=\mathcal Y_\delta(\mathcal Y_\gamma(w^{(i)},z_i-z_j)w^{(j)},z_j),\\
\mathcal Y_{\beta'}(w^{(j)},z_j)\mathcal Y_{\alpha'}(w^{(i)},z_i)=\mathcal Y_\delta(\mathcal Y_{B_+\gamma}(w^{(j)},z_j-z_i)w^{(i)},z_i).
\end{gather}
Then the following braid relation holds:
	\begin{align}
	\mathcal Y_{\alpha}(w^{(i)},z_i)\mathcal Y_\beta(w^{(j)},z_j)=\mathcal Y_{\beta'}(w^{(j)},z_j)\mathcal Y_{\alpha'}(w^{(i)},z_i).\label{eq278}
	\end{align}
\end{pp}
\begin{proof}
	Clearly we have $\arg(z_i-z_j)=\arg(z_j-z_i)+\pi$. So equation \eqref{eq278} follows directly from proposition \ref{lb85}.
\end{proof}

Using braid relations, we can give a global description of correlation functions. Consider the covering map $\pi_n:\widetilde\Conf_n(\mathbb C^\times)\rightarrow \Conf_n(\mathbb C^\times)$. Choose $\varsigma\in S_n$,  let $U_\varsigma=\{(z_1,\dots,z_n):0<|z_{\varsigma(1)}|<|z_{\varsigma(2)}|<\cdots<|z_{\varsigma(n)}|\}$, and choose a connected component $\widetilde U_\varsigma$ of $\pi_n^{-1}(U_\varsigma)$. Then a $(W_{i_0}\otimes W_{i_{\varsigma(1)}}\otimes \cdots\otimes W_{i_{\varsigma(n)}}\otimes W_{\overline k})^*$-valued  correlation function  defined when $(z_{\varsigma(1)},\dots,z_{\varsigma(n)})\in U_\varsigma$ by the left hand side of equation \eqref{eq265} can be lifted through $\pi_n:\widetilde U_\varsigma\rightarrow U_\varsigma$ and analytically continued to a (single-valued) holomorphic function $\varphi$ on $\widetilde\Conf_n(\mathbb C^\times)$. We  define the vector   space $\mathcal V{k\choose i_n~i_{n-1}~\cdots~ i_0}$ of \textbf{$(W_{i_0}\otimes W_{i_1}\otimes \cdots\otimes W_{i_n}\otimes W_{\overline k})^*$-valued (genus $0$) correlation function} to be the vector space of holomorphic functions on $\widetilde\Conf_n(\mathbb C^\times)$ of the form $\varphi$. This definition  does not depend on the choice of $\varsigma$ and $\widetilde U_\varsigma$: If $\varsigma'\in S_n$ and $\widetilde U'_{\varsigma'}$ is a connected component of $\pi_n^{-1}(U_{\varsigma'})$, then by theorem \ref{lb78}, for any  $\varphi\in\mathcal V{k\choose i_n~i_{n-1}~\cdots~ i_0}$ defined on $\widetilde\Conf_n(\mathbb C^\times)$,  it is not hard to find a $(W_{i_0}\otimes W_{i_{\varsigma'(1)}}\otimes \cdots\otimes W_{i_{\varsigma'(n)}}\otimes W_{\overline k})^*$-valued  correlation function  defined when $(z_{\varsigma'(1)},\dots,z_{\varsigma'(n)})\in U_{\varsigma'}$ which can  be lifted through $\pi_n:\widetilde U'_{\varsigma'}\rightarrow U_{\varsigma'}$ and analytically continued to the function $\varphi$.

\subsection[Braiding and fusion with  $Y_i$ and  $\mathcal Y^i_{i0}$]{Braiding and fusion with vertex operators and creation operators}

In this section, we prove some useful braid and fusion relations. These relations are not only important for constructing a braided tensor category of representations of $V$, but also necessary for studying generalized intertwining operators.
\subsubsection*{Braiding and fusion with vertex operators}

The  Jacobi identity \eqref{eq128} can be interpreted in terms of braid and fusion relations:

\begin{pp}\label{lb80}
	Let $\mathcal Y_\alpha$ be a type $k\choose i~j$ intertwining operator of $V$.  Choose $z,\zeta\in \mathbb C^\times$ satisfying $0<|z-\zeta|<|z|=|\zeta|$. Choose an argument $\arg z$. Then for any $u\in V, w^{(i)}\in W_i$, we have
\begin{align}
Y_k(u,\zeta)\mathcal Y_\alpha(w^{(i)},z)=\mathcal Y_\alpha(w^{(i)},z)Y_j(u,\zeta)=\mathcal Y_\alpha\big(Y_i(u,\zeta-z)w^{(i)},z\big).
\end{align}
\end{pp}
\begin{proof}
	The above braid and fusion relations are equivalent to the following statement: for any $w^{(j)}\in W_j,w^{(\overline k)}\in W_{\overline k}$, and for any $z\in\mathbb C^\times$, the functions of $\zeta$:
	\begin{gather}
	\langle \mathcal Y_\alpha(w^{(i)},z)Y_j(u,\zeta) w^{(j)},w^{(\overline k)} \rangle,\\
	\langle \mathcal Y_\alpha\big(Y_i(u,\zeta-z)w^{(i)},z\big) w^{(j)},w^{(\overline k)} \rangle,\\
	\langle Y_k(u,\zeta)\mathcal Y_\alpha(w^{(i)},z) w^{(j)},w^{(\overline k)} \rangle
	\end{gather}
	defined respectively near $0$, near $z$, and near $\infty$ can be analytically continued to the same (single-valued) meromorphic function on $\mathbb P^1$ whose  poles are inside $\{0,z,\infty\}$. This is equivalent to  that for any $f(\zeta,z)\in\mathbb C[\zeta^{\pm1},(\zeta-z)^{-1}]$,
	\begin{align}
	&\Res_{\zeta=0}\big( \langle \mathcal Y_\alpha(w^{(i)},z)Y_j(u,\zeta) w^{(j)},w^{(\overline k)} \rangle\cdot f(\zeta,z)d\zeta\big)\nonumber\\
	+&\Res_{\zeta=z}\big( \langle \mathcal Y_\alpha\big(Y_i(u,\zeta-z)w^{(i)},z\big) w^{(j)},w^{(\overline k)} \rangle\cdot f(\zeta,z)d\zeta\big)\nonumber\\
	+&\Res_{\zeta=\infty}\big(\langle Y_k(u,\zeta)\mathcal Y_\alpha(w^{(i)},z) w^{(j)},w^{(\overline k)} \rangle\cdot f(\zeta,z)d\zeta\big)=0.\label{eq266}
	\end{align}\footnote{Here we use the following well-known Mittag-Leffler type theorem: Let $M$ be a compact Riemann surface, $p_1,\dots,p_n$ are distinct points on $M$, $\zeta_1,\dots,\zeta_n$ are local coordinates at $p_1,\dots,p_n$ respectively. For any $j=1,\dots,n$ we associate a locally defined formal meromorphic function $\varphi_j\in\mathbb C((\zeta_j))$ near $p_j$. Then the following two conditions are equivalent. (a) There exists a global meromorphic function $\varphi$ on $M$ with no poles outside $p_1,\dots,p_n$, such that the  series expansion of $\varphi$ near each $p_j$ is $\varphi_j$. (b) For any global meromorphic $1$-form $\omega$ on $M$ with no poles outside $p_1,\dots,p_n$, the relation $\sum_j \Res_{p_j}\varphi_i\omega=0$ holds. (a)$\Rightarrow$(b) is obvious from residue theorem. (b)$\Rightarrow$(a) is not hard to prove using Serre duality. See \cite{Ueno} theorem 1.22, or \cite{Muk} theorem 1.}
Note that $\mathbb C[\zeta^{\pm1},(\zeta-z)^{-1}]$ is spanned by functions of the form $\zeta^m(\zeta-z)^n$, where $m,n\in\mathbb Z$. If we let $f(\zeta,z)=\zeta^m(\zeta-z)^n$, then one can compute that the left hand side of \eqref{eq266} becomes $\sum_{s\in\mathbb R}c_sz^{-s-1}$, where each coefficient $c_s$ is the difference between the left hand side and the right hand side of the Jacobi identity \eqref{eq128}. This shows that \eqref{eq266} is equivalent to the Jacobi identity \eqref{eq128}. 
\end{proof}

The above intertwining property can be generalized to any correlation function.

\begin{pp}\label{lb81} \footnote{One can use proposition \ref{lb81} and  the translation property to define correlation functions (parallel sections of conformal blocks). cf. \cite{Conformal blocks} chapter 10. }
Let $z_0=0$, choose $(z_1,z_2,\dots,z_n)\in\Conf_n(\mathbb C^\times)$, and choose a correlation function $\varphi\in\mathcal V{k\choose i_n~i_{n-1}~\cdots~i_1~i_0}$ defined near $(z_1,z_2,\dots,z_n)$. Then for any $u\in V,w^{(i_0)}\in W_{i_0},w^{(i_1)}\in W_{i_1},\dots,w^{(i_n)}\in W_{i_n},w^{(\overline k)}\in W_{\overline k}$, and any $0\leq m\leq n$, the following formal series in $\mathbb C((\zeta-z_m))$:
	\begin{align}
	&\psi_{i_m}(\zeta,z_1,z_2,\dots,z_n)\nonumber\\
	=&\varphi(w^{(i_0)},\dots,w^{(i_{m-1})},Y_{i_m}(u,\zeta-z_m)w^{(i_m)},w^{(i_{m+1})},\dots,w^{(i_n)},w^{(\overline k)};z_1,z_2,\dots,z_n),
	\end{align}
	and the following formal series in $\mathbb C((\zeta^{-1}))$:
	\begin{align}
	&\psi_k(\zeta,z_1,z_2,\dots,z_n)\nonumber\\
	=&\varphi(w^{(i_0)},w^{(i_1)},w^{(i_2)},\dots,w^{(i_n)},Y_k(u,\zeta)^\tr w^{(\overline k)};z_1,z_2,\dots,z_n)
	\end{align}
	are expansions of the same (single-valued) holomorphic function on $\mathbb P\setminus\{z_0,z_1,z_2,\dots,z_n,\infty \}$ near the poles $\zeta=z_m$ ($0\leq m\leq n$) and $\zeta=\infty$ respectively.
\end{pp}

\begin{proof}
	When $0<|z_1|<|z_2|<\cdots<|z_n|$, we can prove this property easily using proposition \ref{lb79}, proposition \ref{lb80}, and theorem $\ref{lb12}$. Note that this property is equivalent to  that for any $f(\zeta,z_1,\dots,z_n)\in\mathbb C[\zeta^{\pm1},(\zeta-z_1)^{-1},\dots,(\zeta-z_n)^{-1}]$,
	\begin{align}
	&\sum_{0\leq m\leq n}\Res_{\zeta=z_m}\big(\psi_{i_m}(\zeta,z_1,\dots,z_n)f(\zeta,z_1,\dots,z_n)d\zeta\big)\nonumber\\
	=&-\Res_{\zeta=\infty}\big(\psi_k(\zeta,z_1,\dots,z_n)f(\zeta,z_1,\dots,z_n)d\zeta\big).\label{eq267}
	\end{align}
 If we write down the above equations explicitly, we will find that condition \eqref{eq267} is equivalent to a set of linear equations of $\varphi$, the coefficients of which are $\End\big((W_{i_0}\otimes W_{i_1}\otimes\cdots\otimes W_{i_n}\otimes W_{\overline k})^*\big)$-valued single-valued holomorphic functions on $\Conf_n(\mathbb C^\times)$. Since $\varphi$ satisfies these equations locally, it must satisfy them globally. Therefore $\varphi$ satisfies the desired property at any point in $\Conf_n(\mathbb C^\times)$.
\end{proof}

As an application of this intertwining property, we prove a very useful uniqueness property for correlation functions.

\begin{co}\label{lb84}
Fix $(z_1,z_2,\dots,z_n)\in\Conf_n(\mathbb C^\times)$.	Let $\varphi\in\mathcal V{\overline{i_{n+1}}\choose i_n~i_{n-1}~\cdots~i_1~i_0}$ be a correlation function defined near $(z_1,z_2,\dots,z_n)$. Choose $l\in\{0,1,2,\dots,n+1 \}$.	For any $m\in\{0,1,2,\dots,n+1 \}$ such that $m\neq l$, we assume that $W_{i_m}$ is irreducible, and choose a nonzero vector $w^{(i_m)}_0\in W_{i_m}$. Suppose that for any $w^{(i_l)}\in W_{i_l}$,
	\begin{align}
	\varphi(w^{(i_0)}_0,\dots,w^{(i_{l-1})}_0,w^{(i_l)},w^{(i_{l+1})}_0,\dots,w^{(i_{n+1})}_0;z_1,z_2,\dots,z_n)=0,\label{eq268}
	\end{align}
	then $\varphi=0$.
\end{co}

\begin{proof}
	We assume that $l\leq n$. The case that $l=n+1$ can be proved in a similar way. Suppose that \eqref{eq268} holds. Then for any $u\in V$, the formal series in $\mathbb C((\zeta-z_l))$:
	\begin{align}
	\varphi(w^{(i_0)}_0,\dots,w^{(i_{l-1})}_0,Y_{i_l}(u,\zeta-z_l)w^{(i_l)},w^{(i_{l+1})}_0,\dots,w^{(i_{n+1})}_0;z_1,z_2,\dots,z_n)\label{eq269}
	\end{align}
	equals zero. By proposition \ref{lb81}, \eqref{eq269} is the expansion of a global holomorphic function (which must be zero) on $\mathbb P\setminus\{z_0,z_1,\dots,z_n,\infty\}$, and when $\zeta$ is near $z_0=0$, this function becomes
	\begin{align}
	\varphi(Y_{i_0}(u,\zeta)w^{(i_0)}_0,w^{(i_1)}_0\dots,w^{(i_{l-1})}_0,w^{(i_l)},w^{(i_{l+1})}_0,\dots,w^{(i_{n+1})}_0;z_1,z_2,\dots,z_n),
	\end{align}
	which is zero. Therefore, for each mode $Y_{i_1}(u,s)$ ($s\in\mathbb Z$), we have
	\begin{align}
	\varphi(Y_{i_0}(u,s)w^{(i_0)}_0,w^{(i_1)}_0\dots,w^{(i_{l-1})}_0,w^{(i_l)},w^{(i_{l+1})}_0,\dots,w^{(i_{n+1})}_0;z_1,z_2,\dots,z_n)=0.
	\end{align}
	Since $W_{i_0}$ is irreducible, for any $w^{(i_0)}\in W_0$ we have
	\begin{align}
	\varphi(w^{(i_0)},w^{(i_1)}_0\dots,w^{(i_{l-1})}_0,w^{(i_l)},w^{(i_{l+1})}_0,\dots,w^{(i_{n+1})}_0;z_1,z_2,\dots,z_n)=0.
	\end{align}
	If we repeat this argument several times,  we see that for any $w^{(i_0)}\in W_{i_0},w^{(i_1)}\in W_{i_1},\dots,w^{(i_{n+1})}\in W_{i_{n+1}}$,
	\begin{align}
	\varphi(w^{(i_0)},w^{(i_1)},\dots,w^{(i_{n+1})};z_1,\dots,z_n)=0.
	\end{align}
	Hence $\varphi$ equals zero at $(z_1,\dots,z_n)$. By theorem \ref{lb75} and the translation property, the value of $\varphi$ equals zero at any point.
\end{proof}

\subsubsection*{Braiding and fusion with creation operators}

\begin{lm}\label{lb1}

	Let $\mathcal Y_\alpha$ be a type $k\choose i~j$ intertwining operator. Then for any $w^{(i)}\in W_i,w^{(j)}\in W_j,w^{(\overline k)}\in W_{\overline k}$, $z\in \mathbb C^\times$ and $z_0\in\mathbb C$:\\
	(1) If $0\leq|z_0|<|z|$, and $\arg (z-z_0)$ is close to $\arg z$ as $z_0\rightarrow 0$, then
	\begin{align}
	\sum_{s\in\mathbb R}\langle w^{(\overline k)},\mathcal Y_\alpha(w^{(i)},z)P_se^{z_0L_{-1}}w^{(j)}\rangle
	\end{align}
	converges absolutely and equals 
	\begin{align}
	\langle w^{(\overline k)},e^{z_0L_{-1}}\mathcal Y_\alpha(w^{(i)},z-z_0)w^{(j)}\rangle.
	\end{align}
	We simply write
	\begin{equation}\label{eq10}
	e^{z_0L_{-1}}\mathcal Y_\alpha(w^{(i)},z-z_0)=\mathcal Y_\alpha(w^{(i)},z)e^{z_0L_{-1}}.
	\end{equation}
	(2) If $0\leq|z_0|<|z|^{-1}$ and $\arg (1-zz_0)$ is close to $\arg 1=0$ as $z_0\rightarrow 0$, then 
	\begin{align}
	\sum_{s\in\mathbb R}\langle w^{(\overline k)},e^{z_0L_1}P_s\mathcal Y_\alpha(w^{(i)},z)w^{(j)}\rangle
	\end{align}
	converges absolutely and equals
	\begin{align}
	\big\langle w^{(\overline k)},\mathcal Y_\alpha\big(e^{z_0(1-zz_0)L_1}(1-zz_0)^{-2L_0}w^{(i)},z/(1-zz_0)\big)e^{z_0L_1}w^{(j)}\big\rangle.\label{eq274}
	\end{align}
	We simply write
	\begin{align}
	e^{z_0L_1}\mathcal Y_\alpha(w^{(i)},z)=\mathcal Y_\alpha\big(e^{z_0(1-zz_0)L_1}(1-zz_0)^{-2L_0}w^{(i)},z/(1-zz_0)\big)e^{z_0L_1}.\label{eq11}
	\end{align}
\end{lm}

\begin{proof}
	Assume without loss of generality that all the vectors are homogeneous.	
	
	(1)
	Let $x,x_0,x_1$ be commuting independent formal variables.  Note first of all that \eqref{eq10} holds in the formal sense:
	\begin{align}
	\langle w^{(\overline k)},e^{x_0L_{-1}}\mathcal Y_\alpha(w^{(i)},x-x_0)w^{(j)}\rangle=\langle w^{(\overline k)},\mathcal Y_\alpha(w^{(i)},x)e^{x_0L_{-1}}w^{(j)}\rangle,\label{eq21}
	\end{align}
	where $$\mathcal Y_\alpha(w^{(i)},x-x_0)=\sum_{s\in\mathbb R}\sum_{r\in\mathbb Z_{\geq0}}\mathcal Y_\alpha(w^{(i)},s){-s-1\choose r}x^{-s-1-r}(-x_0)^r.$$ Equation \eqref{eq21} can be proved using the relation $[L_{-1},\mathcal Y_\alpha(w^{(i)},x)]=\frac d{dx}\mathcal Y_\alpha(w^{(i)},x)$. (See  \cite{FHL} section 5.4 for more details.) Write
	\begin{align}
	\langle w^{(\overline k)},e^{x_0L_{-1}}\mathcal Y_\alpha(w^{(i)},x_1)w^{(j)}\rangle=\sum_{m\in\mathbb Z_{\geq0}}c_mx_0^mx_1^{d-m}
	\end{align}
	where $d\in\mathbb R$ and $c_m\in\mathbb C$. Clearly $c_m=0$ for all but finitely many $m$.	Then the left hand side of \eqref{eq21} equals
	\begin{align*}
	\sum_{m,l\in\mathbb Z_{\geq0}}c_mx_0^m\cdot{d-m\choose l}x^{d-m-l}(-x_0)^l.
	\end{align*}

	We now substitute $z$ and $z_0$ for $x$ and $x_0$ in equation \eqref{eq21}. For any $z_0$ satisfying $0\leq|z_0|<|z|$,  let $\arg(z-z_0)$ be close to $\arg z$ as $z_0\rightarrow 0$. Then
	\begin{align}
	&\langle w^{(\overline k)},e^{z_0L_{-1}}\mathcal Y_\alpha(w^{(i)},z-z_0)w^{(j)}\rangle\nonumber\\
	=&\langle w^{(\overline k)},e^{x_0L_{-1}}\mathcal Y_\alpha(w^{(i)},x_1)w^{(j)}\rangle\big|_{x_0=z_0,x_1=z-z_0}\nonumber\\
	=&\sum_{m\in\mathbb Z_{\geq0}}c_mz_0^m(z-z_0)^{d-m}\nonumber\\
	=&\sum_{m,l\in\mathbb Z_{\geq0}}c_mz_0^m\cdot{d-m\choose l}z^{d-m-l}(-z_0)^l,\label{eq271}
	\end{align}
	which converges absolutely and equals
	\begin{align}
	&\langle w^{(\overline k)},\mathcal Y_\alpha(w^{(i)},x)e^{x_0L_{-1}}w^{(j)}\rangle\big|_{x=z,x_0=z_0}\nonumber\\
	=&\langle w^{(\overline k)},\mathcal Y_\alpha(w^{(i)},z)e^{z_0L_{-1}}w^{(j)}\rangle.\label{eq22}
	\end{align}
	This proves part (1).\\
	
	(2) Since $\alpha=C^{-1}C\alpha$, we have
	\begin{align*}
	&\sum_{s\in\mathbb R}\langle w^{(\overline k)},e^{z_0L_1}P_s\mathcal Y_\alpha(w^{(i)},z)w^{(j)}\rangle\\
=&\sum_{s\in\mathbb R}\langle P_se^{z_0L_1}w^{(\overline k)},\mathcal Y_{C^{-1}C\alpha}(w^{(i)},z)w^{(j)}\rangle\\
	=&\sum_{s\in\mathbb R}\big\langle\mathcal Y_{C\alpha}\big(e^{zL_1}(e^{i\pi}z^{-2})^{L_0}w^{(i)},z^{-1}\big) P_se^{z_0L_{-1}} w^{(\overline k)},w^{(j)}\big\rangle,
	\end{align*}
	which, according to part (1), converges absolutely and equals
	\begin{gather}
	\big\langle e^{z_0L_{-1}}\mathcal Y_{C\alpha}\big(e^{zL_1}(e^{i\pi}z^{-2})^{L_0}w^{(i)},z^{-1}-z_0\big)  w^{(\overline k)},w^{(j)}\big\rangle,\label{eq272}
	\end{gather}
	where $\arg(z^{-1}-z_0)$ is close to $\arg(z^{-1})=-\arg z$ as $z_0\rightarrow 0$.	This is equivalent to saying that $\arg(1-zz_0)$ is close to $0$ as $z_0\rightarrow 0$.
	
	By the definition of $C\alpha$, \eqref{eq272} equals
	\begin{align}
	&\big\langle \mathcal Y_{C\alpha}\big(e^{zL_1}(e^{i\pi}z^{-2})^{L_0}w^{(i)},z^{-1}-z_0\big)  w^{(\overline k)},e^{z_0L_1}w^{(j)}\big\rangle\nonumber\\
	=	&\big\langle   w^{(\overline k)},\mathcal Y_\alpha\big(e^{(z^{-1}-z_0)L_1}(e^{-i\pi}(z^{-1}-z_0)^{-2})^{L_0}\nonumber\\
	&\cdot e^{zL_1}(e^{i\pi}z^{-2})^{L_0}w^{(i)},(z^{-1}-z_0)^{-1}\big)e^{z_0L_{1}}w^{(j)}\big\rangle.\label{eq273}
	\end{align}
	Note that \eqref{eq101} also holds when	$x\in\mathbb C,x_0\in\mathbb C^\times$. Therefore, by applying relation \eqref{eq101}, expression \eqref{eq273} equals \eqref{eq274}. This finishes the proof of part (2).
\end{proof}

\begin{pp}\label{lb10}
	Let $z_1,\dots,z_n\in\mathbb C^\times$ satisfy $|z_1|<|z_2|<\dots<|z_n|$ and $|z_2-z_1|<\dots<|z_n-z_1|<|z_1|$. Choose arguments $\arg z_1,\arg z_2,\dots,\arg z_n$. For each $2\leq m\leq n$, we let $\arg (z_m-z_1)$ be close to $\arg z_m$ as $z_1\rightarrow 0$.
	Let $\mathcal Y_{\sigma_2},\dots,\mathcal Y_{\sigma_n}$ be a chain of intertwining operators of $V$ with charge spaces $W_{i_2},\dots,W_{i_n}$ respectively. Let $W_{i_1}$ be the source space of $\mathcal Y_{\sigma_2}$, and let $W_i$ be the target space of $\mathcal Y_{\sigma_n}$. Then for any $w^{(i_1)}\in W_{i_1},w^{(i_2)}\in W_{i_2},\dots,w^{(i_n)}\in W_{i_n}$, we have the  fusion relation
	\begin{align}
	&\mathcal Y^i_{i0}\big(\mathcal Y_{\sigma_n}(w^{(i_n)},z_n-z_1)\cdots\mathcal Y_{\sigma_2}(w^{(i_2)},z_2-z_1)w^{(i_1)},z_1\big)\nonumber\\
	=&\mathcal Y_{\sigma_n}(w^{(i_n)},z_n)\cdots\mathcal Y_{\sigma_2}(w^{(i_2)},z_2)\mathcal Y^{i_1}_{i_10}(w^{(i_1)},z_1).\label{eq37}
	\end{align}
\end{pp}
\begin{proof}
	We  assume that $z_1,z_2,\dots,z_n$ are on the same ray emitting from the origin (e.g. on $\mathbb R_{>0}$). (We don't assume, however, that these complex values have the same argument.) Then for each $2\leq m\leq n$, these complex numbers satisfy
	\begin{align}
	|z_1|+|z_m-z_1|<|z_{m+1}|.\label{eq280}
	\end{align}
	If \eqref{eq37} is proved at these points, then by theorem \ref{lb75} and analytic continuation, \eqref{eq37} holds in general.
	
	Choose any $w^{(\overline i)}\in W_{\overline i}$. Using equations \eqref{eq233} and \eqref{eq10} several times, we have
	\begin{align}
	&\langle\mathcal Y_{\sigma_n}(w^{(i_n)},z_n)\cdots\mathcal Y_{\sigma_3}(w^{(i_3)},z_3)\mathcal Y_{\sigma_2}(w^{(i_2)},z_2)\mathcal Y^{i_1}_{i_10}(w^{(i_1)},z_1)\Omega,w^{(\overline i)}\rangle\nonumber\\
	=&\langle \mathcal Y_{\sigma_n}(w^{(i_n)},z_n)\cdots\mathcal Y_{\sigma_3}(w^{(i_3)},z_3)\mathcal Y_{\sigma_2}(w^{(i_2)},z_2)e^{z_1L_{-1}}w^{(i_1)},w^{(\overline i)}\rangle\nonumber\\
	=&\langle \mathcal Y_{\sigma_n}(w^{(i_n)},z_n)\cdots\mathcal Y_{\sigma_3}(w^{(i_3)},z_3)e^{z_1L_{-1}}\mathcal Y_{\sigma_2}(w^{(i_2)},z_2-z_1)w^{(i_1)},w^{(\overline i)}\rangle\nonumber\\
	=&\langle \mathcal Y_{\sigma_n}(w^{(i_n)},z_n)\cdots e^{z_1L_{-1}}\mathcal Y_{\sigma_3}(w^{(i_3)},z_3-z_1)\mathcal Y_{\sigma_2}(w^{(i_2)},z_2-z_1)w^{(i_1)},w^{(\overline i)}\rangle\nonumber\\
	&\qquad\qquad\qquad\qquad\qquad\qquad\qquad\vdots\nonumber\\
	=&\langle e^{z_1L_{-1}}\mathcal Y_{\sigma_n}(w^{(i_n)},z_n-z_1)\cdots \mathcal Y_{\sigma_3}(w^{(i_3)},z_3-z_1)\mathcal Y_{\sigma_2}(w^{(i_2)},z_2-z_1)w^{(i_1)},w^{(\overline i)}\rangle\nonumber\\
	=&\big\langle \mathcal Y^i_{i0}\big(\mathcal Y_{\sigma_n}(w^{(i_n)},z_n-z_1)\cdots \mathcal Y_{\sigma_2}(w^{(i_2)},z_2-z_1)w^{(i_1)},z_1\big)\Omega,w^{(\overline i)}\big\rangle.
	\end{align}
	Note that in order to make the above argument  valid, we have to check that the expression in each step converges absolutely. To see this, we choose any $m=1,\dots,n$, and let $W_{j_m}$ be the target space of $\mathcal Y_{\sigma_m}$. Then
	\begin{align}
	&\langle \mathcal Y_{\sigma_n}(w^{(i_n)},z_n)\cdots e^{z_1L_{-1}}\mathcal Y_{\sigma_m}(w^{(i_m)},z_m-z_1)\cdots\mathcal Y_{\sigma_2}(w^{(i_2)},z_2-z_1)w^{(i_1)},w^{(\overline i)}\rangle\nonumber\\
	=&\sum_{s_1,\dots,s_{n-1}\in\mathbb R}\langle   \mathcal Y_{\sigma_n}(w^{(i_n)},z_n)P_{s_{n-1}} \cdots P_{s_1}e^{z_1L_{-1}}P_{s_m}\mathcal Y_{\sigma_m}(w^{(i_m)},z_m-z_1)P_{s_{m-1}}\nonumber\\
	&\qquad\qquad \cdots P_{s_2}\mathcal Y_{\sigma_2}(w^{(i_2)},z_2-z_1)w^{(i_1)},w^{(\overline i)}\rangle\nonumber\\
	=&\sum_{s_1,\dots,s_{n-1}\in\mathbb R}\big\langle   \mathcal Y_{\sigma_n}(w^{(i_n)},z_n)P_{s_{n-1}} \cdots P_{s_1}\mathcal Y^{j_m}_{j_m0}\big(P_{s_m}\mathcal Y_{\sigma_m}(w^{(i_m)},z_m-z_1)P_{s_{m-1}}\nonumber\\
	&\qquad\qquad \cdots P_{s_2}\mathcal Y_{\sigma_2}(w^{(i_2)},z_2-z_1)w^{(i_1)},z_1\big)\Omega,w^{(\overline i)}\big\rangle,
	\end{align}
	which, by \eqref{eq280} and theorem \ref{lb12},	converges absolutely. Therefore, equation \eqref{eq37} holds when both sides act on the vacuum vector $\Omega$. By (the proof of) corollary \ref{lb84}, equation \eqref{eq37} holds when acting on any vector $v\in V$.
\end{proof}

\begin{co}\label{lb17}
	Let $\mathcal Y_\alpha\in\mathcal V {k\choose i~j}$. Let $z_i,z_j\in S^1$ with arguments satisfying $\arg z_j<\arg z_i<\arg z_j+2\pi$. Then for any $w^{(i)}\in W_i$ and $w^{(j)}\in W_j$, we have the braid relation
	\begin{align}
	\mathcal Y_\alpha(w^{(i)},z_i)\mathcal Y^j_{j0}(w^{(j)},z_j)=\mathcal Y_{B_+\alpha}(w^{(j)},z_j)\mathcal Y^i_{i0}(w^{(i)},z_i).
	\end{align}
\end{co}
\begin{proof}
	By analytic continuation, we may assume, without loss of generality, that $0<|z_i-z_j|<1$. Let $\arg(z_i-z_j)$ be close to $\arg z_i$ as $z_j\rightarrow 0$, and let $\arg(z_j-z_i)$ be close to $\arg z_j$ as $z_i\rightarrow 0$. Then by propositions \ref{lb10} and \ref{lb86},
	\begin{align*}
	&\mathcal Y_\alpha(w^{(i)},z_i)\mathcal Y^j_{j0}(w^{(j)},z_j)\\
	=&\mathcal Y^k_{k0}\big(\mathcal Y_\alpha(w^{(i)},z_i-z_j)w^{(j)},z_j \big)\\
	=&\mathcal Y^k_{k0}\big(\mathcal Y_{B_+\alpha}(w^{(j)},z_j-z_i)w^{(i)},z_i \big)\\
	=&\mathcal Y_{B_+\alpha}(w^{(j)},z_j)\mathcal Y^i_{i0}(w^{(i)},z_i).
	\end{align*}
\end{proof}

\subsection{The ribbon categories associated to VOAs}\label{lb33}

We refer the reader to \cite{Tur16} for the general theory of ribbon categories and modular tensor categories. See also \cite{BK01,Etingof}. In this section, we review the construction of the ribbon category $\Rep(V)$ for $V$  by Huang and Lepowspky. (cf. \cite{HL94} and \cite{H Rigidity}.) As an additive categoy, $\Rep(V)$ is the representation category of $V$:  Objects of $\Rep(V)$ are $V$-modules, and the vector space of morphisms from $W_i$ to $W_j$ is $\Hom_V(W_i,W_j)$. We now equip with $\Rep(V)$ a structure of a ribbon category.

The \textbf{tensor product} of two $V$-modules $W_i,W_j$ is defined to be
\begin{gather}
W_{ij}\equiv W_i\boxtimes W_j=\bigoplus_{k\in\mathcal E}\mathcal V{k \choose i~j}^*\otimes W_k,\nonumber\\
Y_{ij}(v,x)=\bigoplus_{k\in\mathcal E}\id\otimes Y_k(v,x)\qquad(v\in V), \label{eq270}
\end{gather}
where $\mathcal V{k \choose i~j}^*$ is the dual space of $\mathcal V{k \choose i~j}$. (Recall our notations at the beginning of this chapter.) Thus for any $k\in\mathcal E$, we can define an isomorphism $$\mathcal V{k\choose i~j}\rightarrow \Hom_V(W_{ij},W_k),~~~\mathcal Y\mapsto R_{\mathcal Y},$$ such that if $\widecheck {\mathcal Y}\in \mathcal V{k\choose i~j}^*$ and $w^{(k)}\in W_{k}$, then
\begin{equation}
R_{\mathcal Y}(\widecheck{\mathcal Y}\otimes w^{(k)})=\langle \mathcal {\widecheck{\mathcal Y}},\mathcal Y\rangle w^{(k)}.
\end{equation}
$R_{\mathcal Y}$ is called the homomorphism represented by $\mathcal Y$.

The tensor product of two morphisms are defined as follows: If  $F\in\Hom_V(W_{i_1},W_{i_2}),G\in\Hom_V(W_{j_1},W_{j_2})$, then for each $k\in\mathcal E$ we have a linear map $(F\otimes G)^\tr:\mathcal V{k\choose i_2~j_2}\rightarrow \mathcal V{k\choose i_1~j_1}$, such that if $\mathcal Y\in\mathcal V{k\choose i_2~j_2}$, then  $(F\otimes G)^\tr\mathcal Y\in\mathcal V{k\choose i_1~j_1}$, and for any $w^{(i_1)}\in W_{i_1},w^{(j_1)}\in W_{j_1}$,
\begin{align}
\big((F\otimes G)^\tr\mathcal Y\big)(w^{(i_1)},x)w^{(j_1)}=\mathcal Y(Fw^{(i_1)},x)Gw^{(j_1)}.\label{eq141}
\end{align}
Then $F\otimes G:\mathcal V{k\choose i_1~j_1}^*\rightarrow\mathcal V{k\choose i_2~j_2} ^*$ is defined  to be the transpose of $(F\otimes G)^\tr$, and can be extended to a homomorphism $$F\otimes G=\bigoplus_{k\in\mathcal E}(F\otimes G)\otimes \id_{k}:W_{i_1}\boxtimes W_{j_1}\rightarrow W_{i_2}\boxtimes W_{j_2}.$$ Hence we've define the tensor product $F\otimes G$ of $F$ and $G$. 

Let $W_0=V$ be the unit object  of $\Rep(V)$. The functorial isomorphisms $\lambda_i:W_0\boxtimes W_i\rightarrow W_i$ and $\rho_j:W_i\boxtimes W_0\rightarrow W_i$ are defined as follows: If $i\in\mathcal E$, then $\lambda_i$ is represented by the intertwining operator $Y_i$, and $\rho_i$ is represented by $\mathcal Y^i_{i0}$. In general, $\lambda_i$ (resp. $\rho_i$) is defined to be the unique isomorphism satisfying that for any $k\in\mathcal E$ and any $R\in\Hom_V(W_i,W_k)$, $R\lambda_i=\lambda_k(\id_0\otimes R)$ (resp. $R\rho_i=\rho_k(R\otimes\id_0)$).

We now define the associator. First of all, to simplify our notations, we assume the following:
\begin{cv}
Let $W_i,W_j,W_k,W_{i'},W_{j'},W_{k'}$ be $V$-modules. Let $\mathcal Y_\alpha\in\mathcal V{k'\choose i'~j'}$. If either $W_i\neq W_{i'},W_j\neq W_{j'}$, or $W_{k}\neq W_{k'}$, then for any $w^{(i)}\in W_i,w^{(j)}\in W_j,w^{(\overline k)}\in W_{\overline k},z\in\mathbb C^\times$, we let $$\langle \mathcal Y_\alpha(w^{(i)},z)w^{(j)},w^{(\overline k)}\rangle=0.$$
Therefore, $\mathcal Y_\beta(w^{(i)},z_2)\mathcal Y_\alpha(w^{(j)},z_1)=0$ if the target space of $\mathcal Y_\alpha$ does not equal the source space of $\mathcal Y_\beta$;  $\mathcal Y_\gamma\big(\mathcal Y_\delta(w^{(i)},z_1-z_2)w^{(j)},z_2 \big)=0$  if the target space of $\mathcal Y_\delta$ does not equal the charge space of $\mathcal Y_\gamma$.
\end{cv}

Given three $V$-modules $W_i,W_j,W_k$,  we have
\begin{gather}
(W_i\boxtimes W_j)\boxtimes W_k=\bigoplus_{s,t\in\mathcal E}\mathcal V{t\choose s~k}^*\otimes\mathcal V{s\choose i~j}^* \otimes W_t,\\
W_i\boxtimes (W_j\boxtimes W_k)=\bigoplus_{r,t\in\mathcal E} \mathcal V{t\choose i~r}^*\otimes\mathcal V{r\choose j~k}^*\otimes W_t.
\end{gather}
Choose basis $\Theta^t_{sk},\Theta^s_{ij},\Theta^t_{ir},\Theta^r_{jk}$ of these spaces of intertwining operators. Choose $z_i,z_j\in\mathbb C^\times$ satisfying $0<|z_i-z_j|<|z_j|<|z_i|$. Choose $\arg z_i$. Let $\arg z_j$ be close to $\arg z_i$ as $z_i-z_j\rightarrow 0$, and let $\arg (z_i-z_j)$ be close to $\arg z_i$ as $z_j\rightarrow 0$.  For any $t\in\mathcal E,\alpha\in\Theta^t_{i*},\beta\in\Theta^*_{jk}$, there exist complex numbers $F^{\beta'\alpha'}_{\alpha\beta}$ independent of the choice of $z_i,z_j$, such that for any $w^{(i)}\in W_i,w^{(j)}\in W_j$, we have the fusion relation
\begin{align}
\mathcal Y_\alpha(w^{(i)},z_i)\mathcal Y_\beta(w^{(j)},z_j)=\sum_{\alpha'\in\Theta^*_{ij},\beta'\in\Theta^t_{*k}}F^{\beta'\alpha'}_{\alpha\beta}\mathcal Y_{\beta'}(\mathcal Y_{\alpha'}(w^{(i)},z_i-z_j)w^{(j)},z_j).\label{eq139}
\end{align}
If the source space of $\mathcal Y_\alpha$ does not equal the target space of $\mathcal Y_\beta$, or if the target space of $\mathcal Y_{\alpha'}$ does not equal the charge space of $\mathcal Y_{\beta'}$,  we set $F^{\beta'\alpha'}_{\alpha\beta}=0$. Then, by the proof of proposition \ref{lb66}, the numbers $F^{\beta'\alpha'}_{\alpha\beta}$ are uniquely determined by the basis chosen. The matrix $\{F^{\beta'\alpha'}_{\alpha\beta}\}_{\alpha\in\Theta^t_{i*},\beta\in\Theta^*_{jk}}^{\alpha'\in\Theta^*_{ij},\beta'\in\Theta^t_{*k}}$ is called a \textbf{fusion matrix}. Define an isomorphism
\begin{gather}
A^\tr:\bigoplus_{r\in\mathcal E} \mathcal V{t\choose i~r}\otimes\mathcal V{r\choose j~k}\rightarrow\bigoplus_{s\in\mathcal E} \mathcal V{t\choose s~k}\otimes \mathcal V{s\choose i~j},\nonumber\\
\mathcal Y_\alpha\otimes \mathcal Y_\beta\mapsto \sum_{\alpha'\in\Theta^*_{ij},\beta'\in\Theta^t_{*k}}F^{\beta'\alpha'}_{\alpha\beta}\mathcal Y_{\beta'}\otimes\mathcal Y_{\alpha'}.\label{eq140}
\end{gather}
Clearly $A^\tr$ is independent of the basis chosen. Define
\begin{equation}
A:\bigoplus_{s\in\mathcal E} \mathcal V{t\choose s~k}^*\otimes \mathcal V{s\choose i~j}^*\rightarrow\bigoplus_{r\in\mathcal E} \mathcal V{t\choose i~r}^*\otimes\mathcal V{r\choose j~k}^*
\end{equation}
to be the transpose of $A^\tr$, and extend it to
\begin{align}
A=\sum_{t\in\mathcal E}A\otimes\id_t: (W_i\boxtimes W_j)\boxtimes W_k\rightarrow W_i\boxtimes (W_j\boxtimes W_k),
\end{align}
which is an  \textbf{associator} of $\Rep(V)$. One can prove the pentagon axiom  using theorem \ref{lb12} and corollary \ref{lb13}, and prove the triangle axiom using propositions \ref{lb80} and \ref{lb85}.

Recall the linear isomorphisms $$B_\pm:\mathcal V{k\choose j~i}\rightarrow \mathcal V{k\choose i~j},~~~\mathcal Y\mapsto B_\pm\mathcal Y.$$ We let $\sigma_{i,j}:\mathcal V{k\choose i~j}^*\rightarrow \mathcal V{k\choose j~i}^*$ be the transpose of $B_+$ and extend it to a morphism
\begin{align}
\sigma_{i,j}=\sum_{t\in\mathcal E}\sigma_{i,j}\otimes\id_t:W_i\boxtimes W_j\rightarrow W_j\boxtimes W_i.
\end{align}
This gives the \textbf{braid operator}. The  hexagon axiom can be proved using  propositions \ref{lb79}, \ref{lb86}, and theorem \ref{lb12}.

For each object $i$, the twist is just the operator $\vartheta_i=\vartheta_{W_i}$ defined in section \ref{lb29}.

With these structural maps, Huang proved in \cite{H MI,H Verlinde,H Rigidity} that $\Rep(V)$ is rigid and in fact a modular tensor category.  From his proof, it is clear that $\overline i$ is the right  dual of $i$: there exist homomorphisms $\coev_i:V\rightarrow W_i\boxtimes W_{\overline i}$ and $\ev_i:W_{\overline i}\boxtimes W_i\rightarrow V$ satisfying
\begin{gather}
(\id_i\otimes\ev_i)\circ(\coev_i\otimes \id_i)=\id_i,\label{eq145}\\
(\ev_i\otimes \id_{\overline i})\circ(\id_{\overline i}\otimes\coev_i)=\id_{\overline i}.\label{eq146}
\end{gather} 
Since $i=\overline{\overline i}$, $\overline i$ is also the left dual of $i$.\\

Now assume that $V$ is unitary. The additive category $\Rep^\uni(V)$ is defined to be the representation category of unitary $V$-modules. We  show that $\Rep^\uni(V)$ is a $C^*$-category. First, we need the following  easy consequence of Schur's lemma.
\begin{lm}\label{lb82}
Choose for each $k\in\mathcal E^\uni$ a number $n_k\in\mathbb Z_{\geq0}$. Define the unitary $V$-module $$W=\bigoplus^\perp_{k\in\mathcal E^\uni}W_k\otimes \mathbb C^{n_k}=\bigoplus^\perp_{k\in\mathcal E^\uni}\underbrace{W_k\oplus^\perp W_k\oplus^\perp\cdots\oplus^\perp W_k}_{n_k}.$$
Then we have
\begin{align}
\End_V(W)=\bigoplus_{k\in\mathcal E^\uni}\id_k\otimes \End(\mathbb C^{n_k}).
\end{align}
\end{lm}

\begin{thm}\label{lb114}
$\Rep^\uni(V)$ is a $C^*$-category, i.e., $\Rep^\uni(V)$ is equipped with an involutive antilinear contravariant endofunctor $*$ which is the identity on objects; The positivity condition is satisfied: If $W_i,W_j$ are unitary and $F\in\Hom_V(W_i,W_j)$, then there exists $R\in\End_V(W_i)$ such that $F^*F=R^*R$; The hom-spaces $\Hom_V(W_i,W_j)$ are normed spaces and the norms satisfy
\begin{gather}
\|GF\|\leq\|G\|\|F\|,\quad \|F^*F\|=\|F\|^2
\end{gather}
for all $F\in\Hom(i,j),G\in\Hom(j,k)$. 
\end{thm}
\begin{proof}
For any $F\in\Hom_V(W_i,W_j)$, we let $F^*$ be the formal adjoint of $F$ , i.e. the unique homomorphism $F^*\in \Hom_V(W_j,W_i)$ satisfying $\langle Fw^{(i)}|w^{(j)} \rangle=\langle w^{(i)}|F^*w^{(j)}\rangle$ for all $w^{(i)}\in W_i,w^{(j)}\in W_j$. The existence of $F^*$ follows from lemma \ref{lb82} applied to  $W\cong W_i\oplus^\perp W_j$. Let $\|F\|$ be the operator norm of $F$, i.e., $\|F\|=\sup_{w^{(i)}\in W_i\setminus\{0\}}(\|Fw^{(i)}\|/\|w^{(i)}\|)$. Using lemma \ref{lb82}, it is easy to check that $\Rep^\uni(V)$ satisfies all the conditions to be a $C^*$-category.
\end{proof}

It is not clear whether  unitarizable $V$-modules are closed under tensor product. So it may not be a good idea to define a structure of a ribbon category on $\Rep^\uni(V)$. We consider instead certain subcategories.  Let $\mathcal G$ be a collection of unitary $V$-modules. We say that $\mathcal G$ is \textbf{additively closed}, if the following conditions are satisfied:\\
(1) If $i\in\mathcal G$ and $W_j$ is isomorphic to a submodule of  $W_i$, then $j\in\mathcal G$.\\
(2) If $i_1,i_2,\dots,i_n\in\mathcal G$, then $W_{i_1}\oplus^\perp W_{i_2}\oplus^\perp \cdots\oplus^\perp W_{i_n}\in\mathcal G$.\\
If $\mathcal G$ is additively closed, we define the  additive category $\Rep^\uni_{\mathcal G}(V)$ to be the subcategory of $\Rep^\uni(V)$ whose objects are elements in $\mathcal G$.

We say that $\mathcal G$ is \textbf{multiplicatively closed}, if $\mathcal G$ is additively closed, and the following conditions are  satisfied:\\
(a) $0\in\mathcal G$.\\
(b) If $i\in\mathcal G$, then $\overline i\in\mathcal G$.\\
(c) If $i,j\in\mathcal G$, then $W_{ij}=W_i\boxtimes W_j$ is unitarizable, and any unitarization  of $W_{ij}$ is inside $\mathcal G$.

Suppose that $\mathcal G$ is multiplicatively closed. A \textbf{unitary structure}  on $\mathcal G$ assigns to each triplet $(i,j,k)\in \mathcal G\times\mathcal G\times\mathcal E$ an inner product on $\mathcal V{k\choose i~j}^*$. For each unitary structure  on $\mathcal G$, we define $\Rep^\uni_{\mathcal G}(V)$ to be a ribbon  category in the following way: If $i,j\in\mathcal G$, then $W_{ij}$ as a $V$-module is defined, as before, to be $\bigoplus_{k\in\mathcal E}\mathcal V{k \choose i~j}^*\otimes W_k$. Since $\mathcal G$ is multiplicatively closed, each $W_k$ in $\mathcal E$ satisfying $N^k_{ij}>0$ must be equipped with a unitary structure. Hence the inner products on all  $\mathcal V{k\choose i~j}^*$'s give rise to a unitary structure on $W_{ij}$. $W_{ij}$ now becomes a unitary $V$-module. The other functors and structural maps are defined in the same way as we did for $\Rep(V)$. Clearly $\Rep^\uni_{\mathcal G}(V)$ is a ribbon fusion category and is equivalent to a ribbon fusion subcategory of $\Rep(V)$.

Our main goal in this two-part series is to define a unitary structure on $\mathcal G$, under which $\Rep^\uni_{\mathcal G}(V)$ becomes a unitary ribbon fusion category. More explicitly, we want to show (cf. \cite{Gal12}) that for any $i_1,i_2,j_1,j_2\in\mathcal G$ and any $F\in \Hom_V(W_{i_1},W_{i_2}),G\in\Hom_V(W_{j_1},W_{j_2})$, 
\begin{equation}\label{eq142}
(F\otimes G)^*=F^*\otimes G^*;
\end{equation}
that  the associators, the operators $\lambda_i,\rho_i$ ($i\in\mathcal G$), and the braid operators of $\Rep^\uni_{\mathcal G}(V)$ are unitary; and that for each $i\in\mathcal G$, $\vartheta_i$ is unitary, and $\ev_i$ and $\coev_i$ can be chosen in such a way that the following equations hold:
\begin{gather}
(\coev_i)^*=\ev_i\circ \sigma_{i,\overline i}\circ(\vartheta_i\otimes\id_{\overline i}),\label{eq143}\\
(\ev_i)^*=(\id_{\overline i}\otimes\vartheta_i^{-1})\circ \sigma_{\overline i,i}^{-1}\circ\coev_i.\label{eq144}
\end{gather}

\section{Analytic aspects of vertex operator algebras}\label{lb90}

\subsection{Intertwining operators with energy bounds}\label{lb76}

The energy bounds conditions for vertex operators are important when one tries to construct conformal nets/loop groups representations from unitary VOAs/infinite dimensional Lie algebras. This can be seen, for instance, in \cite{GW}, \cite{BSM}, and \cite{CKLW}. In this section, we generalize this notion to intertwining operators of VOAs.

We assume in this chapter that $V$ is unitary. If $W_i$ is a unitary $V$-module, we let the Hilbert space $\mathcal H_i$ be the norm completion of $W_i$, and view $W_i$ as a norm-dense subspace of $\mathcal H_i$. It is clear that the unbounded operator $L_0$ on $\mathcal H_i$ (with domain $W_i$) is essentially self-adjoint, and  its closure $\overline {L_0}$ is positive.

Now for any $r\in\mathbb R$, we let $\mathcal H_i^r$ be the domain $\mathscr D((1+\overline{L_0})^r)$ of $(1+\overline{L_0})^r$. If $\xi\in\mathcal H^r_i$, we define the $r$-th order Sobolev norm to be
$$\lVert\xi\lVert_r=\lVert(1+\overline{L_0})^r\xi\lVert.$$
Note that the $0$-th Sobolev norm is just the vector norm. We let $$\mathcal H_i^\infty=\bigcap_{r\geq 0}\mathcal H^r_i.$$ Clearly $\mathcal H_i^\infty$ contains $W_i$. Vectors inside $\mathcal H_i^\infty$ are said to be \textbf{smooth}.

\begin{df}
Let $W_i,W_j,W_k$ be unitary  $V$-modules,  $\mathcal Y_\alpha\in\mathcal V{k\choose i~j}$, and $w^{(i)}\in W_i$ be homogeneous. Choose $r\geq0$. We say that $\mathcal Y_\alpha(w^{(i)},x)$ satisfies \textbf{$r$-th order energy bounds}, if there exist $M\geq0,t\geq0$, such that for any $s\in\mathbb R$ and $w^{(j)}\in W_j$, 
\begin{align}
\lVert \mathcal Y_\alpha(w^{(i)},s)w^{(j)} \lVert\leq M(1+|s|)^t\lVert w^{(j)} \lVert_r.\label{eq211}
\end{align}
Here $\mathcal Y_\alpha(w^{(i)},s)$ is the $s$-th mode of the intertwining operator $\mathcal Y_\alpha(w^{(i)},x)$. It is clear that if $r_1\leq r_2$ and $\mathcal Y_\alpha(w^{(i)},x)$ satisfies $r_1$-th order energy bounds, then $\mathcal Y_\alpha(w^{(i)},x)$ also satisfies $r_2$-th order energy bounds.

$1$-st order energy bounds are called \textbf{linear energy bounds}. We say that $\mathcal Y_\alpha(w^{(i)},x)$ is \textbf{energy-bounded}  if it satisfies energy bounds of some positive order. If for every $w^{(i)}\in W_i$, $\mathcal Y_\alpha(w^{(i)},x)$ is energy-bounded,  we say that $\mathcal Y_\alpha$ is  \textbf{energy-bounded}. A unitary $V$-module $W_i$ is called \textbf{energy-bounded} if $Y_i$ is energy-bounded. The unitary VOA $V$ is called \textbf{energy-bounded} if the vacuum module $V=W_0$ is energy-bounded. We now prove some useful properties concerning energy-boundedness.
\end{df}	

\begin{pp}\label{lb53}
If $w^{(i)}\in W_i$ is homogeneous and inequality \eqref{eq211} holds, then for any $p\in\mathbb R$, there exists $M_p>0$ such that for any   $w^{(j)}\in W_j$, 
\begin{align}
\lVert \mathcal Y_\alpha(w^{(i)},s)w^{(j)} \lVert_p\leq M_p(1+|s|)^{|p|+t}\lVert w^{(j)} \lVert_{p+r}.\label{eq295}
\end{align}
\end{pp}
\begin{proof}
(cf. \cite{Toledano} chapter II proposition 1.2.1) We want to show that
\begin{align}
\lVert \mathcal Y_\alpha(w^{(i)},s)w^{(j)} \lVert_p^2\leq M_p^2(1+|s|)^{2(|p|+t)}\lVert w^{(j)} \lVert_{p+r}^2.\label{eq212}
\end{align}
Since 
\begin{gather*}
\lVert \mathcal Y_\alpha(w^{(i)},s)w^{(j)} \lVert_p^2=\sum_{q\in\mathbb R}\lVert P_{q-s-1+\Delta_{w^{(i)}}}\mathcal Y_\alpha(w^{(i)},s)w^{(j)} \lVert_p^2=\sum_{q\in\mathbb R}\lVert \mathcal Y_\alpha(w^{(i)},s)P_qw^{(j)} \lVert_p^2,\\
\lVert w^{(j)} \lVert_{p+r}^2=\sum_{q\in\mathbb R}\lVert P_qw^{(j)} \lVert_{p+r}^2,
\end{gather*}
 it suffices to assume that $w^{(j)}$ is homogeneous. We also assume that $\mathcal Y_\alpha(w^{(i)},s)w^{(j)}\neq 0$. Then by proposition \ref{lb88}, $\Delta_{w^{(i)}}+\Delta_{w^{(j)}}-1-s\geq0$.

By \eqref{eq211} we have
\begin{align}
\lVert \mathcal Y_\alpha(w^{(i)},s)w^{(j)} \lVert^2\leq M^2(1+|s|)^{2t}(1+\Delta_{w^{(j)}})^{2r}\lVert w^{(j)} \lVert^2.
\end{align}
Hence
\begin{align}
&\lVert \mathcal Y_\alpha(w^{(i)},s)w^{(j)} \lVert_p^2\nonumber\\
=&(\Delta_{w^{(i)}}+\Delta_{w^{(j)}}-s)^{2p}\lVert \mathcal Y_\alpha(w^{(i)},s)w^{(j)} \lVert^2\nonumber\\
\leq& (\Delta_{w^{(i)}}+\Delta_{w^{(j)}}-s)^{2p}M^2(1+|s|)^{2t}(1+\Delta_{w^{(j)}})^{2r}\lVert w^{(j)} \lVert^2\nonumber\\
=&M^2\Big(\frac {\Delta_{w^{(i)}}+\Delta_{w^{(j)}}-s}{1+\Delta_{w^{(j)}}}\Big)^{2p}(1+|s|)^{2t}(1+\Delta_{w^{(j)}})^{2(p+r)}\lVert w^{(j)} \lVert^2\nonumber\\
=&M^2\Big(\frac {\Delta_{w^{(i)}}-s+\Delta_{w^{(j)}}}{1+\Delta_{w^{(j)}}}\Big)^{2p}(1+|s|)^{2t}\lVert w^{(j)} \lVert^2_{p+r}.
\end{align}
If $p\geq0$ then
\begin{align}
&\Big(\frac {\Delta_{w^{(i)}}-s+\Delta_{w^{(j)}}}{1+\Delta_{w^{(j)}}}\Big)^{2p}\nonumber\\
\leq&\Big(\frac {1+\Delta_{w^{(i)}}+|s|+\Delta_{w^{(j)}}}{1+\Delta_{w^{(j)}}}\Big)^{2p}\nonumber\\
\leq&(1+\Delta_{w^{(i)}}+|s|)^{2p}\leq (1+\Delta_{w^{(i)}})^{2p}(1+|s|)^{2p}.
\end{align}
If $p<0$ and $1\leq\Delta_{w^{(i)}}-s$, then
\begin{align}
&\Big(\frac {\Delta_{w^{(i)}}-s+\Delta_{w^{(j)}}}{1+\Delta_{w^{(j)}}}\Big)^{2p}\nonumber\\
=&\Big(\frac{1+\Delta_{w^{(j)}}} {\Delta_{w^{(i)}}-s+\Delta_{w^{(j)}}}\Big)^{2|p|}
\leq1.
\end{align}
If $p<0$ and $1\geq\Delta_{w^{(i)}}-s$, then since $\Delta_{w^{(i)}}-s+\Delta_{w^{(j)}}\geq 1$,
\begin{align}
&\Big(\frac{1+\Delta_{w^{(j)}}} {\Delta_{w^{(i)}}-s+\Delta_{w^{(j)}}}\Big)^{2|p|}\nonumber\\
=&\Big(1+\frac{1+s-\Delta_{w^{(i)}}}{\Delta_{w^{(i)}}-s+\Delta_{w^{(j)}}}  \Big)^{2|p|}\nonumber\\
\leq&(2-\Delta_{w^{(i)}}+s)^{2|p|}\nonumber\\
\leq&(2+2\Delta_{w^{(i)}}+2|s|)^{2|p|}\nonumber\\
\leq& 2^{2|p|}(1+\Delta_{w^{(i)}})^{2|p|}(1+|s|)^{2|p|}.
\end{align}
Therefore, if we let $M_p=2^{|p|}(1+\Delta_{w^{(i)}})^{|p|}$, then   \eqref{eq212} is always true.
\end{proof}

The next property is obvious.
\begin{pp}\label{lb94}
	If $\mathcal Y\in\mathcal V{k\choose i~j}$ is unitary, $w^{(i)}\in W_i$ is homogeneous, and $\mathcal Y_\alpha(w^{(i)},x)$ satisfies $r$-th order energy bounds, then $\mathcal Y_{\overline \alpha}(\overline {w^{(i)}},x)$	satisfies $r$-th order energy bounds.
\end{pp}

\begin{pp}\label{lb56}
Suppose that $\mathcal Y_\alpha\in \mathcal V{k\choose i~j}$ is unitary, $w^{(i)}\in W_i$ is homogeneous, $r\geq0$, and for any $m\in\mathbb Z_{\geq0}$, $\mathcal Y_\alpha(L_1^mw^{(i)},x)$ satisfies $r$-th order energy bounds. Then  $\mathcal Y_{\alpha^*}(\overline{w^{(i)}},x)$ and $\mathcal Y_{C^{\pm1}\alpha}(w^{(i)},x)$ satisfy $r$-th order energy bounds.
\end{pp}
\begin{proof}
	First we note that $L_1^mw^{(i)}=0$ for $m$ large enough. Now suppose that \eqref{eq211} holds for all $m$ if we replace $w^{(i)}$ by $L_1^mw^{(i)}$. Then by \eqref{eq213}, for any $w^{(j)}\in W_j,w^{(k)}\in W_k$ and $s\in\mathbb R$,
	\begin{align*}
	&\Big|\big\langle \mathcal Y_{\alpha^*}(\overline{w^{(i)}},s)w^{(j)}\big|w^{(k)} \big\rangle\Big|\nonumber\\
	\leq&\sum_{m\geq0}\frac 1 {m!}\Big|\big\langle w^{(j)}\big|\mathcal Y_\alpha(L_1^mw^{(i)},-s-k-2+2\Delta_{w^{(i)}})w^{(k)} \big\rangle\Big|\\
	=&\sum_{m\geq0}\frac 1 {m!}\Big|\big\langle(1+L_0)^r w^{(j)}\big|(1+L_0)^{-r}\mathcal Y_\alpha(L_1^mw^{(i)},-s-k-2+2\Delta_{w^{(i)}})w^{(k)} \big\rangle\Big|\\
	\leq&\sum_{m\geq0}\frac 1 {m!}\big\lVert w^{(j)} \big\lVert_r\big\lVert \mathcal Y_\alpha(L_1^mw^{(i)},-s-k-2+2\Delta_{w^{(i)}})w^{(k)} \big\lVert_{-r}.
	\end{align*}
	By proposition \ref{lb53}, we can find  positive numbers $C_1,C_2$ independent of  $w^{(j)},w^{(k)}$, such that
	\begin{align*}
	&\big\lVert \mathcal Y_\alpha(L_1^mw^{(i)},-s-m-2+2\Delta_{w^{(i)}})w^{(k)} \big\lVert_{-r}\nonumber\\
	\leq&C_1\big(1+|s+m+2-2\Delta_{w^{(i)}}|\big)^{r+t}\big\lVert w^{(k)} \big\lVert\\
	\leq& C_2\big(1+|s|\big)^{r+t}\big\lVert w^{(k)} \big\lVert.
	\end{align*}
	Thus there exists $C_3>0$ independent of $w^{(j)},w^{(k)}$, such that
	\begin{align*}
	\Big|\big\langle \mathcal Y_{\alpha^*}(\overline{w^{(i)}},s)w^{(j)}\big|w^{(k)} \big\rangle\Big|\leq C_3\big(1+|s|\big)^{r+t}\big\lVert w^{(j)} \big\lVert_r\big\lVert w^{(k)} \big\lVert.
	\end{align*}
	This proves that
	\begin{align}
	\big\lVert \mathcal Y_{\alpha^*}(\overline{w^{(i)}},s)w^{(j)} \big\lVert\leq C_3\big(1+|s|\big)^{r+t}\big\lVert w^{(j)} \big\lVert_r.
	\end{align}
Therefore $\mathcal Y_{\alpha^*}(\overline{w^{(i)}},x)$ satisfies $r$-th order energy bounds. Since $C\alpha=\overline{\alpha^*}$ and $\mathcal Y_{C^{-1}\alpha}(w^{(i)},x)=e^{2i\pi\Delta_{w^{(i)}}}\mathcal Y_{C\alpha}(w^{(i)},x)$, by proposition \ref{lb94}, $\mathcal Y_{C^{\pm1}\alpha}(w^{(i)},x)$ also satisfy $r$-th order energy bounds.
\end{proof}

\begin{pp}\label{lb55}
	Let $W_i,W_j,W_k$ be unitary $V$-modules, $\mathcal Y_\alpha\in \mathcal V{k\choose i~j}$, and choose homogeneous vectors $w^{(i)}\in W_i,u\in V$. Suppose that $\mathcal Y_\alpha(w^{(i)},x),Y_j(u,x),Y_k(u,x)$ are energy-bounded. Then for any $n\in\mathbb Z$, $\mathcal Y_\alpha\big(Y_i(u,n)w^{(i)},x\big)$ is energy-bounded.
\end{pp}
\begin{proof}
	By Jacobi Identity, for any $s\in\mathbb R$ we have
	\begin{align}
	&\mathcal Y_\alpha(Y_i(u,n)w^{(i)},s)\nonumber\\
	=&\sum_{l\in\mathbb Z_{\geq0}}(-1)^l{n\choose l}Y_k(u,n-l)\mathcal Y_\alpha(w^{(i)},s+l)-\sum_{l\in\mathbb Z_{\geq0}}(-1)^{l+n}{n\choose l}\mathcal Y_\alpha(w^{(i)},n+s-l)Y_j(u,l).
	\end{align}
	It can be shown by induction on $|n|$ that
	$$\limsup_{l\rightarrow \infty}\bigg|{n\choose l}\bigg|l^{-|n|}<+\infty.$$
	Choose any homogeneous vector $w^{(j)}\in W_j$ with energy $\Delta_{w^{(j)}}$. Then by energy-boundedness of $\mathcal Y_\alpha(w^{(i)},x),Y_j(u,x),Y_k(u,x)$ and proposition \ref{lb53}, there exist positive constants $C_1,C_2,\dots,C_8$ and $r_1,t_1,r_2,t_2,r_3,t_3$ independent of $w^{(j)}$ and $s$, such that
	\begin{align}
	&\bigg\lVert \sum_{l\geq0}(-1)^{l+n}{n\choose l}\mathcal Y_\alpha(w^{(i)},n+s-l)Y_j(u,l)w^{(j)} \bigg\lVert\nonumber\\
	\leq&\sum_{l\geq0}C_1l^{|n|}\big\lVert \mathcal Y_\alpha(w^{(i)},n+s-l)Y_j(u,l)w^{(j)} \big\lVert\nonumber\\
	\leq&\sum_{l\geq0}C_2l^{|n|}\big(1+|n+s-l|\big)^{t_1}\big\lVert Y_j(u,l)w^{(j)} \big\lVert_{r_1}\nonumber\\
	\leq&\sum_{0\leq l\leq\Delta_u+\Delta_{w^{(j)}}-1}C_3l^{|n|}(1+|s|)^{t_1}(1+l)^{t_1}\cdot (1+l)^{r_1+t_2}\lVert w^{(j)} \lVert_{r_1+r_2}\nonumber\\
	\leq&\sum_{0\leq l\leq\Delta_u+\Delta_{w^{(j)}}-1}C_3(1+|s|)^{t_1} (1+l)^{|n|+t_1+r_1+t_2}\lVert w^{(j)} \lVert_{r_1+r_2}\nonumber\\
	\leq& C_4(1+|s|)^{t_1}(1+\Delta_{w^{(j)}})^{1+|n|+t_1+r_1+t_2}\lVert w^{(j)} \lVert_{r_1+r_2}\nonumber\\
	=&C_4(1+|s|)^{t_1}\lVert w^{(j)} \lVert_{1+|n|+t_1+t_2+2r_1+r_2}.
	\end{align}
	Here the inequality $l\leq\Delta_u+\Delta_{w^{(j)}}-1$ comes from the fact that every nonzero  $Y_j(u,l)w^{(j)}$ must have non-negative energy. Similarly we have
	\begin{align}
	&\bigg\lVert \sum_{l\geq0}(-1)^l{n\choose l}Y_k(u,n-l)\mathcal Y_\alpha(w^{(i)},s+l)w^{(j)} \bigg\lVert\nonumber\\
	\leq&\sum_{l\geq0}C_5l^{|n|}\big\lVert Y_k(u,n-l)\mathcal Y_\alpha(w^{(i)},s+l)w^{(j)} \big\lVert\nonumber\\
	\leq&\sum_{l\geq0}C_6l^{|n|}(1+|n-l|)^{t_3}\big\lVert \mathcal Y_\alpha(w^{(i)},s+l)w^{(j)} \big\lVert_{r_3}\nonumber\\
	\leq&\sum_{0\leq l\leq\Delta_{w^{(i)}}+\Delta_{w^{(j)}}-s-1} C_7l^{|n|}(1+l)^{t_3}(1+|s+l|)^{r_3+t_2}\lVert w^{(j)} \lVert_{r_3+r_2}\nonumber\\
	\leq&\sum_{0\leq l\leq\Delta_{w^{(i)}}+\Delta_{w^{(j)}}-s-1}C_7(1+|s|)^{r_3+t_2}(1+l)^{|n|+t_3+r_3+t_2}\lVert w^{(j)} \lVert_{r_3+r_2}\nonumber\\
	\leq& C_8(1+|s|)^{r_3+t_2}(1+\Delta_{w^{(j)}}+|s|)^{1+|n|+t_3+r_3+t_2}\lVert w^{(j)} \lVert_{r_3+r_2}\nonumber\\
	\leq&C_8(1+|s|)^{2r_3+2t_2+1+|n|+t_3}(1+\Delta_{w^{(j)}})^{1+|n|+t_3+r_3+t_2}\lVert w^{(j)} \lVert_{r_3+r_2}\nonumber\\
	=&C_8(1+|s|)^{2r_3+2t_2+1+|n|+t_3}\lVert w^{(j)} \lVert_{2r_3+r_2+t_2+1+|n|+t_3}.
	\end{align}
	The energy-boundedness of $\mathcal Y_\alpha(Y_i(u,n)w^{(i)},x)$ follows from these two inequalities.
\end{proof}

The following proposition is also very useful. One can prove it using the argument in \cite{BSM} section 2.
\begin{pp}\label{lb89}
	If $v=\nu$ or $v\in V(1)$, then for any unitary $V$-module $W_i$, $Y_i(v,x)$ satisfies linear energy bounds.
\end{pp}

We summarize the results  in this section as follows:
\begin{co}\label{lb107}
Let  $W_i,W_j,W_k$ be unitary $V$-modules, and  $\mathcal Y_\alpha\in\mathcal V{k\choose i~j}$.\\
(a) Suppose that $V$ is generated by a set $E$ of homogeneous vectors. If for each $v\in E$, $Y_i(v,x)$ is energy-bounded, then $Y_i$ is energy-bounded.\\
(b) If $W_i$ is irreducible,  $Y_j,Y_k$ are energy-bounded, and there exists a nonzero homogeneous vector $w^{(i)}\in W_i$ such that $\mathcal Y_\alpha(w^{(i)},x)$ is energy-bounded, then $\mathcal Y_\alpha$ is energy-bounded.\\
(c) If $w^{(i)}\in W_i$ is homogeneous, and $\mathcal Y_\alpha(w^{(i)},x)$ is energy-bounded, then $\mathcal Y_{C^{\pm1}\alpha}(w^{(i)},x)$, $\mathcal Y_{\overline \alpha}(\overline {w^{(i)}},x)$, and $\mathcal Y_{\alpha^*}(\overline{w^{(i)}},x)$ are energy-bounded.\\
(d) If $w^{(i)}\in W_i$ is quasi-primary, and $\mathcal Y_\alpha(w^{(i)},x)$ satisfies $r$-th order energy bounds. Then  $\mathcal Y_{C^{\pm1}\alpha}(w^{(i)},x)$, $\mathcal Y_{\overline \alpha}(\overline {w^{(i)}},x)$, and $\mathcal Y_{\alpha^*}(\overline{w^{(i)}},x)$  satisfy $r$-th order energy bounds.
\end{co}
\begin{proof}
(a) and (b) follow from proposition \ref{lb55}. (c) follows from propositions \ref{lb94}, \ref{lb56}, \ref{lb55}, and \ref{lb89}. (d) follows from propositions \ref{lb94} and \ref{lb56}.
\end{proof}

\subsection{Smeared intertwining operators}\label{lb101}

In this section, we construct smeared intertwining operators for energy-bounded intertwining operators, and prove the adjoint relation, the braid relations, the rotation covariance, and the intertwining property for these operators. The proof of the last property requires some knowledge of the strong commutativity of unbounded closed operators on a Hilbert space. We give a brief 
exposition of this theory in chapter \ref{lba1}.

\subsubsection*{The unbounded operator $\mathcal Y_\alpha(w^{(i)},f)$}

For any open subset $I$ of $S^1$, we   denote by $C^\infty_c(I)$ the set of all complex smooth functions on $S^1$ whose supports lie in $I$. If $I=\{e^{it}:a< t< b \}$ ($a,b\in\mathbb R,a<b<a+2\pi$), we say that $I$ is an \textbf{open interval} of $S^1$. We let $\mathcal J$ be the set of all open intervals of $S^1$. In general, if $U$ is an open subset of $S^1$, we let $\mathcal J(U)$ be the set of open intervals of $S^1$ contained in $U$. If $I\in\mathcal J$, then its \textbf{complement} $I^c$ is defined to be $S^1\setminus \overline I$.  If $I_1, I_2\in\mathcal J$, we write $I_1\subset\joinrel\subset I_2$ if $\overline {I_1}\subset I_2$.

Let  $\mathcal Y_\alpha \in\mathcal V{k\choose i~j}$ be unitary. (Recall that this means that $W_i,W_j,W_k$ are unitary $V$-modules.) For any $w^{(i)}\in W_i,z\in\mathbb C^\times$, $\mathcal Y_\alpha(w^{(i)},z)$ is a linear map $W_j\rightarrow\widehat W_k$.  Therefore we can regard $\mathcal Y_\alpha(w^{(i)},z)$ as a sesquilinear form $W_j\times W_k\rightarrow \mathbb C, (w^{(j)}, w^{(k)})\mapsto \langle \mathcal Y_\alpha(w^{(i)},z)w^{(j)}|w^{(k)}\rangle$.

We now define the smeared intertwining operators. Let $\di\theta=  \frac {e^{i\theta}}{2\pi}d\theta$.  For any  $f\in C_c^\infty(S^1\setminus\{-1\})$, we define a sesquilinear form
\begin{gather*}
\mathcal Y_\alpha(w^{(i)},f):W_j\times W_k\rightarrow \mathbb C,~~~
(w^{(j)}, w^{(k)})\mapsto \langle \mathcal Y_\alpha(w^{(i)},f)w^{(j)}|w^{(k)} \rangle
\end{gather*}
satisfying
\begin{align}
\langle \mathcal Y_\alpha(w^{(i)},f)w^{(j)}|w^{(k)} \rangle=\int_{-\pi}^{\pi}\langle \mathcal Y_\alpha(w^{(i)},e^{i\theta})w^{(j)}|w^{(k)} \rangle f(e^{i\theta})\di\theta.\label{eq218}
\end{align}
$\mathcal Y_\alpha(w^{(i)},f)$ can be regarded as a linear map  $W_j\rightarrow\widehat W_k$. In the following, we show that when  $\mathcal Y_\alpha(w^{(i)},x)$ is energy-bounded,  $\mathcal Y_\alpha(w^{(i)},f)$ is a preclosed unbounded operator.

To begin with, we note that for any $f\in C_c^\infty(S^1\setminus\{-1\})$ and any $s\in\mathbb R$,  the  $s$-th mode of $f$ is
\begin{align}
\widehat f(s)=\int_{-\pi}^{\pi}f(e^{i\theta})e^{-is\theta}\cdot\frac {d\theta}{2\pi}.
\end{align}
Then we have
\begin{align}
\mathcal Y_\alpha(w^{(i)},f)=\sum_{s\in\mathbb R}\mathcal Y_\alpha(w^{(i)},s)\widehat f(s).
\end{align}
Define
\begin{gather*}
\mathcal D_V=\{\Delta_i+\Delta_j-\Delta_k:\text{ $W_i,W_j,W_k$ are irreducible $V$-modules} \},\\
\mathbb Z_V=\mathbb Z+\mathcal D_V.
\end{gather*}
Then $\mathcal Y_\alpha(w^{(i)},s)=0$ except possibly when $s\in\mathbb Z_V$. Since $V$ has finitely many equivalence classes of irreducible representations, the set $\mathcal D_V$ is finite.  Now for any $t\in\mathbb R$ we define a norm $|\cdot|_{V,t}$ on $C^\infty_c(S^1\setminus\{-1\})$ to be
\begin{align}
|f|_{V,t}=\sum_{s\in\mathbb Z_V}(1+|s|)^t|\widehat f(s)|,
\end{align}
which is easily seen to be finite. For each $r\in\mathbb R$, we define $e_r:S^1\setminus\{-1\}\rightarrow \mathbb C$ to be
\begin{align}
e_r(e^{i\theta})=e^{ir\theta},~~~(-\pi<\theta<\pi).
\end{align}
When $r\in\mathbb Z$, we regard $e_r$ as a continuous function on $S^1$.

\begin{lm}\label{lb57}
Suppose that $w^{(i)}\in W_i$ is homogeneous, and  $\mathcal Y_\alpha(w^{(i)},x)$ is energy-bounded and satisfies condition \eqref{eq211}.\\
(a) Let $p\in\mathbb R$. Then there exists $M_p\geq0$, such that for any  $f\in C^\infty_c(S^1\setminus\{-1\}),w^{(j)}\in W_j$, we have $\mathcal Y_\alpha(w^{(i)},f)w^{(j)}\in\mathcal H^\infty_k$, and 
\begin{align}
\big\lVert \mathcal Y_\alpha(w^{(i)},f)w^{(j)} \big\lVert_p\leq M_p\big|f\big|_{V,|p|+t}\big\lVert w^{(j)} \big\lVert_{p+r}.\label{eq214}
\end{align}
(b) For any $w^{(j)}\in W_j,w^{(k)}\in W_k$ we have
\begin{align}
\langle w^{(k)}| \mathcal Y_\alpha(w^{(i)},f)w^{(j)} \rangle=\sum_{m\geq0}\frac {e^{-i\pi\Delta_{w^{(i)}}}} {m!} \langle \mathcal Y_{\alpha^*}(\overline {L_1^mw^{(i)}},\overline {e_{(m+2-2\Delta_{w^{(i)}})}f}) w^{(k)}|w^{(j)}\rangle.
\end{align}
\end{lm}
\begin{proof}
(a) We have
\begin{align}
\mathcal Y_\alpha(w^{(i)},f)w^{(j)}=\sum_{s\in\mathbb Z_V}\widehat f(s)\mathcal Y_\alpha(w^{(i)},s)w^{(j)}.
\end{align}
Choose $M_p\geq0$ such that \eqref{eq295} always holds. Then
\begin{align}
&\sum_{s\in\mathbb Z_V}\big\lVert \widehat f(s)\mathcal Y_\alpha(w^{(i)},s)w^{(j)} \big\lVert_p\nonumber\\
\leq& \sum_{s\in\mathbb Z_V}M_p \big|\widehat f(s)\big|(1+|s|)^{|p|+t}\lVert w^{(j)} \lVert_{p+r}\nonumber\\
=& M_p\big|f\big|_{V,|p|+t}\big\lVert w^{(j)} \big\lVert_{p+r}.
\end{align}
In particular, $\mathcal Y_\alpha(w^{(i)},f)w^{(j)}\in\mathcal H_k^\infty$.

(b) For any  $w^{(j)}\in W_j,w^{(k)}\in W_k$, and $z\in\mathbb C^\times$ with argument $\arg z$, we have
\begin{align}
&\langle w^{(k)}| \mathcal Y_\alpha(w^{(i)},z)w^{(j)} \rangle\nonumber\\
=&\langle \mathcal Y_{\alpha^*}(e^{\overline zL_1}(e^{-i\pi}\overline z^{-2})^{L_0}\overline {w^{(i)}},\overline z^{-1})w^{(k)}|w^{(j)} \rangle\nonumber\\
=&e^{-i\pi\Delta_{w^{(i)}}}\sum_{m\geq0} \frac {\overline z^{m-2\Delta_{w^{(i)}} }}{m!}\langle \mathcal Y_{\alpha^*}(\overline {L_1^mw^{(i)}},\overline z^{-1}) w^{(k)} |w^{(j)}\rangle.
\end{align}
Note also that $\overline{\di\theta}=e^{-2i\theta}\di\theta$. Therefore we have
\begin{align}
&\langle w^{(k)}| \mathcal Y_\alpha(w^{(i)},f)w^{(j)} \rangle\nonumber\\
=&\int_{-\pi}^{\pi}\langle w^{(k)}|\mathcal Y_\alpha(w^{(i)},e^{i\theta})w^{(j)}\rangle \overline {f(e^{i\theta})\di\theta}\nonumber\\
=&\sum_{m\geq0}\int_{-\pi}^{\pi}\frac {e^{-i\pi\Delta_{w^{(i)}}}} {m!}\langle \mathcal Y_{\alpha^*}(\overline {L_1^mw^{(i)}},e^{i\theta}) w^{(k)}|w^{(j)}\rangle e^{-i(m+2-2\Delta_{w^{(i)}})\theta}\overline {f(e^{i\theta})}\di\theta\nonumber\\
=&\sum_{m\geq0}\frac {e^{-i\pi\Delta_{w^{(i)}}}} {m!} \langle \mathcal Y_{\alpha^*}(\overline {L_1^mw^{(i)}},\overline {e_{(m+2-2\Delta_{w^{(i)}})}f}) w^{(k)}|w^{(j)}\rangle.
\end{align}
\end{proof}

By lemma \ref{lb57}, if $w^{(i)}$ is homogeneous and $\mathcal Y_\alpha(w^{(i)},x)$ is energy-bounded, then $\mathcal Y_\alpha(w^{(i)},f)$ can be viewed as an unbounded operator from $\mathcal H_j$ to $\mathcal H_k$ with domain $W_j$. Moreover, the domain of $\mathcal Y_\alpha(w^{(i)},f)^*$  contains a dense subspace of $\mathcal H_k$ (which is $W_k$). So \emph{$\mathcal Y_\alpha(w^{(i)},f)$  is preclosed}.  We let $\overline{\mathcal Y_\alpha(w^{(i)},f)}$ be its closure . By inequality \eqref{eq214},
$\mathcal H^\infty_j$ is inside $\mathscr D(\overline{\mathcal Y_\alpha(w^{(i)},f)})$, the domain of $\overline{\mathcal Y_\alpha(w^{(i)},f)}$, and  $\overline{\mathcal Y_\alpha(w^{(i)},f)}\mathcal H^\infty_j\subset\mathcal H^\infty_k$. In the following, we will always view $\mathcal Y_\alpha(w^{(i)},f):\mathcal H^\infty_j\rightarrow\mathcal H^\infty_k$ as the restriction of $\overline{\mathcal Y_\alpha(w^{(i)},f)}$ to $\mathcal H_j^\infty$. $\mathcal Y_\alpha(w^{(i)},f)$ is called a \textbf{smeared intertwining operator}. The closed operator $\mathcal Y_\alpha(w^{(i)},f)^*=\overline{\mathcal Y_\alpha(w^{(i)},f)}^*$ is the adjoint of $\mathcal Y_\alpha(w^{(i)},f)$. The \textbf{formal adjoint} of $\mathcal Y_\alpha(w^{(i)},f)$, which is denoted by $\mathcal Y_\alpha(w^{(i)},f)^\dagger$, is the restriction of $\mathcal Y_\alpha(w^{(i)},f)^*$ to $\mathcal H^\infty_k\rightarrow \mathcal H^\infty_j$.

The following proposition follows directly from lemma \ref{lb57}.
\begin{pp}\label{lb59}
Suppose that $w^{(i)}\in W_i$ is  homogeneous, $\mathcal Y_\alpha(w^{(i)},x)$ is energy-bounded and satisfies condition \eqref{eq211}. Then for any $f\in C_c(S^1\setminus\{-1\})$, the following statements are true:\\
(a) $\mathcal Y_\alpha(w^{(i)},f)\mathcal H^\infty_j\subset \mathcal H^\infty_k$. Moreover, for any $p\in\mathbb R$, there exists $M_p\geq0$ independent of $f$, such that for any $\xi^{(j)}\in\mathcal H^\infty_j$, we have
\begin{align}
\big\lVert \mathcal Y_\alpha(w^{(i)},f)\xi^{(j)} \big\lVert_p\leq M_p\big|f\big|_{V,|p|+t}\big\lVert \xi^{(j)} \big\lVert_{p+r}.
\end{align}
(b) $\mathcal Y_\alpha(w^{(i)},f):\mathcal H^\infty_j\rightarrow \mathcal H^\infty_k$ has the formal adjoint $\mathcal Y_\alpha(w^{(i)},f)^\dagger:\mathcal H^\infty_k\rightarrow \mathcal H^\infty_j$, which satisfies
\begin{align}
\mathcal Y_\alpha(w^{(i)},f)^\dagger=\sum_{m\geq0}\frac {e^{-i\pi\Delta_{w^{(i)}}}} {m!}\mathcal Y_{\alpha^*}(\overline{L_1^mw^{(i)}},\overline {e_{(m+2-2\Delta_{w^{(i)}})}f}).\label{eq220}
\end{align}
In particular, if $w^{(i)}$ is quasi-primary, then we have the adjoint relation
\begin{align}
\mathcal Y_\alpha(w^{(i)},f)^\dagger= {e^{-i\pi\Delta_{w^{(i)}}}} \mathcal Y_{\alpha^*}(\overline{w^{(i)}},\overline {e_{(2-2\Delta_{w^{(i)}})}f}).\label{eq300}
\end{align}
\end{pp}
Hence the adjoint relation \eqref{eq300} for smeared intertwining operators is established.
\begin{rem}
	If $\mathcal Y_\alpha\in\mathcal V{k\choose i~j}$ is a unitary energy-bounded intertwining operator of $V$, $w^{(i)}\in W_i$ is not necessarily homogeneous, and $f\in C_c^\infty(S^1\setminus\{-1\})$, then by linearity, we can  define a preclosed operator $\mathcal Y_\alpha(w^{(i)},f):\mathcal H^\infty_j\rightarrow \mathcal H^\infty_k$ to be $\mathcal Y_\alpha(w^{(i)},f)=\sum_{s\in\mathbb R}\mathcal Y_\alpha(P_sw^{(i)},f)$. Proposition \ref{lb59}-(a) still holds in this case.
\end{rem}

\begin{rem}
If $W_i$ is a unitary $V$-module, then $Y_i\in\mathcal V{i\choose 0~i}$. Choose any  vector $v\in V$. Since the powers of $x$ in $Y(v,x)$ are integers, for each $z\in\mathbb C^\times$, $Y_i(v,z)$ does not depend on $\arg z$. Therefore, for any $f\in C^\infty_c(S^1)$, we can defined a smeared vertex operator $Y_i(v,f):\mathcal H^\infty_i\rightarrow \mathcal H^\infty_i$  using \eqref{eq218}. 
\end{rem}

\subsubsection*{Braiding of smeared intertwining operators}
The relation between products of smeared intertwining operators and correlation functions is indicated as follows. 
\begin{pp}\label{lb58}
Let $\mathcal Y_{\alpha_1},\mathcal Y_{\alpha_2},\dots,\mathcal Y_{\alpha_n}$ be a chain of unitary energy-bounded intertwining operators of $V$ with charge spaces $W_{i_1},W_{i_2},\dots,W_{i_n}$ respectively. Let $W_j$ be the source space of $\mathcal Y_{\alpha_1}$, and let $W_k$ be the target space of $\mathcal Y_{\alpha_n}$. Choose mutually disjoint  $I_1,I_2,\dots,I_n\in\mathcal J(S^1\setminus\{-1\})$. For each $m=1,2,\dots,n$ we choose  $w^{(i_m)}\in W_{i_m}$ and $f_m\in C^\infty_c(I_m)$. Then for any $w^{(j)}\in W_j$ and $w^{(k)}\in W_k$,
\begin{align}
&\langle \mathcal Y_{\alpha_n}(w^{(i_n)},f_n)\cdots \mathcal Y_{\alpha_1}(w^{(i_1)},f_1)w^{(j)}|w^{(k)} \rangle\nonumber\\
=&\int_{-\pi}^{\pi}\cdots\int_{-\pi}^{\pi} \langle \mathcal Y_{\alpha_n}(w^{(i_n)},e^{i\theta_n})\cdots \mathcal Y_{\alpha_1}(w^{(i_1)},e^{i\theta_1})w^{(j)}|w^{(k)} \rangle f_1(e^{i\theta_1})\cdots f_n(e^{i\theta_n})\cdot\di\theta_1\cdots\di\theta_n.\label{eq215}
\end{align}
\end{pp}

\begin{proof}	
\begin{align}
&\sum_{s_1,\dots,s_n\in\mathbb R}\big\lVert P_{s_n}\mathcal Y_{\alpha_n}(w^{(i_n)},f_n)P_{s_{n-1}}\mathcal Y_{\alpha_{n-1}}(w^{(i_{n-1})},f_{n-1})P_{s_{n-2}}\cdots P_{s_1}\mathcal Y_{\alpha_1}(w^{(i_1)},f_1)w^{(j)} \big\lVert\nonumber\\
=&\sum_{t_1,\dots,t_n\in\mathbb Z_V}\big\lVert \mathcal Y_{\alpha_n}(w^{(i_n)},t_n)\cdots \mathcal Y_{\alpha_1}(w^{(i_1)},t_1) w^{(j)}\big\lVert\cdot\big|\widehat f_1(t_1)\cdots\widehat f_n(t_n)\big|,
\end{align}	
which, by proposition \ref{lb53}, is finite. Hence, for all $r_1,\dots,r_n,r_1/r_2,\dots,r_{n-1}/r_n\in[1/2,1]$, the following functions of $s_1,\dots,s_n$:
\begin{align}
&\Big|\langle P_{s_n} \mathcal Y_{\alpha_n}(w^{(i_n)},f_n)P_{s_{n-1}}\cdots P_{s_1}\mathcal Y_{\alpha_1}(w^{(i_1)},f_1)w^{(j)}|w^{(k)} \rangle\nonumber\\
&~\cdot r_1^{-\Delta_{w^{(i_1)}}}\cdots r_n^{-\Delta_{w^{(i_n)}}}\Big(\frac {r_1}{r_2}\Big)^{s_1}\cdots\Big(\frac {r_{n-1}}{r_n}\Big)^{s_{n-1}}r_n^{\Delta_{w^{(k)}}}\Big|
\end{align}
are bounded by a constant multiplied by
\begin{align}
\Big|\langle P_{s_n} \mathcal Y_{\alpha_n}(w^{(i_n)},f_n)P_{s_{n-1}}\cdots P_{s_1}\mathcal Y_{\alpha_1}(w^{(i_1)},f_1)w^{(j)}|w^{(k)} \rangle\Big|,
\end{align}
the sum of which over $s_1,\dots,s_n$ is finite. Therefore, if we always assume that $r_1,\dots,r_n>0$ and $0<r_1/r_2<\cdots<r_{n-1}/r_n\leq 1$, then by dominated convergence theorem and relation \eqref{eq231},
\begin{align}
&\langle \mathcal Y_{\alpha_n}(w^{(i_n)},f_n)\cdots \mathcal Y_{\alpha_1}(w^{(i_1)},f_1)w^{(j)}|w^{(k)} \rangle\nonumber\\
=&\sum_{s_1,\dots,s_n\in\mathbb R}\langle P_{s_n}\mathcal Y_{\alpha_n}(w^{(i_n)},f_n)P_{s_{n-1}}\cdots P_{s_1}\mathcal Y_{\alpha_1}(w^{(i_1)},f_1)w^{(j)}|w^{(k)}\rangle\nonumber\\
=&\sum_{s_1,\dots,s_n\in\mathbb R}\lim_{r_1,\dots,r_n\rightarrow 1}\Big(\langle P_{s_n}\mathcal Y_{\alpha_n}(w^{(i_n)},f_n)P_{s_{n-1}}\cdots P_{s_1}\mathcal Y_{\alpha_1}(w^{(i_1)},f_1)w^{(j)}|w^{(k)}\rangle\nonumber\\
&~\cdot r_1^{-\Delta_{w^{(i_1)}}}\cdots r_n^{-\Delta_{w^{(i_n)}}}\Big(\frac {r_1}{r_2}\Big)^{s_1}\cdots\Big(\frac {r_{n-1}}{r_n}\Big)^{s_{n-1}}r_n^{\Delta_{w^{(k)}}}\Big)\nonumber\\
=&\lim_{r_1,\dots,r_n\rightarrow 1}\sum_{s_1,\dots,s_n\in\mathbb R}\Big(\langle P_{s_n}\mathcal Y_{\alpha_n}(w^{(i_n)},f_n)P_{s_{n-1}}\cdots P_{s_1}\mathcal Y_{\alpha_1}(w^{(i_1)},f_1)w^{(j)}|w^{(k)}\rangle\nonumber\\
&~\cdot r_1^{-\Delta_{w^{(i_1)}}}\cdots r_n^{-\Delta_{w^{(i_n)}}}\Big(\frac {r_1}{r_2}\Big)^{s_1}\cdots\Big(\frac {r_{n-1}}{r_n}\Big)^{s_{n-1}}r_n^{\Delta_{w^{(k)}}}\Big)\nonumber\\
=&\lim_{r_1,\dots,r_n\rightarrow 1}\sum_{s_1,\dots,s_n\in\mathbb R}\int_{-\pi}^{\pi}\cdots\int_{-\pi}^{\pi}\langle P_{s_n}\mathcal Y_{\alpha_n}(w^{(i_n)},e^{i\theta_n})P_{s_{n-1}}\nonumber\\
&~\cdots P_{s_1}\mathcal Y_{\alpha_1}(w^{(i_1)},e^{i\theta_1})w^{(j)}|w^{(k)}\rangle  r_1^{-\Delta_{w^{(i_1)}}}\cdots r_n^{-\Delta_{w^{(i_n)}}}\nonumber\\
&~\cdot\Big(\frac {r_1}{r_2}\Big)^{s_1}\cdots\Big(\frac {r_{n-1}}{r_n}\Big)^{s_{n-1}}r_n^{\Delta_{w^{(k)}}}f_1(e^{i\theta_1})\cdots f_n(e^{i\theta_n})\di\theta_1\cdots\di\theta_n\nonumber\\
=&\lim_{r_1,\dots,r_n\rightarrow 1}\sum_{s_1,\dots,s_n\in\mathbb R}\int_{-\pi}^{\pi}\cdots\int_{-\pi}^{\pi}\langle P_{s_n}\mathcal Y_{\alpha_n}(w^{(i_n)},r_ne^{i\theta_n})P_{s_{n-1}}\nonumber\\
&~\cdots P_{s_1}\mathcal Y_{\alpha_1}(w^{(i_1)},r_1e^{i\theta_1})w^{(j)}|w^{(k)}\rangle f_1(e^{i\theta_1})\cdots f_n(e^{i\theta_n})\di\theta_1\cdots\di\theta_n.\label{eq296}
\end{align}
By theorem \ref{lb65} and the discussion below, the sum and the integrals in \eqref{eq296} commute. Therefore \eqref{eq296} equals
\begin{align}
&\lim_{r_1,\dots,r_n\rightarrow 1}\int_{-\pi}^{\pi}\cdots\int_{-\pi}^{\pi}\sum_{s_1,\dots,s_n\in\mathbb R}\langle P_{s_n}\mathcal Y_{\alpha_n}(w^{(i_n)},r_ne^{i\theta_n})P_{s_{n-1}}\nonumber\\
&~\cdots P_{s_1}\mathcal Y_{\alpha_1}(w^{(i_1)},r_1e^{i\theta_1})w^{(j)}|w^{(k)}\rangle f_1(e^{i\theta_1})\cdots f_n(e^{i\theta_n})\di\theta_1\cdots\di\theta_n\nonumber\\
=&\lim_{r_1,\dots,r_n\rightarrow 1}\int_{-\pi}^{\pi}\cdots\int_{-\pi}^{\pi}\langle \mathcal Y_{\alpha_n}(w^{(i_n)},r_ne^{i\theta_n})\nonumber\\
&~\cdots \mathcal Y_{\alpha_1}(w^{(i_1)},r_1e^{i\theta_1})w^{(j)}|w^{(k)}\rangle f_1(e^{i\theta_1})\cdots f_n(e^{i\theta_n})\di\theta_1\cdots\di\theta_n.\label{eq297}
\end{align}
By continuity of correlation functions, the limit and the integrals in \eqref{eq297} commute. So \eqref{eq297} equals the right hand side of equation \eqref{eq215}. Thus the proof is completed.
\end{proof}

\begin{co}\label{lb91}
Let $\mathcal Y_\alpha,\mathcal Y_{\alpha'}$ be unitary energy-bounded intertwining operators of $V$ with common charge space $W_i$,  and $\mathcal Y_\beta,\mathcal Y_{\beta'}$ be unitary energy-bounded intertwining operators of $V$ with common charge space $W_j$. Choose $z_i,z_j\in S^1$ and assume that $\arg z_j<\arg z_i<\arg z_j+2\pi$. Choose disjoint open intervals $I,J\in\mathcal J(S^1\setminus\{-1\})$ such that $I$ is anticlockwise to $J$. Suppose that for any $w^{(i)}\in W_i,w^{(j)}\in W_j$, the following braid relation holds:
\begin{equation}
\mathcal Y_\alpha(w^{(i)},z_i)\mathcal Y_\beta(w^{(j)},z_j)=\mathcal Y_{\beta'}(w^{(j)},z_j)\mathcal Y_{\alpha'}(w^{(i)},z_i).
\end{equation}
Then for any $f\in C^\infty_c(I),g\in C^\infty_c(J)$, we have the braid relation for intertwining operators:
\begin{equation}
\mathcal Y_\alpha(w^{(i)},f)\mathcal Y_\beta(w^{(j)},g)=\mathcal Y_{\beta'}(w^{(j)},g)\mathcal Y_{\alpha'}(w^{(i)},f).\label{eq298}
\end{equation}
\end{co}
Note that if $W_k$ is the source space of $\mathcal Y_\beta$, then both sides of equation \eqref{eq298} are understood to be acting on $\mathcal H^\infty_k$.
\begin{rem}\label{lb92}
If $\mathcal Y_\alpha$ and $\mathcal Y_{\alpha'}$ 
(resp. $\mathcal Y_\beta$ and $\mathcal Y_{\beta'}$) are the vertex operator $Y_k$, then the above corollary still holds if we assume that $I\in\mathcal J$ (resp. $J\in\mathcal J$).
\end{rem}

\subsubsection*{Rotation covariance of smeared intertwining operators}

For each $t\in\mathbb R$, we define an action
\begin{align}
\mathfrak r(t):S^1\rightarrow S^1,~~~\mathfrak r(t)(e^{i\theta})=e^{i(\theta+t)}.
\end{align}
For any $g\in C^\infty_c(S^1)$, we let
\begin{align}
\mathfrak r(t)g=g\circ \mathfrak r(-t).
\end{align}
Therefore, if $J\in\mathcal J$, then $\mathfrak r(t)C^\infty_c(J)=C^\infty_c(\mathfrak r(t)J)$. We also define $g'\in C^\infty_c(S^1)$ to be
\begin{align}
g'(e^{i\theta})=\frac d {d\theta}g(e^{i\theta})
\end{align}
Rotation covariance is stated as follows.
\begin{pp}\label{lb96}
	Suppose that $\mathcal Y_\alpha\in\mathcal V{k\choose i~j}$ is unitary, $w^{(i)}\in W_j$ is homogeneous, $\mathcal Y(w^{(i)},x)$ is energy bounded, and $J\in\mathcal J(S^1\setminus\{-1\})$. Choose $\varepsilon>0$ such that $\mathfrak r(t)J\subset S^1\setminus\{-1\}$ for any $t\in(-\varepsilon,\varepsilon)$. Then for any $g\in C^\infty_c(J)$ and $t\in(-\varepsilon,\varepsilon)$, the following equations hold when both sides  act  on $\mathcal H^\infty_j$:
	\begin{gather}
	[\overline {L_0},\mathcal Y_\alpha(w^{(i)},g)]=\mathcal Y_\alpha\big(w^{(i)},(\Delta_{w^{(i)}}-1)g+ig'\big),\label{eq222}\\
	e^{it{\overline {L_0}}}\mathcal Y_\alpha(w^{(i)},g)e^{-it\overline {L_0}}=\mathcal Y_\alpha\big(w^{(i)},e^{i(\Delta_{w^{(i)}}-1)t}\mathfrak r(t)g\big).\label{eq223}
	\end{gather}
\end{pp}

\begin{proof}
	By equation \eqref{eqa2}, for any $z=e^{i\theta}\in J$ we have
	\begin{align*}
	&[\overline {L_0},\mathcal Y_\alpha(w^{(i)},z)]\nonumber\\
	=&\Delta_{w^{(i)}}\mathcal Y_\alpha(w^{(i)},z)+z\partial_{z}\mathcal Y_\alpha(w^{(i)},z)\\
	=&\Delta_{w^{(i)}}\mathcal Y_\alpha(w^{(i)},e^{i\theta})-i\partial_{\theta}\mathcal Y_\alpha(w^{(i)},e^{i\theta})
	\end{align*}
	when evaluated between vectors inside $W_j$ and $W_k$. Thus we have
	\begin{align*}
	&[\overline {L_0},\mathcal Y_\alpha(w^{(i)},g)]=\int^{\pi}_{-\pi}[\overline {L_0},\mathcal Y_\alpha(w^{(i)},e^{i\theta})]g(e^{i\theta})\di\theta\\
	=&\int^{\pi}_{-\pi}\big(\Delta_{w^{(i)}}\mathcal Y_\alpha(w^{(i)},e^{i\theta})-i\partial_{\theta}\mathcal Y_\alpha(w^{(i)},e^{i\theta})\big)g(e^{i\theta})\di\theta\\
	=&\Delta_{w^{(i)}}\mathcal Y_\alpha(w^{(i)},g)-i\int^{\pi}_{-\pi}\partial_{\theta}\mathcal Y_\alpha(w^{(i)},e^{i\theta})g(e^{i\theta})\frac{e^{i\theta}}{2\pi}d\theta\\
	=&\Delta_{w^{(i)}}\mathcal Y_\alpha(w^{(i)},g)+i\int^{\pi}_{-\pi}\mathcal Y_\alpha(w^{(i)},e^{i\theta})\frac{d}{d\theta}\Big(g(e^{i\theta})\frac{e^{i\theta}}{2\pi}\Big)d\theta\\
	=&\Delta_{w^{(i)}}\mathcal Y_\alpha(w^{(i)},g)+i\int^{\pi}_{-\pi}\mathcal Y_\alpha(w^{(i)},e^{i\theta})\big(g'(e^{i\theta})+ig(e^{i\theta})\big)\frac{e^{i\theta}}{2\pi}d\theta\\
	=&(\Delta_{w^{(i)}}-1)\mathcal Y_\alpha(w^{(i)},g)+i\mathcal Y_\alpha(w^{(i)},g').
	\end{align*}
	This proves the first equation. To prove the second one, we first note that for any $\tau\geq0$, when  $h\in\mathbb R$ is small enough,  the $|\cdot|_{V,\tau}$-norm of the function
	\begin{align*}
	&e^{i(\Delta_{w^{(i)}}-1)(t+h)}r(t+h)g-e^{i(\Delta_{w^{(i)}}-1)t}\mathfrak r(t)g\\
	-&\big(i(\Delta_{w^{(i)}}-1)e^{i(\Delta_{w^{(i)}}-1)t}\mathfrak r(t)g-e^{i(\Delta_{w^{(i)}}-1)t}\mathfrak r(t)g'\big)h
	\end{align*}
	is  $o(h)$. For any $\xi^{(j)}\in\mathcal H^\infty_j$, we define a function $\Xi(t)$ for $|t|<\varepsilon$ to be
	\begin{align*}
	\Xi(t)=e^{-it\overline {L_0}}\mathcal Y_\alpha(w^{(i)},e^{i(\Delta_{w^{(i)}}-1)t}\mathfrak r(t)g)e^{it\overline {L_0}}\xi^{(j)}.
	\end{align*}
	Now we can apply relation \eqref{eq222} and proposition \ref{lb59} to see that the vector norm of $\Xi(t+h)-\Xi(t)$ is $o(h)$ for any $|t|<\varepsilon$. (In fact this is true for any Sobolev norm.) This shows that the derivative of $\Xi(t)$ exists and equals $0$. So $\Xi(t)$ is a constant function. In particular, we have $\Xi(0)=\Xi(t)$, which  implies 
	\eqref{eq223}.
\end{proof}

\subsubsection*{The strong intertwining property of smeared intertwining operators}

\begin{pp}\label{lb95}
Let $\mathcal Y_\alpha\in\mathcal V{k\choose i~j}$ be unitary,  $w^{(i)}\in W_i$ be homogeneous, and $v\in V$ be quasi-primary.  Suppose that $\theta v=v$, $\mathcal Y_\alpha(w^{(i)},x)$ is energy bounded, and $Y_j(v,x),Y_k(v,x)$ satisfy linear energy bounds. Let $I\in\mathcal J,J\in\mathcal J(S^1\setminus\{-1\})$ be disjoint. Choose $f\in C^\infty_c(I),g\in C^\infty_c(J)$. Assume that $f$ satisfies 
\begin{align}
e^{i\pi\Delta_v/2}e_{1-\Delta_v}f=\overline {e^{i\pi\Delta_v/2}e_{1-\Delta_v}f}.\label{eq299}
\end{align}
Then $Y_j(v,f)$ and $Y_k(v,f)$ are essentially self-adjoint, and for any $t\in\mathbb R$, we have
\begin{gather}
e^{it\overline{Y_j(v,f)}}\mathcal H^\infty_j\subset\mathcal H^\infty_j,~~~e^{it\overline{Y_k(v,f)}}\mathcal H^\infty_k\subset\mathcal H^\infty_k,\label{eqa1}\\
e^{it\overline{Y_k(v,f)}}\cdot\overline{\mathcal Y_\alpha(w^{(i)},g)}=\overline{\mathcal Y_\alpha(w^{(i)},g)}\cdot e^{it\overline{Y_j(v,f)}}.\label{eq221}
\end{gather} 
\end{pp}

\begin{proof}
Define the direct sum $V$-module $W_l=W_j\oplus^\perp W_k$ of $W_j$ and $W_k$. Then $\mathcal H_l$ is the norm completion of $W_l$, $\mathcal H_l^\infty$ is the dense subspace of smooth vectors, and $Y_l(v,f)=\diag\big(Y_j(v,f),Y_k(v,f)\big)$. By equations \eqref{eq299} and \eqref{eq300}, $Y_l(v,f)$ is symmetric (i.e., $Y_l(v,f)^\dagger=Y_l(v,f)$). Since $Y_l(v,x)$ satisfies linear energy bounds, by proposition \ref{lb59}-(a), relation \eqref{eq222}, and lemma \ref{lb52},  $Y_l(v,f)$ is essentially self-adjoint, and $e^{it\overline{Y_l(v,f)}}\mathcal H^\infty_l\subset\mathcal H^\infty_l$. This is equivalent to saying that $Y_j(v,f)$ and $Y_k(v,f)$ are essentially self-adjoint, and relation \eqref{eqa1} holds.

Let $A=Y_l(v,f)$. Regard $B=\mathcal Y_\alpha(w^{(i)},g)$ as an unbounded operator on $\mathcal H_l$, being the original one when acting on $\mathcal H_j$, and zero when acting on $\mathcal H_k$. (So the domain of $B$ is $\mathcal H_j^\infty\oplus^\perp\mathcal H_k$.) By propositions \ref{lb80}, \ref{lb91}, and remark \ref{lb92},  $AB=BA$ when both sides of the equation act on $\mathcal H^\infty_l$. By theorem \ref{lb60}, $\overline A$ commutes strongly with $\overline B$. Therefore $e^{it\overline A}\cdot\overline B=\overline B\cdot e^{it\overline A}$, which is equivalent to equation \eqref{eq221}.
\end{proof}

\appendix

\section{Appendix for chapter \ref{lb77}}
\subsection{Uniqueness of formal series expansions}

Using Cauchy's integral formula, the coefficients of a Laurent series $\sum_{n\geq N}a_nz^n$ are determined by the values of this series when $z$ is near $0$. This uniqueness property can be generalized to formal series, as we now see.

Let $\mathscr G_0$ be a \emph{finite} subset of $\mathbb R$, and let $\mathscr G=\mathscr G_0+\mathbb Z_{\geq0}=\{\mu+m:\mu\in\mathscr G_0,m\in\mathbb Z_{\geq0} \}$. It is clear that the series
\begin{align}
f(z_1,\dots,z_n)=\sum_{\mu_1,\dots,\mu_n\in\mathscr G}c_{\mu_1,\dots,\mu_n}z_1^{\mu_1}\cdots z_n^{\mu_n}\label{eq242}
\end{align}
converges absolutely if and only if for any $\mu_1,\dots,\mu_n\in\mathscr G_0$, the power series
$$\sum_{m_1,\dots,m_n\in\mathbb Z_{\geq0}}c_{\mu_1+m_1,\dots,\mu_n+m_n}z_1^{\mu_1+m_1}\cdots z_n^{\mu_n+m_n}$$ converges absolutely. Hence, by root test, if $f(z_1,\dots,z_n)$ converges absolutely for some $z_1,\dots,z_n\neq0$, then $f(\zeta_1,\dots,\zeta_n)$ converges absolutely whenever $0<|\zeta_1|<|z_1|,\dots,0<|\zeta_n|<|z_n|$.

 The uniqueness property is stated as follows:
\begin{pp}\label{lb70}
	Let $r_1,\dots,r_n>0$. For any $1\leq l\leq n$, we choose a sequence of complex numbers $\{z_l(m_l):0<|z_l(m_l)|<r_l \}_{m_l\in\mathbb Z_{>0}}$ such that $\lim_{m_l\rightarrow \infty}z_l(m_l)=0$.	 Suppose that \eqref{eq242} converges absolutely when $0<|z_1|<r_1,\dots,0<|z_n|<r_n$, and that for any $m_1,\dots,m_n$, we have $f\big(z_1(m_1),\dots,z_n(m_n)\big)=0$. Then for any $\mu_1,\dots,\mu_n\in\mathscr G$, the coefficient $c_{\mu_1,\dots,\mu_n}=0$.
\end{pp}
\begin{proof}(cf. \cite{H 4} section 15.4)
By induction, it suffices to prove the case when $n=1$.  Then the series can be written as $f(z)=\sum_{k\in\mathbb Z_{\geq1}}c_{\mu_k}z^{\mu_k}$, where $\mu_{k+1}>\mu_k$ for any $k$, and we have a sequence of complex values $\{z_m \}$ converging to zero, on which the values of $f$ vanish. Define a series $g(z)=\sum_{k\in\mathbb Z_{\geq2}}c_{\mu_k}z^{\mu_k-\mu_2}$. Then the series $g(z)$ converges absolutely when $0<|z|<r$, and $\limsup_{z\rightarrow 0}|g(z)|<+\infty$ .  Since $f(z)z^{-\mu_1}=c_{\mu_1}+z^{\mu_2-\mu_1}g(z)$, we have $c_{\mu_1}=\lim_{m\rightarrow\infty}f\big(z
(m)\big)z(m)^{-\mu_1}=0$. This proves that $c_{\mu_1}=0$. Repeat the same argument, we see that $c_{\mu_k}=0$ for any $k$.
\end{proof}

\subsection{Linear independence of products of intertwining operators}\label{lb68}

This  section is devoted to the proof of proposition \ref{lb66}. First, we need the following lemma, the proof of which is an easy exercise.
\begin{lm}\label{lb67}
	Let $W_i$ be an irreducible $V$-module. Let $n=1,2,\dots$. Consider  the $V$-module $W_i^{\oplus n}=\underbrace{W_i\oplus W_i\oplus\cdots\oplus W_i}_n$. Then for any $V$-module homomorphism $R:W_i\rightarrow W_i^{\oplus n}$, there exist complex numbers $\lambda_1,\dots,\lambda_n$ such that
	\begin{align}
	R(w^{(i)})=(\lambda_1w^{(i)},\lambda_2w^{(i)},\dots,\lambda_nw^{(i)})\qquad(w^{(i)}\in V).\label{eq239}
	\end{align}
\end{lm}
\begin{proof}
	For any $1\leq m\leq n$, let $p_m$ be the projection of $W_i^{\oplus n}$ onto its $m$-th component. Then $p_mR\in\End_V(W_i)$. Since $W_i$ is irreducible, there exists $\lambda_m\in\mathbb C$ such that $p_mR=\lambda_m\id_{W_i}$. \eqref{eq239} now follows immediately.
\end{proof}	

Let $W_i,W_j$ be two $V$-modules. For any $k\in\mathcal E$ we choose a basis $\{\mathcal Y_\alpha:\alpha\in\Theta^k_{ij} \}$ of $\mathcal V{k\choose i~j}$. Consider the $V$-module $W_l=\bigoplus_{k\in\mathcal E}\big(\bigoplus_{\alpha\in\Theta^k_{ij}}W_k^\alpha\big)$, where each $W^\alpha_k$ is a $V$-module equivalent to $W_k$. It's contragredient module is $W_{\overline l}=\bigoplus_{k\in\mathcal E}\big(\bigoplus_{\alpha\in\Theta^k_{ij}}W_{\overline k}^\alpha\big)$, where $W_{\overline k}^\alpha$ is the contragredient module of $W_k^\alpha$.  Consider a type $l\choose i~j$ intertwining operator $\mathcal Y$ defined as follows: for any $w^{(i)}\in W_i,w^{(j)}\in W_j$, we let
\begin{align}
\mathcal Y(w^{(i)},x)w^{(j)}=\bigoplus_{k\in\mathcal E}\bigg(\bigoplus_{\alpha\in\Theta^k_{ij}}\mathcal Y_\alpha(w^{(i)},x)w^{(j)}\bigg),
\end{align}
i.e., the projection of $\mathcal Y(w^{(i)},x)w^{(j)}$ to $W^\alpha_k$ is $\mathcal Y_\alpha(w^{(i)},x)w^{(j)}$.

The following property is due to Huang. See \cite{H 4} lemma 14.9. The notations and terminologies in that article are different from ours, so we include a proof here.
\begin{pp}\label{lb98}
Choose $z\in\mathbb C^\times$ with  argument $\arg z$. Let $w^{(\overline l)}\in W_{\overline l}$. If for any $w^{(i)}\in W_i,w^{(j)}\in W_j$, we have
	\begin{align}
	\langle w^{(\overline l)},\mathcal Y(w^{(i)},z)w^{(j)}  \rangle=0,\label{eq240}
	\end{align}
	then $w^{(\overline l)}=0$.
\end{pp}

\begin{proof}
	Let $W_1$ be the subspace of all $w^{(\overline l)}\in W_{\overline l}$ satisfying \eqref{eq240}. We show that $W_1=0$. 
	
	Note that by relation \eqref{eq106},  for any $u\in V,m\in\mathbb Z$ we have
	\begin{align}
	Y_l(u,m)\mathcal Y(w^{(i)},z)-\mathcal Y(w^{(i)},z)Y_k(u,m)=\sum_{h\in\mathbb Z_{\geq0}}{m\choose h}\mathcal Y(Y_i(u,h)w^{(i)},z)z^{m-h}.
	\end{align}
From this we see that $W_1$ is a $V$-submodule of $W_{\overline l}$. If $W_1\neq0$, then $W_1$ contains an irreducible submodule equivalent to $W_{\overline k}$ for some $k\in\mathcal E$. This implies that we have a non-zero $V$-module homomorphism $R:W_{\overline k}\rightarrow \bigoplus_{\alpha\in\Theta^k_{ij}}W_{\overline k}^\alpha\subset W_{\overline l}$, and that the image of $R$ is inside $W_1$.
	
	By lemma \ref{lb67}, we can choose complex numbers $\{\lambda_\alpha:\alpha\in\Theta^k_{ij} \}$, not all of which are zero, such that for any $w^{(\overline k)}$, $Rw^{(\overline k)}=\bigoplus_{\alpha\in\Theta^k_{ij}}\lambda_\alpha w^{(\overline k)}$. Hence for any $w^{(i)}\in W_i,w^{(j)}\in W_j,w^{(\overline k)}\in W_k$, we have
	\begin{align*}
	\sum_{\alpha\in\Theta^k_{ij}}\lambda_\alpha\langle w^{(\overline k)},\mathcal Y_\alpha(w^{(i)},z)w^{(j)}  \rangle=0.
	\end{align*}
	Since $3$-point correlation functions are determined by their values at the point $z$, we have
	\begin{align*}
	\sum_{\alpha\in\Theta^k_{ij}}\lambda_\alpha\langle w^{(\overline k)},\mathcal Y_\alpha(w^{(i)},x)w^{(j)}  \rangle=0,
	\end{align*}
	where $x$ is a formal variable. But we know that $\{\mathcal Y_\alpha:\alpha\in\Theta^k_{ij} \}$ are linearly independent, which forces all the coefficients $\lambda_\alpha$ to be  zero. Hence we have a contradiction.
\end{proof}

\begin{co}\label{lb69}
	Vectors of the form $\mathcal Y(w^{(i)},s)w^{(j)}$ ($w^{(i)}\in W_i,w^{(j)}\in W_j,s\in\mathbb R$) span the vector space $W_l$.
\end{co}
\begin{proof}
Choose any $w^{(\overline l)}\in W_{\overline l}$ satisfying that for any $w^{(i)}\in W_i,w^{(j)}\in W_j,s\in\mathbb R$,
	\begin{align}
	\langle w^{(\overline l)},\mathcal Y(w^{(i)},s)w^{(j)}  \rangle=0.
	\end{align}
Then for any $z\in\mathbb C^\times$, equation \eqref{eq240} holds. So $w^{(\overline l)}$ must be zero.
\end{proof}

\begin{proof}[Proof of proposition \ref{lb66}]
	It is clear that $\Phi$ is surjective. So we only need to prove that $\Phi$ is injective. By induction, it suffices to prove that the linear map $\Psi$:
	\begin{gather*}
	\bigoplus_{j\in\mathcal E}\Bigg( \mathcal V{k\choose{i_n~i_{n-1}~\cdots~i_2~j}} \otimes \mathcal V{j\choose i_1~i_0}   \Bigg)\rightarrow\mathcal V{k\choose{i_n~i_{n-1}~\cdots~i_1~i_0}},\\
	\mathcal X \otimes\mathcal Y_{\alpha}\mapsto\mathcal X \mathcal Y_{\alpha}
	\end{gather*}
	is injective. To prove this, we choose, for any $j\in\mathcal E$, a linear basis $\{\mathcal Y_{\alpha}:\alpha\in\Theta^j_{i_1i_0} \}$ of $\mathcal V{j\choose i_1~i_0}$. If we can prove, for any $j\in\mathcal E,\alpha\in\Theta^j_{i_1i_0},\mathcal X _\alpha\in\mathcal V{k\choose{i_n~i_{n-1}~\cdots~i_2~j}}$,  that
	\begin{align}
	\sum_{j\in\mathcal E}\sum_{\alpha\in\Theta^j_{i_1i_0}}\mathcal X _\alpha\mathcal Y_\alpha=0\label{eq241}
	\end{align}
	always implies that $\mathcal X _\alpha=0$ for all $\alpha$, then the injectivity of $\Psi$ follows immediately.
	
	Now suppose that \eqref{eq241} is true. Then  for any $w^{(i_0)}\in W_{i_0},w^{(i_1)}\in W_{i_1},\dots,w^{(i_n)}\in W_{i_n}, s\in\mathbb R$, and $z_2,\dots,z_n$ satisfying $0<|z_2|<\cdots<|z_n|$, we have, by proposition \ref{lb70},
	\begin{align}
	\sum_{j\in\mathcal E}\sum_{\alpha\in\Theta^j_{i_1i_0}}\mathcal X _\alpha(w^{(i_n)},\dots,w^{(i_2)};z_n,\dots,z_2)\mathcal Y_\alpha(w^{(i_1)},s)w^{(i_0)}=0.
	\end{align}
	By corollary \ref{lb69}, for any $j\in\mathcal E,w^{(j)}\in W_j$ and $\alpha\in\Theta^j_{i_1i_0}$, there exist $w^{(i_0)}_1,\dots,w^{(i_0)}_m\in W_{i_0},w^{(i_1)}_1,\dots,w^{(i_1)}_m\in W_{i_1},s_1,\dots,s_m\in\mathbb R$, such that
	$$\mathcal Y_\alpha(w^{(i_1)}_1,s_1)w^{(i_0)}_1+\cdots+\mathcal Y_\alpha(w^{(i_1)}_m,s_m)w^{(i_0)}_m=w^{(j)},$$
	and that for any $\beta\neq\alpha$,
	$$\mathcal Y_\beta(w^{(i_1)}_1,s_1)w^{(i_0)}_1+\cdots+\mathcal Y_\beta(w^{(i_1)}_m,s_m)w^{(i_0)}_m=0.$$
	Hence $\mathcal X _\alpha(w^{(i_n)},\dots,w^{(i_2)};z_n,\dots,z_2)w^{(j)}=0$.
\end{proof}

\subsection{General braiding and fusion relations}\label{lb71}
In this section, we prove all the results claimed in section \ref{lb72}. In the following proofs of absolute convergence, the idea of analytic continuation (lemma \ref{lb74}) and induction is due to \cite{HLZ} proposition 12.7. The trick in step 1 of the proof of theorem \ref{lb73} using $B_\pm$ to transform certain types of absolute convergence to other types can be found in \cite{H 4} proposition 14.1. What's new in our proofs is the change-of-variable trick: we replace the original complex variables $z_1,z_2,\dots$ with the moduli parameters $\omega_1,\omega_2,\dots$ which control the shape of gluing together  Riemann spheres with three holes (pants).

We first introduce some temporary notations. For any $r>0$, let $D(r)=\{z\in\mathbb C:|z|<r \},D^\times(r)=D(r)\setminus\{0\}$, and $E(r)=D(r)\cap(0,+\infty)$. Then we  have the following:
\begin{lm}\label{lb74}
	Given a  power series
	\begin{align}
	\sum_{n_0,n_1,\dots,n_l\in\mathbb Z_{\geq0}}c_{n_0n_1\dots n_l}z_0^{n_0}z_1^{n_1}\cdots z_l^{n_l}\label{eq244}
	\end{align}
	of the complex variables $z_0,z_1,\dots,z_l$, where $l\in\mathbb Z_{>0}$ and each $c_{n_0n_1\dots n_l}\in\mathbb C$. Suppose that there exist $r_0,r_1,\dots,r_l>0$, such that for any $n_0$, the power series
	\begin{align}
	g_{n_0}(z_1,\dots,z_l)=\sum_{n_1,\dots,n_l\in\mathbb Z_{\geq0}}c_{n_0n_1\dots n_l}z_1^{n_1}\cdots z_l^{n_l}
	\end{align}
	converges absolutely on $D(r_1)\times\cdots \times D(r_l)$;  that for any $z_1\in E(r_1),\dots,z_l\in E(r_l)$, 
	\begin{align}
	f(z_0,z_1,\dots,z_l)=\sum_{n_0\in\mathbb Z_{\geq0}}g_{n_0}(z_1,\dots,z_n)
	z_0^n,\label{eq243}
	\end{align}
converges absolutely as a power series of $z_0$ on $D(r_0)$; and that $f$ can be analytically continued to a multivalued holomorphic function on $D^\times(r_0)\times D^\times(r_1)\times\cdots \times D^\times(r_l)$. Then the power series \eqref{eq244} converges absolutely on $D(r_0)\times D(r_1)\times\cdots \times D(r_l)$.
\end{lm}
\begin{proof}
	Consider the multivalued holomorphic function $f$. From \eqref{eq243}, we know that for any $z_1\in E(r_1),\dots,z_l\in E(r_l)$, $f$ is single-valued for $z_0\in D^\times(r_0)$.
	So $f$ is single-valued on $z_0$ for any $z_1\in D^\times(r_1),\dots,z_l\in D^\times(r_l)$. 
	
	Now, for any $n_0\in\mathbb Z$,
	\begin{align}
	\widetilde g_{n_0}(z_1,\dots,z_n)=\oint_{0} f(z_0,z_1,\dots,z_l)z_0^{-n-1}\frac{d z_0}{2i\pi}
	\end{align}
	is a multivalued holomorphic function on $D^\times(r_1)\times\cdots \times D^\times(r_l)$. If $n_0\geq0$, then by \eqref{eq243}, we must have $\widetilde g_{n_0}=g_{n_0}$ on $E(r_1)\times\cdots\times E(r_l)$. Since $g_{n_0}$ is holomorphic, $\widetilde g_{n_0}=g_{n_0}$ on $D^\times(r_1)\times\cdots \times D^\times(r_l)$. Hence $\widetilde g_{n_0}$ is single-valued. Similarly, when $n_0<0$, we have $\widetilde g_{n_0}(z_1,\dots,z_n)=0$ on $E(r_1)\times\cdots\times E(r_l)$, and hence on $D^\times(r_1)\times\cdots \times D^\times(r_l)$. Therefore, $f(z_0,z_1,\dots,z_n)=\sum_{n_0\in\mathbb Z}\widetilde g_{n_0}(z_1,\dots,z_n)z_0^{n_0}$ is single-valued on $D^\times(r_0)\times D^\times (r_1)\times \cdots\times D^\times(r_n)$, and the Laurant series expansion of $f$ near the origin has no negative powers of $z_0,z_1,\dots,z_n$. So $f$ is a single-valued holomorphic function on $D(r_0)\times D(r_1)\times\cdots \times D(r_l)$ with power series expansion \eqref{eq244}. We can thus conclude that \eqref{eq244} converges absolutely on $D(r_0)\times D(r_1)\times\cdots \times D(r_l)$.
\end{proof}

Recall that a  series $f(z_1,\dots,z_n)=\sum_{s_1,\dots,s_n\in\mathbb R}c_{s_1\dots s_n}z_1^{s_1}\cdots z_n^{s_n}$ is called a \textbf{quasi power series} of $z_1,\dots,z_n$, if $f$ equals a power series multiplied by a monomial of $z_1,\dots,z_n$, i.e., if there exist $t_1,\dots,t_n\in\mathbb C$ such that $f(z_1,\dots,z_n)z_1^{t_1}\cdots z_n^{t_n}\in \mathbb C[[z_1,\dots,z_n]]$. 

\begin{proof}[Proof of theorem \ref{lb73}]
	Step 1. We first prove the convergence. Let $W_i$ be the charge space of $\mathcal Y_\gamma$. Then for any $w^{(i_0)}\in W_{i_0},w^{(i)}\in W_{i}$, we have
	\begin{align*}
	&\mathcal Y_\gamma(w^{(i)},x)w^{(i_0)}\\
	=&\mathcal Y_{B_+B_-\gamma}(w^{(i)},x)w^{(i_0)}\\
	=&e^{xL_{-1}}\mathcal Y_{B_-\gamma}(w^{(i_0)},e^{i\pi}x)w^{(i)},
	\end{align*}	
	where $x$ is a formal variable. Then for any $w^{(\overline{k})}\in W_{\overline{k}}$, we have
	\begin{align*}
	&\langle\mathcal Y_\gamma(w^{(i)},z_1)w^{(i_0)}, w^{(\overline{k})}\rangle
	\\
	=&\langle\mathcal Y_\gamma(w^{(i)},x)w^{(i_0)}, w^{(\overline{k})}\rangle\big|_{x=z_1}\\
	=&\langle e^{xL_{-1}} \mathcal Y_{B_-\gamma}(w^{(i_0)},e^{i\pi}x)w^{(i)}, w^{(\overline{k})}\rangle\big|_{x=z_1}\\
	=&\langle \mathcal Y_{B_-\gamma}(w^{(i_0)},e^{i\pi}x)w^{(i)}, e^{xL_{1}}w^{(\overline{k})}\rangle\big|_{x=z_1}\\
	=&\langle \mathcal Y_{B_-\gamma}(w^{(i_0)},e^{i\pi}z_1)w^{(i)}, e^{z_1L_{1}}w^{(\overline{k})}\rangle.
	\end{align*}
	Therefore,
	\begin{align}
	&\big\langle\mathcal Y_\gamma\big( P_{s_n}\mathcal Y_{\sigma_n}(w^{(i_n)},z_n-z_1) P_{s_{n-1}}\mathcal Y_{\sigma_{n-1}}(w^{(i_{n-1})},z_{n-1}-z_1)\nonumber\\
	&\qquad\qquad\cdots P_{s_2}\mathcal Y_{\sigma_2}(w^{(i_2)},z_2-z_1)w^{(i_1)},z_1\big)w^{(i_0)},w^{(\overline k)}\big\rangle\nonumber\\
	=&\langle\mathcal Y_{B_-\gamma}(w^{(i_0)},e^{i\pi}z_1) P_{s_n}\mathcal Y_{\sigma_n}(w^{(i_n)},z_n-z_1) P_{s_{n-1}}\mathcal Y_{\sigma_{n-1}}(w^{(i_{n-1})},z_{n-1}-z_1)\nonumber\\
	&\qquad\qquad\cdots P_{s_2}\mathcal Y_{\sigma_2}(w^{(i_2)},z_2-z_1)w^{(i_1)}, e^{z_1L_{1}}w^{(\overline{k})}\rangle.\label{eq247}
	\end{align}
	Hence, by theorem \ref{lb65} and the discussion below, the sum of  \eqref{eq247} over $s_2,s_3,\dots,s_n\in\mathbb R$ converges absolutely and locally uniformly.\\

Step 2. Assume that
\begin{gather}
0<|z_1|<|z_2|<\cdots<|z_n|,\nonumber\\
0<|z_2-z_1|<|z_3-z_1|\cdots<|z_n-z_1|<|z_1|,
\end{gather}
and choose arguments $\arg z_1,\arg z_2,\dots,\arg z_n,\arg(z_2-z_1),\dots,\arg(z_n-z_1)$. We prove, by induction on $n$, that  \eqref{eq245} defined near the point $(z_1,z_2,\dots,z_n)$ is a correlation function, i.e., it can be written as a product of a chain of intertwining operators. The case $n=2$ was proved in \cite{H 4} and \cite{H ODE}. Suppose  this theorem holds for  $n-1$, we now prove it for $n$. By analytic continuation, it suffices to assume also that
\begin{gather}
|z_1|+|z_2-z_1|<|z_3|.
\end{gather}

Let $W_{j_2}$ be the target space of $\mathcal Y_{\sigma_2}$. By induction, there exists a chain of intertwining operators $\mathcal Y_{\delta},\mathcal Y_{\alpha_3},\mathcal Y_{\alpha_4},\dots,\mathcal Y_{\alpha_n}$ with charge spaces $W_{j_2},W_{i_3},W_{i_4},\dots,W_{i_n}$ respectively, such that $W_{i_0}$ is the source space of $\mathcal Y_{\delta}$, that $W_k$ is the target space of $\mathcal Y_{\alpha_n}$, and that
	for any $w^{(i_0)}\in W_{i_0},w^{(j_2)}\in W_{j_2},w^{(i_3)}\in W_{i_3},w^{(i_4)}\in W_{i_4},\dots,w^{(i_n)}\in W_{i_n}$, we have the fusion relation
	\begin{align}
	&\mathcal Y_\gamma\big(\mathcal Y_{\sigma_n}(w^{(i_n)},z_n-z_1)\mathcal Y_{\sigma_{n-1}}(w^{(i_{n-1})},z_{n-1}-z_1)\cdots \mathcal Y_{\sigma_3}(w^{(i_3)},z_3-z_1)w^{(j_2)},z_1\big)w^{(i_0)}\nonumber\\
	=&\mathcal Y_{\alpha_n}(w^{(i_n)},z_n)\mathcal Y_{\alpha_{n-1}}(w^{(i_{n-1})},z_{n-1})\cdots\mathcal Y_{\alpha_3}(w^{(i_3)},z_3)\mathcal Y_{\delta}(w^{(j_2)},z_1)w^{(i_0)}
	\end{align}
near the point $(z_1,z_3,z_4,\dots,z_n)$.
	
	There also exists a chain of intertwining operator $\mathcal Y_{\alpha_1},\mathcal Y_{\alpha_2}$ with charge spaces $W_{i_1},W_{i_2}$, such that the source space of $\mathcal Y_{\alpha_1}$ is $W_{i_0}$, that the target space of $\mathcal Y_{\alpha_2}$ equals that of $\mathcal Y_\delta$, and that  the fusion relation
	\begin{align}
	\mathcal Y_\delta\big(\mathcal Y_{\sigma_2}(w^{(i_2)},z_2-z_1)w^{(i_1)},z_1 \big)=\mathcal Y_{\alpha_2}(w^{(i_2)},z_2)\mathcal Y_{\alpha_1}(w^{(i_1)},z_1)
	\end{align}
	holds near the point $(z_1,z_2)$. Now we compute, omitting the evaluation under any $w^{(\overline k)}\in W_{\overline k}$, that
	\begin{align}
	&\mathcal Y_\gamma\big(\mathcal Y_{\sigma_n}(w^{(i_n)},z_n-z_1)\mathcal Y_{\sigma_{n-1}}(w^{(i_{n-1})},z_{n-1}-z_1)\cdots\mathcal Y_{\sigma_2}(w^{(i_2)},z_2-z_1)w^{(i_1)},z_1\big)w^{(i_0)}\nonumber\\
	=&\sum_{s_1\in\mathbb R}\mathcal Y_\gamma\big(\mathcal Y_{\sigma_n}(w^{(i_n)},z_n-z_1)\mathcal Y_{\sigma_{n-1}}(w^{(i_{n-1})},z_{n-1}-z_1)\cdots P_{s_1}\mathcal Y_{\sigma_2}(w^{(i_2)},z_2-z_1)w^{(i_1)},z_1\big)w^{(i_0)}\nonumber\\
	=&\sum_{s_1\in\mathbb R}\mathcal Y_{\alpha_n}(w^{(i_n)},z_n)\mathcal Y_{\alpha_{n-1}}(w^{(i_{n-1})},z_{n-1})\cdots\mathcal Y_{\alpha_3}(w^{(i_3)},z_3)\nonumber\\
	&\qquad\cdot\mathcal Y_{\delta}\big(P_{s_1}\mathcal Y_{\sigma_2}(w^{(i_2)},z_2-z_1)w^{(i_1)},z_1\big)w^{(i_0)}\nonumber\\
	=&\sum_{s_1\in\mathbb R}\sum_{s_2,\dots,s_{n-1}\in\mathbb R}\mathcal Y_{\alpha_n}(w^{(i_n)},z_n)P_{s_{n-1}}\mathcal Y_{\alpha_{n-1}}(w^{(i_{n-1})},z_{n-1})P_{s_{n-2}}\nonumber\\
	&\qquad\cdots P_{s_3}\mathcal Y_{\alpha_3}(w^{(i_3)},z_3)P_{s_2}\mathcal Y_{\delta}\big(P_{s_1}\mathcal Y_{\sigma_2}(w^{(i_2)},z_2-z_1)w^{(i_1)},z_1\big)w^{(i_0)}.\label{eq248}
	\end{align}
	If we can prove, for any $w^{(\overline k)}\in W_{\overline k}$, and any $z_1,z_2,\dots,z_n$ satisfying
\begin{gather}
0<|z_2-z_1|<|z_1|<|z_3|<|z_4|<\cdots<|z_n|,\nonumber\\
|z_1|+|z_2-z_1|<|z_3|,\label{eq250}
\end{gather}	
that the expression
	\begin{align}
	&\langle \mathcal Y_{\alpha_n}(w^{(i_n)},z_n)\mathcal Y_{\alpha_{n-1}}(w^{(i_{n-1})},z_{n-1})\cdots\mathcal Y_{\alpha_3}(w^{(i_3)},z_3)\nonumber\\
	&~\cdot\mathcal Y_{\delta}\big(\mathcal Y_{\sigma_2}(w^{(i_2)},z_2-z_1)w^{(i_1)},z_1\big)w^{(i_0)},w^{(\overline k)}\rangle\label{eq252}
	\end{align}
	converges absolutely, i.e., the sum of the absolute values of
	\begin{align}
	&\big\langle P_{s_n}\mathcal Y_{\alpha_n}(w^{(i_n)},z_n)P_{s_{n-1}}\mathcal Y_{\alpha_{n-1}}(w^{(i_{n-1})},z_{n-1})P_{s_{n-2}}\nonumber\\
	&\cdots P_{s_3}\mathcal Y_{\alpha_3}(w^{(i_3)},z_3)P_{s_2}\mathcal Y_{\delta}\big(P_{s_1}\mathcal Y_{\sigma_2}(w^{(i_2)},z_2-z_1)w^{(i_1)},z_1\big)w^{(i_0)},w^{(\overline k)}\big\rangle\label{eq249}
	\end{align}
	over $s_1,s_2,\dots,s_n\in\mathbb R$ is a finite number, then the two sums on the right hand side of \eqref{eq248} commute. Hence \eqref{eq248} equals
	\begin{align}
	&\sum_{s_2,\dots,s_n\in\mathbb R}\sum_{s_1\in\mathbb R}P_{s_n}\mathcal Y_{\alpha_n}(w^{(i_n)},z_n)P_{s_{n-1}}\mathcal Y_{\alpha_{n-1}}(w^{(i_{n-1})},z_{n-1})P_{s_{n-2}}\nonumber\\
	&\qquad\cdots P_{s_3}\mathcal Y_{\alpha_3}(w^{(i_3)},z_3)P_{s_2}\mathcal Y_{\delta}\big(P_{s_1}\mathcal Y_{\sigma_2}(w^{(i_2)},z_2-z_1)w^{(i_1)},z_1\big)w^{(i_0)}\nonumber\\
	=&\sum_{s_2,\dots,s_n\in\mathbb R}\sum_{s_1\in\mathbb R}P_{s_n}\mathcal Y_{\alpha_n}(w^{(i_n)},z_n)P_{s_{n-1}}\mathcal Y_{\alpha_{n-1}}(w^{(i_{n-1})},z_{n-1})P_{s_{n-2}}\nonumber\\
	&\qquad\cdots P_{s_3}\mathcal Y_{\alpha_3}(w^{(i_3)},z_3)P_{s_2}\mathcal Y_{\alpha_2}(w^{(i_2)},z_2)P_{s_1}\mathcal Y_{\alpha_1}(w^{(i_1)},z_1)w^{(i_0)}\nonumber\\
	=&\mathcal Y_{\alpha_n}(w^{(i_n)},z_n)\mathcal Y_{\alpha_{n-1}}(w^{(i_{n-1})},z_{n-1})\cdots\mathcal Y_{\alpha_1}(w^{(i_1)},z_1)w^{(i_0)}.
	\end{align}
	Therefore, if the series \eqref{eq252} converges absolutely, then \eqref{eq245} defines an $(n+2)$-point correlation function of $V$. The converse statement (every $(n+2)$-point function can be written in the form \eqref{eq252}) can be proved in a similar way.\\
	
	Step 3. We show  that when \eqref{eq250} holds,  \eqref{eq252} converges absolutely. Assume, without loss of generality, that all the intertwining operators in \eqref{eq252} are irreducible, and that all the vectors in \eqref{eq252} are homogeneous. Define a new set of variables $\omega_1,\omega_2,\dots,\omega_n$ by setting
	\begin{gather*}
	z_m=\omega_{m}\omega_{m+1}\cdots\omega_n\quad(3\leq m\leq n),\\
	z_1=\omega_2\omega_3\cdots\omega_n,\\
	z_2-z_1=\omega_1\omega_2\cdots\omega_n.
	\end{gather*}
	Then condition \eqref{eq250} is equivalent to the condition
	\begin{gather}
	0<|\omega_m|<1\quad(1\leq m\leq n-1),\nonumber\\
	0<|\omega_n|,\nonumber\\
	|\omega_2|(1+|\omega_1|)<1.\label{eq251}
	\end{gather}
	It is clear that if $\mathring \omega_1,\mathring \omega_2,\dots\mathring \omega_n$ are complex numbers satisfying condition \eqref{eq251}, then there exist positive numbers $r_1>|\mathring \omega_1|,r_2>|\mathring \omega_2|,\dots,r_n>|\mathring \omega_n|$, such that whenever $0<|\omega_m|<r_m$ ($1\leq m\leq n$), condition \eqref{eq251} is satisfied. We now prove that the sum of \eqref{eq249} over $s_1,\dots,s_n$ converges absolutely  on $\{0<|\omega_1|<r_1,\dots,0<|\omega_n|<r_n \}$.
	
	Let
	\begin{align}
	&c_{s_1s_2\dots s_n}\nonumber\\
	=&\langle P_{s_n}\mathcal Y_{\alpha_n}(w^{(i_n)},1)P_{s_{n-1}}\mathcal Y_{\alpha_{n-1}}(w^{(i_{n-1})},1)P_{s_{n-2}}\cdots \nonumber\\
	&\cdot P_{s_3}\mathcal Y_{\alpha_3}(w^{(i_3)},1)P_{s_2}\mathcal Y_{\delta}\big(P_{s_1}\mathcal Y_{\sigma_2}(w^{(i_2)},1)w^{(i_1)},1\big)w^{(i_0)},w^{(\overline k)}\rangle,
	\end{align}
	where each $\mathcal Y_\cdot(\cdot,1)=\mathcal Y_\cdot(\cdot,x)\big|_{x=1}$. By relation \eqref{eq231}, it is easy to see that \eqref{eq249} equals
	\begin{align}
	&\langle P_{s_n}\omega_n^{L_0}\mathcal Y_{\alpha_n}(w^{(i_n)},1)P_{s_{n-1}}\omega_{n-1}^{L_0}\mathcal Y_{\alpha_{n-1}}(w^{(i_{n-1})},1)P_{s_{n-2}}\cdots \nonumber\\
	&\cdot P_{s_3}\omega_3^{L_0}\mathcal Y_{\alpha_3}(w^{(i_3)},1)P_{s_2}\omega_2^{L_0}\mathcal Y_{\delta}\big(P_{s_1}\omega_1^{L_0}\mathcal Y_{\sigma_2}(w^{(i_2)},1)w^{(i_1)},1\big)w^{(i_0)},w^{(\overline k)}\rangle\nonumber\\
	=&c_{s_1s_2\dots s_n}\omega_1^{s_1}\omega_2^{s_2}\cdots\omega_n^{s_n}\label{eq254}
	\end{align}
	multiplied by a monomial  $\omega_1^{r_1}\omega_2^{r_2}\cdots\omega_n^{r_n}$, where the powers $r_1,r_2,\dots,r_n\in\mathbb R$ are independent of $s_1,s_2,\dots,s_n$. Therefore, the absolute convergence of \eqref{eq252} is equivalent to the absolute convergence of the series
	\begin{align}
	\sum_{s_1,s_2,\dots,s_n\in\mathbb R}c_{s_1s_2\dots s_n}\omega_1^{s_1}\omega_2^{s_2}\cdots\omega_n^{s_n}\label{eq253}
	\end{align}
	on $\{0<|\omega_1|<r_1,0<|\omega_2|<r_2,\dots,0<|\omega_n|<r_n \}$. Note that by irreducibility of the intertwining operators, \eqref{eq253} is a quasi power series of $\omega_1,\omega_2,\dots,\omega_n$. So we are going to prove the absolute convergence of \eqref{eq253} by checking that \eqref{eq253} satisfies all the conditions in lemma \ref{lb74}.
	
	Since  \eqref{eq249} equals \eqref{eq254}  multiplied by  $\omega_1^{r_1}\omega_2^{r_2}\cdots\omega_n^{r_n}$, for each $s_2\in\mathbb R$, step 1 and theorem \ref{lb65} imply that the series
	\begin{align}
	\sum_{s_1,s_3,s_4,\dots,s_n\in\mathbb R}c_{s_1s_2s_3\dots s_n}\omega_1^{s_1}\omega_3^{s_3}\omega_4^{s_4}\cdots\omega_n^{s_n}
	\end{align}
	converges absolutely on $\{0<|\omega_1|<r_1,0<|\omega_3|<r_3,0<|\omega_4|<r_4,\dots,0<|\omega_n|<r_n \}$. If we assume moreover that $0<\omega_1<r_1$, then $0<|\omega_2|<r_2$ clearly  implies  $0<|z_1|<|z_2|<\cdots<|z_n|$ and $0<|z_2-z_1|<|z_1|$. Hence, the following quasi power series of $\omega_2$
	\begin{align}
	&\omega_1^{r_1}\omega_2^{r_2}\cdots\omega_n^{r_n}\cdot\bigg(\sum_{s_2\in\mathbb R}\bigg(\sum_{s_1,s_3,\dots,s_n\in\mathbb R}c_{s_1s_2s_3\dots s_n}\omega_1^{s_1}\omega_3^{s_3}\cdots\omega_n^{s_n}\bigg)\omega_2^{s_2}\bigg)\nonumber\\
	=&\sum_{s_2\in\mathbb R}\langle\mathcal Y_{\alpha_n}(w^{(i_n)},z_n)\mathcal Y_{\alpha_{n-1}}(w^{(i_{n-1})},z_{n-1})\cdots \mathcal Y_{\alpha_3}(w^{(i_3)},z_3)\nonumber\\
	&\qquad \cdot P_{s_2}\mathcal Y_{\delta}\big(\mathcal Y_{\sigma_2}(w^{(i_2)},z_2-z_1)w^{(i_1)},z_1\big)w^{(i_0)},w^{(\overline k)}\rangle\nonumber\\
	=&\sum_{s_2\in\mathbb R}\langle\mathcal Y_{\alpha_n}(w^{(i_n)},z_n)\mathcal Y_{\alpha_{n-1}}(w^{(i_{n-1})},z_{n-1})\cdots\mathcal Y_{\alpha_3}(w^{(i_3)},z_3)\nonumber\\
	&\qquad\cdot P_{s_2}\mathcal Y_{\alpha_2}(w^{(i_2)},z_2)\mathcal Y_{\alpha_1}(w^{(i_1)},z_1) w^{(i_0)},w^{(\overline k)}\rangle\label{eq255}
	\end{align}
	must converge absolutely on $\{0<|\omega_2|<r_2 \}$. By theorem \ref{lb75}, the function \eqref{eq255} defined on $\{0<\omega_1<r_1,0<|\omega_2|<r_2,\dots,0<|\omega_n|<r_n \}$ can be analytically continued to a multivalued holomorphic function on $\{0<|\omega_1|<r_1,0<|\omega_2|<r_2,\dots,0<|\omega_n|<r_n \}$. Hence by lemma \ref{lb74}, the quasi power series \eqref{eq253}  converges absolutely on $\{0<|\omega_1|<r_1,\dots,0<|\omega_n|<r_n \}$.
\end{proof}

\begin{proof}[Proof of theorem \ref{lb12}] 
	The argument here is similar to  step 3 of the proof of theorem \ref{lb73}. Assume, without loss of generality, that all the intertwining operators in \eqref{eq48} are irreducible, and all the vectors in it are homogeneous. We prove this theorem by induction on $m$. The case that $m=1$ is proved in theorem \ref{lb73}. Suppose that the theorem holds for $m-1$, we prove this for $m$. 
	
	Define a new set of variables $\{\omega^a_b:1\leq a\leq m,1\leq b\leq n_a \}$  in the following way: For any $1\leq a\leq m$, we set
	\begin{gather}
	z^a_1=\omega^a_1\omega^{a+1}_1\cdots \omega^m_1,\label{eq53}
	\end{gather}
	and if  $2\leq b\leq n_a$, we set
	\begin{gather}
	z^a_b-z^a_1=\omega^a_1\omega^{a+1}_1\cdots\omega^m_1\cdot\omega^a_b\omega^a_{b+1}\cdots\omega^a_{n_a}.\label{eq54}
	\end{gather}
	Then the condition (1) and (2) on $\{z^a_b:1\leq a\leq m,1\leq b\leq n_a \}$ is equivalent to the condition
	\begin{gather}
	0<|\omega^a_b|<1\quad(1\leq a\leq m,2\leq b\leq n_a),\nonumber\\
	0<|\omega^m_1|,\nonumber\\
	0<|\omega^a_1|\big(1+(1-\delta_{n_a,1})|\omega^a_{n_a}|\big)<1-\big(1-\delta_{n_{a+1},1}\big)|\omega^{a+1}_{n_{a+1}}|\quad(1\leq a\leq m-1).\label{eq256}
	\end{gather}	It is clear that if $\{\mathring\omega^a_b:1\leq a\leq m,1\leq b\leq n_a \}$ are complex numbers satisfying condition \eqref{eq256}, then there exist positive numbers $\{r^a_b>|\mathring\omega^a_b| \}$, such that whenever $0<|\omega^a_b|<r^a_b$ for all $a$ and $b$, then \eqref{eq256} is true. If, moreover, any $\omega^a_b$ except $\omega^1_1$ satisfies $0<\omega^a_b<r^a_b$, then condition (3) also also holds for $\{z^a_b:1\leq a\leq m,1\leq b\leq n_a \}$.
	
	Let $\vec s$ be the sequence $\{s^a_b\}$, $\vec \omega$ be $\{\omega^a_b\}$, $\vec s\setminus {s^1_1}$ be   $\{\text{all }s^a_b\text{ except }s^1_1\}$, and $\vec \omega\setminus {\omega^1_1}$ be   $\{\text{all }\omega^a_b\text{ except }\omega^1_1\}$. We let ${\vec\omega}^{\vec s}=\prod_{1\leq a\leq m,1\leq b\leq n_a}(\omega^a_b)^{s^a_b}$. For each $\vec s$, we define
	\begin{align}
	c_{\vec s}=\Big\langle \Big[\prod_{m\geq a\geq 1}P_{s^a_1}\mathcal Y_{\alpha^a}\Big(\Big(\prod_{n_a\geq b\geq 2}P_{s^a_b}\mathcal Y_{\alpha^a_b}(w^a_b,1)\Big) w^a_1,1\Big)\Big]w^i,w^{\overline k} \Big\rangle,
	\end{align}
	where each $\mathcal Y_\cdot(\cdot,1)$ means $\mathcal Y_\cdot(\cdot,x)|_{x=1}$. Then by \eqref{eq231}, the expression
	\begin{align}
	\Big\langle \Big[\prod_{m\geq a\geq 1}P_{s^a_1}\mathcal Y_{\alpha^a}\Big(\Big(\prod_{n_a\geq b\geq 2}P_{s^a_b}\mathcal Y_{\alpha^a_b}(w^a_b,z^a_b-z^a_1)\Big) w^a_1,z^a_1\Big)\Big]w^i,w^{\overline k} \Big\rangle
	\end{align}
	equals $c_{\vec s}\cdot{\vec\omega}^{\vec s}$ multiplied by a monomial of $\vec{\omega}$ whose power is independent of $\vec s$.
	By induction, we can show that for each $s^1_1\in\mathbb R$, the series $\sum_{\vec s\setminus {s^1_1}} c_{\vec s}\cdot{\vec\omega}^{\vec s}\cdot(\omega^1_1)^{-s^1_1}$ of $\vec \omega\setminus {\omega^1_1}$ converges absolutely on $\{\vec \omega\setminus {\omega^1_1}:0<|\omega^a_b|<r^a_b \}$;  that for all $\vec \omega\setminus {\omega^1_1}$ satisfying $0<\omega^a_b<r^a_b$, 
	\begin{align}
	\sum_{s^1_1\in\mathbb R}\sum_{\vec s\setminus {s^1_1}} c_{\vec s}\cdot{\vec\omega}^{\vec s},\label{eq257}
	\end{align}
	as a series of $\omega^1_1$, converges absolutely on $\{\omega^1_1:0<|\omega^1_1|<r^1_1 \}$; and that as a function of ${\vec{\omega}}$, \eqref{eq257} can be analytically continued to a multivalued holomorphic function on $\{\vec \omega:0<|\omega^a_b|<r^a_b \}$.  Hence, by lemma \ref{lb74}, the quasi power series $\sum_{\vec s} c_{\vec s}\cdot{\vec\omega}^{\vec s}$ converges absolutely on  $\{\vec \omega:0<|\omega^a_b|<r^a_b \}$. If, moreover, $\{z^a_b \}$ satisfy condition (3), then by induction and the argument in step 2 of the proof of theorem \ref{lb73}, \eqref{eq48} can be written as a product of a chain of intertwining operators. So it is a correlation function defined near $\{z^a_b \}$.
\end{proof}

\begin{proof}[Proof of corollary \ref{lb13}]
One can prove this corollary, either by theorem \ref{lb12} and the argument in step 1 of the proof of theorem \ref{lb73}, or by induction and the argument in step 3 of the proof of theorem \ref{lb73}. We leave the details to the reader. 
\end{proof}

\begin{proof}[Proof of proposition \ref{lb85}]
Fix $z_i\in\mathbb C^\times$. Let $w_1$ (resp. $w_2$) be a vector in the source space (resp. in the contragredient module of the target space) of $\mathcal Y_\delta$. Let $x_i,x_{ji},\widetilde x_{ji}$ be commuting independent formal variables. It is easy to check that for any $w^{(k)}\in W_k$,
\begin{align}
&\big\langle\mathcal Y_\delta\big(e^{\widetilde x_{ji}L_{-1}}w^{(k)},x_i \big)w_1,w_2\big\rangle=\big\langle\mathcal Y_\delta\big(w^{(k)},x_i+\widetilde x_{ji} \big)w_1,w_2\big\rangle\nonumber\\
:=&\sum_{s\in\mathbb R,l\in\mathbb Z_{\geq0}}\big\langle\mathcal Y_\delta(w^{(k)},s)w_1,w_2\big\rangle{-s-1\choose l}x_i^{-s-1-l}\widetilde x_{ji}^l.
\end{align}
Put $x_i=z_i$, we have
\begin{align}
&\big\langle\mathcal Y_\delta\big(e^{\widetilde x_{ji}L_{-1}}w^{(k)},z_i \big)w_1,w_2\big\rangle=\big\langle\mathcal Y_\delta\big(w^{(k)},z_i+\widetilde x_{ji} \big)w_1,w_2\big\rangle\nonumber\\
:=&\sum_{s\in\mathbb R,l\in\mathbb Z_{\geq0}}\big\langle\mathcal Y_\delta(w^{(k)},s)w_1,w_2\big\rangle{-s-1\choose l}z_i^{-s-1-l}\widetilde x_{ji}^l.
\end{align}	
Clearly
\begin{align}
\big\langle \mathcal Y_\delta\big(\mathcal Y_{\gamma}(w^{(i)},e^{\pm i\pi}z_{ji})w^{(j)},z_i+\widetilde z_{ji} \big)w_1,w_2\big\rangle\label{eq276}
\end{align}
is a multivalued holomorphic function of $z_{ji},\widetilde z_{ji}$ when $0<|z_{ji}|,|\widetilde z_{ji}|<\frac 1 2|z_i|$. Since the series
\begin{align}
\sum_{s\in\mathbb R}\big\langle \mathcal Y_\delta\big(P_s\mathcal Y_{\gamma}(w^{(i)},e^{\pm i\pi}z_{ji})w^{(j)},z_i+\widetilde z_{ji} \big)w_1,w_2\big\rangle
\end{align}
converges absolutely and locally uniformly, the infinite sum commutes with  Cauchy's integrals around the pole $\widetilde z_{ji}=0$. From this we see that \eqref{eq276}  has the  series expansion 
\begin{align}
\big\langle\mathcal Y_\delta\big(\mathcal Y_{\gamma}(w^{(i)},e^{\pm i\pi}x_{ji})w^{(j)},z_i+\widetilde x_{ji} \big)w_1,w_2\big\rangle\Big|_{x_{ji}=z_{ji},\widetilde x_{ji}=\widetilde z_{ji}},
\end{align}
which must be absolute convergent, and also equals
\begin{align}
\big\langle\mathcal Y_\delta\big(e^{\widetilde x_{ji}L_{-1}}\mathcal Y_{\gamma}(w^{(i)},e^{\pm i\pi}x_{ji})w^{(j)},z_i \big)w_1,w_2\big\rangle\Big|_{x_{ji}=z_{ji},\widetilde x_{ji}=\widetilde z_{ji}}.
\end{align}
Therefore, when $0<|z_j-z_i|<\frac 1 2|z_i|$, the series
\begin{align}
\sum_{r,s\in\mathbb R}\big\langle \mathcal Y_\delta\big(P_re^{(z_j-z_i)L_{-1}}P_s\mathcal Y_{\gamma}(w^{(i)},e^{\pm i\pi}(z_j-z_i))w^{(j)},z_i \big)w_1,w_2\big\rangle\label{eq277}
\end{align}
converges absolutely and equals \eqref{eq276} with $z_{ji}=\widetilde z_{ji}=z_j-z_i$.

One the other hand,
\begin{align}
&\big\langle \mathcal Y_\delta\big(\mathcal Y_{B_\pm\gamma}(w^{(j)},z_j-z_i)w^{(i)},z_i\big)  w_1,w_2\big\rangle\nonumber\\
=&\sum_{r\in\mathbb R}\big\langle \mathcal Y_\delta\big(P_r\mathcal Y_{B_\pm\gamma}(w^{(j)},z_j-z_i)w^{(i)},z_i\big)  w_1,w_2\big\rangle\nonumber\\
=&\sum_{r\in\mathbb R}\big\langle \mathcal Y_\delta\big(P_re^{(z_j-z_i)L_{-1}}\mathcal Y_{\gamma}(w^{(i)},e^{\pm i\pi}(z_j-z_i) )w^{(j)},z_i\big)  w_1,w_2\big\rangle\nonumber,
\end{align} 
which is just \eqref{eq277}. So it also equals \eqref{eq276} with $z_{ji}=\widetilde z_{ji}=z_j-z_i$. This proves relation \eqref{eq275} when $0<|z_j-z_i|<\frac 1 2|z_i|$. The general case follows from analytic continuation.
\end{proof}

\begin{proof}[Proof of theorem \ref{lb78}]
	
The case $n=2$ follows immediately from proposition \ref{lb85} and the fusion relations of two intertwining operators. We now prove the general case.	
	
Since $S_n$ is generated by adjacent transpositions, we can assume that $\varsigma$ exchanges $m,m+1$ and fixes the other elements in $\{1,2,\dots,n \}$. 	Write
\begin{gather*}
\mathcal X_1=\mathcal Y_{\alpha_{m-1}}(w^{(i_{m-1})},z_{m-1})\cdots\mathcal Y_{\alpha_1}(w^{(i_1)},z_1),\\
\mathcal X_2=\mathcal Y_{\alpha_n}(w^{(i_n)},z_n)\cdots\mathcal Y_{\alpha_{m+2}}(w^{(i_{m+2})},z_{m+2}).
\end{gather*}
To proof the braid relation in this case, it is equivalent to showing that if $0<|z_1|<\cdots<|z_{m-1}|<|z_{m+1}|<|z_m|<|z_{m+2}|<\cdots<|z_n|$, and if we move $z_m,z_{m+1}$ to satisfy $0<|z_1|<\cdots<|z_{m-1}|<|z_m|<|z_{m+1}|<|z_{m+2}|<\cdots<|z_n|$ by scaling the norms of $z_m,z_{m+1}$, then we can find intertwining operators $\mathcal Y_{\beta_m},\mathcal Y_{\beta_{m+1}}$ independent of the choice of vectors, such that
\begin{align}
\langle \mathcal X_2  \mathcal Y_{\alpha_m}(w^{(i_m)},z_m)\mathcal Y_{\alpha_{m+1}}(w^{(i_{m+1})},z_{m+1}) \mathcal X_1 w^{(i_0)},w^{(\overline k)}\rangle\label{eqa4}
\end{align}
can be analytically continued to
	\begin{align}
\langle\mathcal X_2 \mathcal Y_{\beta_{m+1}}(w^{(i_{m+1})},z_{m+1})  \mathcal Y_{\beta_m}(w^{(i_m)},z_m) \mathcal X_1 w^{(i_0)},w^{(\overline k)}\rangle.\label{eqa5}
\end{align}
By analytic continuation, we can also assume that during the process of moving $z_m,z_{m+1}$, conditions $0<|z_1|<\cdots<|z_{m-1}|<|z_m|,|z_{m+1}|<|z_{m+2}|<\cdots<|z_n|$ and $0<|z_m-z_{m+1}|<|z_{m+1}|$ are always satisfied.

	Let $W_{j_1}$ be the source space of $\mathcal Y_{\alpha_{m+1}}$ and $W_{j_2}$ be the target space of $\mathcal Y_{\alpha_m}$. By braiding of two intertwining operators, there exists a chain of intertwining operators $\mathcal Y_{\beta_m},\mathcal Y_{\beta_{m+1}}$ with charge spaces $W_{i_m},W_{i_{m+1}}$ respectively, such that the source space of $\mathcal Y_{\beta_m}$ is $W_{j_1}$, that the target space of $\mathcal Y_{\beta_{m+1}}$ is $W_{j_2}$, and that for any $w^{(j_1)}\in W_{j_1},w^{(i_m)}\in W_{i_m},w^{(i_{m+1})}\in W_{i_{m+1}}, w^{(\overline{j_2})}\in W_{\overline {j_2}}$, the expression
	\begin{align}
	\langle  \mathcal Y_{\alpha_m}(w^{(i_m)},z_m)\mathcal Y_{\alpha_{m+1}}(w^{(i_{m+1})},z_{m+1}) w^{(j_1)},w^{(\overline{j_2})}\rangle\label{eq260}
	\end{align}
	defined on $0<|z_{m+1}|<|z_m|$ can be analytically continued to
	\begin{align}
	\langle \mathcal Y_{\beta_{m+1}}(w^{(i_{m+1})},z_{m+1})  \mathcal Y_{\beta_m}(w^{(i_m)},z_m)w^{(j_1)},w^{(\overline{j_2})} \rangle\label{eq261}
	\end{align}
	defined on $0<|z_m|<|z_{m+1}|$ by scaling the norms of $z_m$ and $z_{m+1}$.
	
Now, by fusion of intertwining operators, there exist intertwining operators $\mathcal Y_\delta,\mathcal Y_\gamma$ with suitable charge spaces, source spaces, and target spaces, such that \eqref{eq260} equals
	\begin{align}
	\langle \mathcal Y_\delta\big(\mathcal Y_\gamma(w^{(i_m)},z_m-z_{m+1})w^{(i_{m+1})},z_{m+1} \big)w^{(j_1)},w^{(\overline{j_2})} \rangle \label{eq262}
	\end{align}
	when $|z_{m+1}|<|z_m|$. Then \eqref{eq261} equals \eqref{eq262} when $|z_m|<|z_{m+1}|$. By theorem \ref{lb12}, the expression
	\begin{align}
	\langle \mathcal X_2  \mathcal Y_\delta\big(\mathcal Y_\gamma(w^{(i_m)},z_m-z_{m+1})w^{(i_{m+1})},z_{m+1} \big)\mathcal X_1 w^{(i_0)},w^{(\overline k)}\rangle
	\end{align}
	converges absolutely and locally uniformly. Hence it is a locally defined holomorphic function when $0<|z_1|<\cdots<|z_{m-1}|<|z_m|,|z_{m+1}|<|z_{m+2}|<\cdots<|z_n|$. Therefore \eqref{eqa4}	can be analytically continued to \eqref{eqa5}	from $\{0<|z_{m+1}|<|z_m| \}$ to $\{0<|z_m|<|z_{m+1}| \}$.
\end{proof}

\section{Appendix for chapter \ref{lb90}}\label{lba1}

\subsection{von Neumann algebras generated by closed operators}\label{lb62}
Let $A$ be a (densely defined) unbounded operator on $\mathcal H$ with domain $\mathscr D(A)$. Choose $x\in B(\mathcal H)$, i.e., let $x$ be a bounded operator on $\mathcal H$.  Recall that the notation $xA\subset Ax$ means that $x\mathscr D(A)\subset\mathscr D(A)$, and $xA\xi=Ax\xi$ for any $\xi\in\mathscr D(A)$. The following proposition is easy to show.
\begin{pp}\label{lb97}
Let $A$ be a preclosed operator on $\mathcal H$ with closure $\overline A$.\\
(1) If  $x\in B(\mathcal H)$ and $xA\subset Ax$, then we have $x^*A^*\subset A^*x^*$ and $x\overline A\subset\overline Ax$. \\
(2) If $A$ is closed, then the set of all $x\in B(\mathcal H)$ satisfying $xA\subset Ax$ form a strongly closed subalgebra of $B(\mathcal H)$.
\end{pp}
\begin{proof}
If $xA\subset Ax$ then $(Ax)^*\subset(xA)^*$. Recall that in general, if $A,B$ are two densely defined unbounded operators on $\mathcal H$, and if $AB$ has dense domain, then $B^*A^*\subset(AB)^*$. If $A$ is bounded, then $B^*A^*=(AB)^*$. Thus we have $x^*A^*\subset (Ax)^*\subset (xA)^*=A^*x^*$. Apply this relation to $x^*,A^*$, and note that $A^{**}=\overline A$, then we have $x\overline A\subset \overline Ax$. This proves part (1). Part (2) is a routine check.
\end{proof}

\begin{df}
Let $A$ be a closed operator on a Hilbert space $\mathcal H$ with domain $\mathscr D(A)$, and let $x\in B(\mathcal H)$. We say that $A$ and $x$ \textbf{commute strongly}\footnote{Our definition follows \cite{Neu16} chapter XIV, in which the strong commutativity of an unbounded  operator with a bounded one is called adjoint commutativity.}, if the following relations hold:
\begin{gather}
xA\subset Ax,~~~x^*A\subset Ax^*.
\end{gather}
\end{df}
\begin{co}
Suppose that $\mathfrak S$ is a collection of closed operators on $\mathcal H$. We define its \textbf{commutant} $\mathfrak S'$ to be the set of all bounded operators on $\mathcal H$ which commute strongly with any element of $\mathfrak S$. Then $\mathfrak S'$ is a von Neumann algebra. It's double commutant $\mathfrak S''$, which is the commutant of $\mathfrak S'$, is called the \textbf{von Neumann algebra generated by $\mathfrak S$}.
\end{co}

\begin{lm}\label{lb50}
Suppose that $A$ is a closed operator on $\mathcal H$, and $v\in B(\mathcal H)$ is a unitary operator. Let $A=uH$ (resp. $Hu$) be the left (resp. right) polar decomposition of $A$, such that $u$ the partial isometry and $H$ the self adjoint opertor.  Then the following conditions are equivalent:
\begin{flalign}
&\text{(a) $v$ commutes strongly with $A$.}&\nonumber\\
&\text{(b) }vA=Av.&\\
&\text{(c) }[u,v]=0,\text{ and }[e^{itH},v]=0\text{ for any }t\in\mathbb R.&\label{eq203}
\end{flalign}
\end{lm}
\begin{proof}
We prove this for the left polar decomposition. The other case  can be proved in the same way.

(a)$\Rightarrow$(b): Since $v$ commutes strongly with $A$, we have $vA\subset Av$ and $v^{-1}A\subset Av^{-1}$. Therefore, $v\mathscr D(A)\subset \mathscr D(A)$ and $v^{-1}\mathscr D(A)\subset \mathscr D(A)$. So we must have $v\mathscr D(A)= \mathscr D(A)$, and hence $vA=Av$.

(b)$\Rightarrow$(a): If $vA=Av$, then $vAv^{-1}=A$. So $Av^{-1}=v^{-1}A$, which proves  (a).

(b)$\Rightarrow$(c): We have $vAv^{-1}=A$. Thus by uniqueness of  left polar decompositions, we have $vuv^{-1}=u$ and $vHv^{-1}=H$. Hence for any $t\in\mathbb R$ we have $$ve^{itH}v^{-1}=e^{i v(tH)v^{-1}}=e^{itH}.$$ This proves (c).

(c)$\Rightarrow$(b): Suppose that we have  \eqref{eq203}. Then $vuv^{-1}=u$ and $ve^{itH}v^{-1}=e^{itH}.$ On the other hand, we always have $ve^{itH}v^{-1}=e^{itvHv^{-1}}$ in general. So  $vHv^{-1}$ and $H$ are both generators of the one parameter unitary group $ve^{itH}v^{-1}$. Hence we must have $vHv^{-1}=H$. This implies that $vA=Av$. Therefore (b) is true.
\end{proof}

\begin{pp}\label{lb51}
Let $\mathfrak S$ be a set of closed operators on $\mathcal H$. For each $A\in\mathfrak S$, we either let $A=u_AH_A$ be the left polar decomposition of $A$,  or  let $A=H_Au_A$ be the right polar decomposition  of $A$. Then $\mathfrak S''$ is the von Neumann algebra generated by  the bounded operators $\{u_A,e^{itH_A}:t\in\mathbb R,A\in\mathfrak S \}$.
\end{pp}
\begin{proof}
Let $\mathcal M$ be the von Neumann algebras generated by those $u_A$ and $e^{itH_A}$. We  show that $\mathcal M=\mathfrak S''$.

Let $\mathcal U(\mathfrak S')$ be the set of unitary operators in $\mathfrak S'$. We know that $\mathcal U(\mathfrak S')$ generates $\mathfrak S'$. So $\mathfrak S''=\mathcal U(\mathfrak S')'$. By lemma \ref{lb50} (a)$\Rightarrow$(c) we see that $\mathcal M$ commutes with $\mathcal U(\mathfrak S')$. Hence $\mathcal M\subset\mathcal U(\mathfrak S')'=\mathfrak S''$.

Let $\mathcal U(\mathcal M')$ be the set of unitary operators in $\mathcal M'$, the commutant of $\mathcal M$. Then by lemma \ref{lb50} (c)$\Rightarrow$(a) we also have $\mathcal U(\mathcal M')\subset \mathfrak S'$. Hence $\mathcal M'\subset \mathfrak S'$, which implies that $\mathcal M\supset \mathfrak S''$. Thus we've proved that $\mathcal M=\mathfrak S''$.
\end{proof}

\begin{co}
Assume that $A$ is a closed operator on $\mathcal H$ and $x\in B(\mathcal H)$. Let $A=uH$ (resp. $Hu$) be the left (resp. right) polar decomposition of $A$ with $u$ the partial isometry and $H$ the self adjoint opertor. Then $x$ commutes strongly with $A$ if and only if $[u,x]=0$ and $[e^{itH},x]=0$ for any $t\in\mathbb R$.
\end{co}
\begin{proof}
Let $\mathfrak S=\{A\}$. Then by proposition \ref{lb51}, $\mathfrak S''$ is generated by $u$ and all $e^{itH}$. Thus $x\in\mathfrak S'$ if and only if $x$ commutes with $u$ and all $e^{itH}$.
\end{proof}
\begin{df}
Let $A$ and $B$ be two closed operators on a Hilbert space $\mathcal H$. We say that $A$ and $B$ \textbf{commute strongly}, if the von Neumann algebra generated by $A$ commutes with the one generated by $B$.
\end{df}

If $\mathcal M$ is a von Neumann algebra on $\mathcal H$ and $A$ is a closed operator on $\mathcal H$. We say that $A$ is \textbf{affiliated with} $\mathcal M$, if the von Neumann algebra generated by the single operator $A$ is inside $\mathcal M$. Now suppose that $\mathcal N$ is another von Neumann algebra on a Hilbert space $\mathcal K$, and $\pi:\mathcal M\rightarrow \mathcal N$ is a normal (i.e. $\sigma-$weakly continuous) unital *-homomorphism. We define $\pi(A)$ to be a closed operator on $\mathcal K$ affiliated with $\mathcal N$ in the following way: Let $A=uH$ be its left polar decomposition. Define $\pi(H)$ to be the generator of the one parameter unitary group $\pi(e^{itH})$ acting on $\mathcal H$, i.e., the unique self-adjoint operator on $\mathcal K$ satisfying
\begin{gather}
e^{it\pi(H)}=\pi(e^{itH})~~~(t\in\mathbb R).
\end{gather}
We then define
\begin{align}
\pi(A)=\pi(u)\pi(H).
\end{align}
We can also define $\pi(A)$ using the right polar decomposition of $A$. It is easy to show that these two definitions are the same.

\subsection{A criterion for strong commutativity}

A famous example of Nelson (cf. \cite{Nel}) shows that two unbounded self-adjoint operators  commuting on a common invariant core might not commute strongly. In this section, we give a criterion on the strong commutativity of unbounded closed operators. Our approach follows \cite{Toledano integrating} and \cite{Toledano}. See also
\cite{GJ} section 19.4 for related materials.

 Suppose that $D$ is a self-adjoint positive operator on a Hilbert space $\mathcal H$. For any $r\in\mathbb R$, we let $\mathcal H^r$ be the domain of $(1+D)^r$. It is clear that $\mathcal H^{r_1}\supset\mathcal H^{r_2}$ if $r_1<r_2$. We let $\mathcal H^\infty=\bigcap_{r\geq0}\mathcal H^r$. Define a norm $\lVert\cdot\lVert_r$ on $\mathcal H_r$ to be $\lVert\xi\lVert_r=\lVert(1+D)^r\xi\lVert$. Suppose that $K$ is an unbounded operator on $\mathcal H$ with invariant domain $\mathcal H^\infty$ (``invariant'' means that $K\mathcal H^\infty\subset\mathcal H^\infty$), that $K$ is symmetric, i.e., for any $\xi,\eta\in\mathcal H^\infty$ we have 
 \begin{equation}
 \langle K\xi|\eta\rangle=\langle\xi|K\eta\rangle,
 \end{equation}
 and that for any $n\in\mathbb Z_{\geq0}$ there exist positive numbers $|K|_{n+1}$ and $|K|_{D,n+1}$, such that for any $\xi\in\mathcal H^\infty$
 we have
 \begin{gather}
 \lVert K\xi\lVert_n\leq |K|_{n+1}\lVert\xi\lVert_{n+1},\label{eqa6}\\
 \lVert [D,K]\xi\lVert_n\leq |K|_{D,n+1}\lVert\xi\lVert_{n+1}.
 \end{gather}
Since $K$ is symmetric, it is obviously preclosed. We let $\overline K$ denote the closure $K$. The following lemma is due to Toledano-Laredo (cf. \cite{Toledano integrating} proposition 2.1\footnote{Toledano-Laredo's proof of this proposition was based on a trick in \cite{FL74} theorem 2.} and corollary 2.2).
 \begin{lm}\label{lb52}
 	$\overline K$ is self-adjoint. Moreover, the following statements are true:\\
 	(1) For any $n\in\mathbb Z_{\geq0}$ and $t\in\mathbb R$, the unitary operator $e^{it\overline K}$ restricts to a bounded linear map $\mathcal H^n\rightarrow \mathcal H^n$ with
 	\begin{equation}
 	\lVert e^{it\overline K}\xi\lVert_n\leq e^{ 2nt|K|_{D,n}}\lVert\xi\lVert_n,~~~\xi\in\mathcal H^n.
 	\end{equation}
 	(2) For any $\xi\in\mathcal H^\infty$, $h\in\mathbb R$ and $k=1,2,\dots$, we have
 	\begin{equation}\label{eq204}
 	e^{i(t+h)\overline K}\xi=e^{it\overline K}\xi+\cdots+\frac{h^k}{k!} K^ke^{it\overline K}\xi+R(h),
 	\end{equation}
 	where all terms are in $\mathcal H^\infty$ and $R(h)=o(h^k)$ in each $\lVert\cdot\lVert_n$ norm, i.e., $\lVert R(h)\lVert_nh^{-k}\rightarrow 0$ as $h\rightarrow 0$.
 \end{lm}
 
 This lemma may help us prove the following important criterion for strong commutativity of unbounded closed operators.
 \begin{thm} \label{lb60}
 Let $T$ be another unbounded operator on $\mathcal H$ with invariant domain $\mathcal H^\infty$. Suppose that $T$ satisfies the following conditions:\\
 	(1) There exists  $m\in\mathbb Z_{\geq0}$, such that for any $n\in\mathbb Z_{\geq0}$, we can find a positive number $|T|_{n+m}$, such that
 	\begin{equation}\label{eq208}
 	\lVert T\xi\lVert_n\leq|T|_{n+m}\lVert\xi\lVert_{n+m}\quad(\xi\in\mathcal H^\infty).
 	\end{equation}
 	(2) $T$ is a preclosed operator on $\mathcal H$.\\
 	(3) $KT\xi=TK\xi$ for any $\xi\in\mathcal H^\infty$.\\
 	Then  the self-adjoint operator$\overline K$ commutes strongly with $\overline T$, the closure of $T$.
 \end{thm}
 
 \begin{proof}
 	By lemma \ref{lb52}, for each $t\in\mathbb R$, $e^{it\overline K}$ leaves $\mathcal H^\infty$ invariant. We want to show that
 	\begin{equation}\label{eq210}
 	e^{it\overline K}Te^{-it\overline K}=T ~~~\text{on } \mathcal H^\infty.
 	\end{equation}
 	For any $\xi\in\mathcal H^\infty$ we define a $\mathcal H^\infty$-valued function $\Xi$ on $\mathbb R$ by
 	\begin{equation}
 	\Xi(t)=e^{it\overline K}Te^{-it\overline K}\xi.
 	\end{equation}
 	If we can show that this function is  constant, then we  have $\Xi(t)=\Xi(0)$, which proves \eqref{eq210}. To prove this, it suffices to show that the derivative of this function is always $0$.
 	
For any $t\in\mathbb R$, if $0\neq h\in\mathbb R$, then
 	\begin{align}
 	\Xi(t+h)=&e^{i(t+h)\overline K}Te^{-i(t+h)\overline K}\xi\\
 	=&e^{i(t+h)\overline K}T\big((1-ihK)e^{-it\overline K}\xi+o(h)\big)\label{eq205}\\
 	=&e^{i(t+h)\overline K}T(1-ihK)e^{-it\overline K}\xi+o(h)\label{eq209}\\
 	=&e^{i(t+h)\overline K}Te^{-it\overline K}\xi-ihe^{i(t+h)\overline K}KTe^{-it\overline K}\xi+o(h)\label{eq207}\\
 	=&[e^{it\overline K}(1+ihK)Te^{-it\overline K}\xi+o(h)]\nonumber\\&-ih[e^{it\overline K}(1+ihK)KTe^{-it\overline K}\xi+o(h)]+o(h)\label{eq206}\\
 	=&e^{it\overline K}Te^{-it\overline K}\xi+o(h)=\Xi(t)+o(h),
 	\end{align}
where \eqref{eq205} and \eqref{eq206} follow from \eqref{eq204}, and  \eqref{eq207} follows from the relation $KT=TK$ on $\mathcal H^\infty$. We also used the fact that $To(h)=o(h)$ (which follows from \eqref{eq208}) in \eqref{eq209}. Here the meaning of $o(h)$ is same as that in lemma \ref{lb52}.
 	
Hence we have shown that $\Xi'(t)=0$ for any $t\in\mathbb R$, which proves \eqref{eq210}. Now we regard $T$ as an unbounded operator on $\mathcal H$. By passing to the closure, we have $e^{it\overline K}\overline Te^{-it\overline K}=\overline T$. This shows that $\overline T$ commutes strongly with $\overline K$.
 \end{proof}


\begin{thebibliography}{99}
\small	


\bibitem[BK01]{BK01}
Bakalov, B. and Kirillov, A.A., 2001. Lectures on tensor categories and modular functors (Vol. 21). American Mathematical Soc..

\bibitem[BS90]{BSM}
Buchholz, D. and Schulz-Mirbach, H., 1990. Haag duality in conformal quantum field theory. Reviews in Mathematical Physics, 2(01), pp.105-125.



\bibitem[CKLW15]{CKLW}
Carpi S, Kawahigashi Y, Longo R, Weiner M. From vertex operator algebras to conformal nets and back. arXiv preprint arXiv:1503.01260. 2015 Mar 4.

\bibitem[CKM17]{CKM17}
Creutzig, T., Kanade, S. and McRae, R., 2017. Tensor categories for vertex operator superalgebra extensions. arXiv preprint arXiv:1705.05017.






\bibitem[Con80]{Connes fusion}
Connes, A., 1980. On the spatial theory of von Neumann algebras. Journal of Functional Analysis, 35(2), pp.153-164.

\bibitem[DHR71]{DHR71}
Doplicher, S., Haag, R. and Roberts, J.E., 1971. Local observables and particle statistics I. Communications in Mathematical Physics, 23(3), pp.199-230.

\bibitem[DL14]{DL}
Dong, C. and Lin, X., 2014. Unitary vertex operator algebras. Journal of algebra, 397, pp.252-277.

\bibitem[EGNO04]{Etingof}
Etingof, P.I., Gelaki, S., Nikshych, D. and Ostrik, V., 2015. Tensor categories (Vol. 205). Providence, RI: American Mathematical Society.




\bibitem[FB04]{Conformal blocks}
Frenkel, E. and Ben-Zvi, D., 2004. Vertex algebras and algebraic curves (No. 88), 2nd edition. American Mathematical Soc..




\bibitem[FHL93]{FHL}
Frenkel, I., Huang, Y.Z. and Lepowsky, J., 1993. On axiomatic approaches to vertex operator algebras and modules (Vol. 494). American Mathematical Soc..

\bibitem[FL74]{FL74}
Faris, W.G. and Lavine, R.B., 1974. Commutators and self-adjointness of Hamiltonian operators. Communications in Mathematical Physics, 35(1), pp.39-48.

\bibitem[FLM89]{FLM}
Frenkel, I., Lepowsky, J. and Meurman, A., 1989. Vertex operator algebras and the Monster (Vol. 134). Academic press.

\bibitem[FRS89]{FRS89}
Fredenhagen, K., Rehren, K.H. and Schroer, B., 1989. Superselection sectors with braid group statistics and exchange algebras. Communications in Mathematical Physics, 125(2), pp.201-226.



\bibitem[Gal12]{Gal12}
Galindo, C., 2012. On braided and ribbon unitary fusion categories. arXiv preprint arXiv:1209.2022.

\bibitem[GJ12]{GJ}
Glimm, J. and Jaffe, A., 2012. Quantum physics: a functional integral point of view. Springer Science \& Business Media.


\bibitem[GW84]{GW}
Goodman, R. and Wallach, N.R., 1984. Structure and unitary cocycle representations of loop groups and the group of diffeomorphisms of the circle. energy, 3, p.3.


\bibitem[HK07]{HK07}
Huang, Y.Z. and Kong, L., 2007. Full field algebras. Communications in mathematical physics, 272(2), pp.345-396.

\bibitem[HK10]{HK10}
Huang, Y.Z. and Kong, L., 2010. Modular invariance for conformal full field algebras. Transactions of the American Mathematical Society, 362(6), pp.3027-3067.

\bibitem[HKL15]{HKL15}
Huang, Y.Z., Kirillov, A. and Lepowsky, J., 2015. Braided tensor categories and extensions of vertex operator algebras. Communications in Mathematical Physics, 337(3), pp.1143-1159.

\bibitem[HL94]{HL94}
Huang, Y.Z. and Lepowsky, J., 1994. Tensor products of modules for a vertex operator algebra and vertex tensor categories. Lie Theory and Geometry, in honor of Bertram Kostant, pp.349-383.



\bibitem[HL95a]{H 1}
Huang, Y.Z. and Lepowsky, J., 1995. A theory of tensor products for module categories for a vertex operator algebra, I. Selecta Mathematica, 1(4), p.699.

\bibitem[HL95b]{H 2}
Huang, Y.Z. and Lepowsky, J., 1995. A theory of tensor products for module categories for a vertex operator algebra, II. Selecta Mathematica, 1(4), p.757.

\bibitem[HL95c]{H 3}
Huang, Y.Z. and Lepowsky, J., 1995. A theory of tensor products for module categories for a vertex operator algebra, III. Journal of Pure and Applied Algebra, 100(1-3), pp.141-171.

\bibitem[HL13]{HL13}
Huang, Y.Z. and Lepowsky, J., 2013. Tensor categories and the mathematics of rational and logarithmic conformal field theory. Journal of Physics A: Mathematical and Theoretical, 46(49), p.494009.



\bibitem[HLZ11] {HLZ}
Huang, Y.Z., Lepowsky, J. and Zhang, L., 2011. Logarithmic tensor category theory, VIII: Braided tensor category structure on categories of generalized modules for a conformal vertex algebra. arXiv preprint arXiv:1110.1931. 

\bibitem[Hua95] {H 4}
Huang, Y.Z., 1995. A theory of tensor products for module categories for a vertex operator algebra, IV. Journal of Pure and Applied Algebra, 100(1-3), pp.173-216.



\bibitem[Hua05a] {H ODE}
Huang, Y.Z., 2005. Differential equations and intertwining operators. Communications in Contemporary Mathematics, 7(03), pp.375-400.

\bibitem[Hua05b] {H MI}
Huang, Y.Z., 2005. Differential equations, duality and modular invariance. Communications in Contemporary Mathematics, 7(05), pp.649-706.

\bibitem[Hua08a]{H Verlinde}
Huang, Y.Z., 2008. Vertex operator algebras and the Verlinde conjecture. Communications in Contemporary Mathematics, 10(01), pp.103-154.


\bibitem[Hua08b]{H Rigidity}
Huang, Y.Z., 2008. Rigidity and modularity of vertex tensor categories. Communications in contemporary mathematics, 10(supp01), pp.871-911.

\bibitem[KLM01]{KLM01}
Kawahigashi, Y., Longo, R. and M\"uger, M., 2001. Multi-Interval Subfactors and Modularity of Representations in Conformal Field Theory. Communications in Mathematical Physics, 219(3), pp.631-669.

\bibitem[Kac98]{VOA beginners}
Kac, V.G., 1998. Vertex algebras for beginners (No. 10). American Mathematical Soc..

\bibitem[Kaw15]{Kaw15}
Kawahigashi, Y., 2015. Conformal field theory, tensor categories and operator algebras. Journal of Physics A: Mathematical and Theoretical, 48(30), p.303001.

\bibitem[Kong06]{Kong06}
Kong, L., 2006. Full field algebras, operads and tensor categories. arXiv preprint math/0603065.

\bibitem[Kong08]{Kong08}
Kong, L., 2008. Cardy condition for open-closed field algebras. Communications in Mathematical Physics, 283(1), pp.25-92.



\bibitem[LL12]{LL}
Lepowsky, J. and Li, H., 2012. Introduction to vertex operator algebras and their representations (Vol. 227). Springer Science \& Business Media.

\bibitem[MS88]{MS88}
Moore, G. and Seiberg, N., 1988. Polynomial equations for rational conformal field theories. Physics Letters B, 212(4), pp.451-460.

\bibitem[MS89]{MS89}
Moore, G. and Seiberg, N., 1989. Classical and quantum conformal field theory. Communications in Mathematical Physics, 123(2), pp.177-254.

\bibitem[MS90]{MS90}
Moore, G. and Seiberg, N., 1990. Lectures on RCFT. In Physics, geometry and topology (pp. 263-361). Springer, Boston, MA.


\bibitem[Muk10]{Muk}
Mukhopadhyay, S., 2010. Decomposition of conformal blocks (Doctoral dissertation, Master’s thesis, University of North Carolina at Chapel Hill).


\bibitem[Nel59]{Nel}
Nelson, E., 1959. Analytic vectors. Annals of Mathematics, pp.572-615.

\bibitem[Neu16]{Neu16}
Von Neumann, J., 2016. Functional Operators (AM-22), Volume 2: The Geometry of Orthogonal Spaces.(AM-22) (Vol. 2). Princeton University Press.

\bibitem[OS73]{OS73}
Osterwalder, K. and Schrader, R., 1973. Axioms for Euclidean Green's functions. Communications in mathematical physics, 31(2), pp.83-112.

\bibitem[Seg88]{Seg88}
Segal, G.B., 1988. The definition of conformal field theory. In Differential geometrical methods in theoretical physics (pp. 165-171). Springer, Dordrecht.



\bibitem[TL99]{Toledano integrating}
Toledano-Laredo, V., 1999. Integrating unitary representations of infinite-dimensional Lie groups. Journal of functional analysis, 161(2), pp.478-508.


\bibitem[TL04]{Toledano}
Toledano-Laredo, V., 2004. Fusion of positive energy representations of Lspin (2n). arXiv preprint math/0409044.

\bibitem[Tur16]{Tur16}
Turaev, V.G., 2016. Quantum invariants of knots and 3-manifolds (Vol. 18). Walter de Gruyter GmbH \& Co KG.


\bibitem[Ueno08]{Ueno}
Ueno, K., 2008. Conformal field theory with gauge symmetry (Vol. 24). American Mathematical Soc..




\bibitem[Was98]{Wassermann}

Wassermann, A., 1998. Operator algebras and conformal field theory III. Fusion of positive energy representations of LSU (N) using bounded operators. Inventiones Mathematicae, 133(3), pp.467-538.




\end{thebibliography}
\end{document}